\documentclass[oneside,12pt]{amsart}
\usepackage[T2A]{fontenc}
\usepackage[cp1251]{inputenc}
\usepackage{amssymb}
\usepackage{amsxtra}
\usepackage{amsmath}
\usepackage{amsthm}

\usepackage{hyperref}
\usepackage[all]{xy}

\usepackage[russian,english]{babel}

\DeclareMathOperator{\GL}{GL}

\DeclareMathOperator{\Lie}{Lie}
\DeclareMathOperator{\End}{End}
\DeclareMathOperator{\Hom}{Hom}
\DeclareMathOperator{\Cent}{Cent}
\DeclareMathOperator{\Spec}{Spec}
\DeclareMathOperator{\Ga}{{\mathbf G}_a}

\DeclareMathOperator{\im}{im\,}
\DeclareMathOperator{\Aut}{Aut\,}

\DeclareMathOperator{\Dyn}{Dyn}
\DeclareMathOperator{\Sym}{Sym}
\DeclareMathOperator{\Gal}{Gal}

\DeclareMathOperator{\rank}{rank}

\DeclareMathOperator{\Gm}{{\mathbf G}_m}

\DeclareMathOperator{\SL}{SL}
\DeclareMathOperator{\Sp}{Sp}
\DeclareMathOperator{\SU}{SU}
\DeclareMathOperator{\colim}{\lim_{\longleftarrow}}

\newcommand{\ad}{\mathrm{ad}}

\newcommand{\height}{ht}
\newcommand{\id}{\text{\rm id}}
\newcommand{\pst}{s}

\DeclareMathOperator{\ZZ}{{\mathbb Z}}
\DeclareMathOperator{\QQ}{{\mathbb Q}}
\DeclareMathOperator{\NN}{{\mathbb N}}
\DeclareMathOperator{\FF}{{\mathbb F}}

\newtheorem{lem}{Lemma}[section]
\newtheorem*{thm*}{Theorem}
\newtheorem{thm}[lem]{Theorem}

\newtheorem{cor}[lem]{Corollary}
\newtheorem*{cor*}{Corollary}
\theoremstyle{definition}{  \newtheorem{rem}[lem]{Remark}  }
\theoremstyle{definition}{   }
\theoremstyle{definition}{  \newtheorem{dfn}[lem]{Definition} }

\let\l\left
\let\r\right

\newcommand{\st}{\scriptstyle}
\newcommand{\ds}{\displaystyle}

\newcommand{\CE}{C}

\begin{document}

\title{On the congruence kernel of isotropic groups\\ over rings}
\author{A. Stavrova}
\thanks{The author is a winner of the contest ``Young Russian Mathematics'', and
was supported at different stages of her work by the postdoctoral
grant 6.50.22.2014 ``Structure theory, representation theory and geometry of algebraic groups''
at St. Petersburg State University,
by the J.E. Marsden postdoctoral fellowship of the Fields Institute,
by the RFBR grants 14-01-31515-mol\_a, 13-01-00709, 12-01-33057, 12-01-31100, and the research program 6.38.74.2011 ``Structure theory and geometry
of algebraic groups and their applications in representation theory and algebraic K-theory'' at St.
Petersburg State University.}
\address{Chebyshev Laboratory\\Department of Mathematics and Mechanics\\
St. Petersburg State University\\Russia}
\email{anastasia.stavrova@gmail.com}

\maketitle
\begin{abstract}
Let $R$ be a connected noetherian commutative ring, and let $G$ be a simply connected
reductive group over $R$ of isotropic rank $\ge 2$.
The elementary
subgroup $E(R)$ of $G(R)$ is the subgroup generated
by $U_{P^+}(R)$ and $U_{P^-}(R)$, where $U_{P^\pm}$ are the unipotent radicals of two opposite parabolic subgroups $P^\pm$
of $G$.
Assume that $2\in R^\times$ if $G$ is of type $B_n,C_n,F_4,G_2$ and $3\in R^\times$ if $G$ is of type $G_2$.
We prove that the
congruence kernel of $E(R)$, defined as the kernel of the natural homomorphism $\widehat{E(R)}\to\overline{E(R)}$
between the profinite completion of $E(R)$ and the congruence completion of $E(R)$ with respect to congruence subgroups
of finite index, is central in $\widehat{E(R)}$. In the course of the proof,
we construct Steinberg groups associated to isotropic reducive groups  and show that they are central extensions
of $E(R)$ if $R$ is a local ring.
\end{abstract}

\section{Introduction}

Let $k$ be a global field,
let $S$ be a set of primes of $k$ containing all archimedean primes, and let $O_S$ be the ring of $S$-integers of $k$. Let $G$
be an algebraic group over $k$ provided with a closed $k$-embedding $G\to\GL_{n,k}$. Set $\Gamma=G(k)\cap\GL_n(O_S)$,
and let $\Gamma_I=\Gamma\cap\GL_n(O_S,I)$, where $\GL_n(O_S,I)$ is the subgroup of matrices congruent to $1$ modulo $I$.
The classical congruence subgroup problem,
as stated by H. Bass, J. Milnor, and J.-P. Serre~\cite[p. 128]{BMS},
 asks whether every subgroup of finite index in $\Gamma$ contains an $S$-congruence subgroup
$\Gamma_I$ for some $I\neq 0$. The answer is, in general, negative; however, the deviation from the positive answer
can be measured by the congruence kernel of $\Gamma$, defined as follows. The congruence subgroups $\Gamma_I$, and the
normal subgroups of finite index respectively, constitute bases of neighbourhoods of the unit for two topologies on
$\Gamma$, called the congruence and the profinite topology. The congruence kernel $C(\Gamma)$ is the kernel of
the natural continuous surjective homomorphism $\widehat{\Gamma}\to\overline{\Gamma}$ between the two completions of $\Gamma$
with respect to these topologies, with
$\widehat{\Gamma}$ being the profinite completion, and $\overline{\Gamma}$ being the congruence completion.

The congruence kernel $C(\Gamma)$ had been completely computed (or proved to be infinite) in many cases, notably,  for $G=\SL_n$ ($n\ge 3$) and $G=\Sp_{2n}$
($n\ge 2$) in~\cite{BLS,BMS}, $G=\SL_2$ in~\cite{Serre-cong}, for all $k$-split simply connected simple groups $G$
in~\cite{Ma}, and for all $k$-isotropic simply connected simple groups $G$ in~\cite{Ragh76,Ragh86,PrRagh-met,PrRa-met}.
For several classes of $k$-anisotropic groups the problem is still open;
we refer to~\cite{PrRa-2010,PrRa-2016} for a detailed exposition of available results.

Given the abundant progress on the congruence subgroup problem in the above formulation, it is natural to ask
whether the same problem can be approached for $\Gamma=G(R)$, where $R$ is a general commutative ring instead of
a ring of algebraic $S$-integers. It turns out that at least the centrality of the
congruence kernel $C(\Gamma)$ in $\widehat{\Gamma}$, which is key for
its computation in the classical setting, can be established under much weaker assumptions. First, M. Kassabov
and N. Nikolov~\cite{KN}, while proving
that the group $\Gamma=\SL_n(\mathbb{Z}[x_1,\ldots,x_k])$ ($n\ge 3$, $k\ge 1$) has property $\tau$, established the centrality
of the congruence kernel $C(\Gamma)$. Then, A. Rapinchuk and I. Rapinchuk~\cite{RR}
proved the centrality of the congruence kernel $C(E(R))$, where $R$ is any
noetherian commutative ring and $\Gamma=E(R)$ is the elementary
subgroup in the group of points $G(R)$ of any simply connected Chevalley--Demazure group scheme $G$ of rank $\ge 2$,
with the only restriction that $2\in R^\times$ if $G$ is of type $C_n$ or $G_2$. Note that if $R=O_S$, then
$E(R)=G(R)$ for such $G$ by~\cite{Ma}, and the centrality of the congruence kernel is equivalent to its
finiteness~\cite[Theorem 2]{PrRa-2010}. Our aim is to extend the result of~\cite{RR} to isotropic simply
connected reductive $R$-group schemes in the sense of~\cite{SGA3}.

Let $R$ be a commutative ring with 1, and let $G$ be a reductive group (scheme) over $R$ in the
sense of~\cite{SGA3}. We say that $G$ has isotropic rank $\ge n$ over $R$, if every
semisimple normal $R$-subgroup of $G$ contains a $n$-dimensional split $R$-torus
$(\mathbb{G}_{m,R})^n$. If the isotropic rank is $\ge 1$, then $G$ contains a pair of opposite
parabolic $R$-subgroups $P^\pm$~\cite[Exp. XXVI, Prop. 6.1]{SGA3} which are
strictly proper, i.e. intersect properly every semisimple normal $R$-subgroup of $G$.
Let $E_P(R)$ be the subgroup of $G(R)$
generated by $U_{P^+}(R)$ and $U_{P^-}(R)$, where $U_{P^\pm}$ denotes the unipotent radical  of $P^\pm$.
The main result of~\cite{PS} says that $E_P(R)$ does not depend on a strictly proper parabolic subgroup $P$
if the isotropic rank of $G$ over every localization $R_m$ of $R$ at a maximal ideal $m$ is $\ge 2$.
Under this assumption, we call $E_P(R)$ the {\it elementary
subgroup} of $G(R)$ and denote it simply by $E(R)$. 

For any ideal $I$ of $R$, we denote by $G(R,I)$ the congruence subgroup of $G(R)$ of level $I$, that is, the kernel of
the reduction homomorphism $G(R)\to G(R/I)$. The congruence topology on $G(R)$ (respectively, on $E(R)$) is the topology
whose base of neighbourhoods of the unit consists of all subgroups $G(R,I)$ (respectively, $E(R)\cap G(R,I)$),
where $I$ is of finite index in $R$. The profinite topology on $G(R)$ and $E(R)$ is defined as in the classical setting.
In the same way, we denote by $\widehat{\Gamma}$ and $\overline{\Gamma}$ the profinite and congruence completions of
$\Gamma=E(R)$ or $G(R)$.

If the ring $R$ is connected, then the split reductive group $G_K$ has the same root system type $\Phi$ for every
$R$-algebra $K$ which is an algebraically closed field~\cite[Exp. XXII, Prop. 2.8]{SGA3}.
In this case we call $\Phi$ the absolute root system of $G$.

Our main result is the following theorem.
\begin{thm}\label{thm:main}
Let $G$ be a 
simply connected semisimple reductive group scheme defined over a connected noetherian commutative ring $R$,
such that the absolute root system $\Phi$ of $G$ is irreducible,
and $2\in R^\times$ if $\Phi=B_n,C_n,F_4,G_2$ and $3\in R^\times$ if $\Phi=G_2$.
 Assume that $G$ has isotropic rank $\ge 1$, and for any maximal ideal $m$ of $R$ the group $G_{R_m}$ has isotropic rank $\ge 2$.
Let $E(R)$ be the elementary subgroup of $G(R)$.
Then $\overline{E(R)}=\overline{G(R)}$, and the kernel of the natural homomorphism $p:\widehat{E(R)}\to\overline{G(R)}$ is central in $\widehat{E(R)}$.
\end{thm}

In the particular case where $G$ is
a Chevalley--Demazure, i.e. split, simply connected group of rank $\ge 2$,
the statement of Theorem~\ref{thm:main} specializes to the main result of~\cite{RR}, except that the latter
requires only $2\in R^\times$ for $\Phi=C_n,G_2$.
Although our methods would allow to extend this result in full,
we decided to assume the slighly stronger invertibility condition in order to make the proof for all isotropic
reductive groups more transparent, with a minimal amount of case-by-case considerations.

If $G(R)=E(R)$, then $\ker(p)$ is the full
congruence kernel of $G(R)$.  Thus,
combining the above theorem with the results of~\cite{St-poly}, we obtain the following result for quasi-split groups.

\begin{cor}\label{cor:cong-problem}
Let $G$ be a simply connected quasi-split simple algebraic group of isotropic rank $\ge 2$  over a finite field $\FF_q$,
such that
$2\in \FF_q^\times$ if $\Phi=B_n,C_n,F_4,G_2$, and $3\in \FF_q^\times$ if $\Phi=G_2$.
For any semilocal regular ring $A$ containing $\FF_q$, and any $m,n\ge 0$, the congruence kernel of
$G(A[X_1,\ldots,X_n,Y_1^{\pm 1},\ldots,Y_m^{\pm 1}])$
is central.
\end{cor}

The general idea of the proof of Theorem~\ref{thm:main} is analogous to that of the main theorem of~\cite{RR},
once we replace the standard 1-dimensional root subgroups of Chevalley groups by relative root subschemes of isotropic groups
introduced in~\cite{PS,St-serr} (see~\S~\ref{sec:prel} for their definitions and properties).
The difficulty lies in establishing two key ingredients of the proof which
had been long known for Chevalley--Demazure groups, but not for general isotropic reductive groups.

The first ingredient is the so-called centrality of the non-stable $K_2$-functor associated to $G$ over a local ring.
For Chevalley--Demazure groups over local rings, the corresponding result is due to
M. R. Stein~\cite[Theorem 2.13]{Stein-cent};
for groups of rank $\ge 3$ of simply laced or symplectic type it is even known over general
commutative rings~\cite{vdK-another,Lav,LavSi}.
V. Deodhar defined Steinberg groups for isotropic reductive groups over fields
in terms of a minimal parabolic subgroup of $G$, and established the centrality
of the corresponding non-stable $K_2$-functors~\cite[Prop. 1.16]{Deo}.
We extend the latter result to isotropic groups over local rings. In order to do that,
we are forced to consider Steinberg groups $St_P(R)$ corresponding to an arbitrary
parabolic subgroup $P$ of $G$, since a minimal parabolic subgroup over a local ring
may not stay minimal over the residue field.
The definition of $St_P(R)$ is given in \S~\ref{sec:St} in terms
of the system of relative roots $\Phi_P$ and relative root subschemes
$X_\alpha(V_\alpha)$, $\alpha\in\Phi_P$.
 If $R$ is a field and $P$ is a minimal parabolic subgroup, then $\Phi_P$
is the root system of $G$ in the sense of~\cite{BorelTits}, root subchemes $X_\alpha(V_\alpha)$ are naturally isomorphic to
$U_{(\alpha)}/U_{(2\alpha)}$, and the Steinberg group is the same as in~\cite{Steinberg,Deo}.

\begin{thm}\label{thm:K2}
Let $R$ be a local ring, $G$ be an isotropic simply connected reductive group over $R$, $P$ a
parabolic $R$-subgroup of $G$. Assume that all irreducible components of $\Phi_P$ have rank $\ge 2$. Then
the natural projection $St_P(R)\to E(R)$, where $St_P(R)$ is the
Steinberg group associated to $P$ over $R$, has central kernel.
\end{thm}

The second ingredient, whose proof is much more complicated than in the split group case, is the coincidence of
the profinite and congruence topologies on a root subscheme of $G$.
The main difference is that while root subschemes of a Chevalley--Demazure group are 1-dimensional subgroups whose groups of points
are isomorphic to
the base ring $R$, the points of root subschemes $X_\alpha(V_\alpha)$ in isotropic groups are parametrized by finitely generated projective
$R$-modules  $V_\alpha$ of varied dimension, and they are not subgroups of $G(R)$. Another source
of difficulties is that we limit ourselves with the isotropic rank $\ge 1$ assumption over the base ring, compared to the
rank $\ge 2$ assumption in the Chevalley--Demazure group case.

\begin{thm}\label{thm:E-normal}
Let $R$ be connected commutative ring, let $G$ be a reductive group over $R$ such that the absolute
root system $\Phi$ of $G$ is irreducible,
and $2\in R^\times$ if $\Phi=B_n,C_n,F_4,G_2$ and $3\in R^\times$ if $\Phi=G_2$.
Assume that for any maximal ideal $m$ of $R$ the group $G_{R_m}$ has isotropic rank $\ge 2$.
Let $P$ be a proper parabolic subgroup of $G$, and let $E(R)=E_P(R)$ be the elementary subgroup of $G(R)$.
Then for any normal subgroup $N\le E(R)$ there exists an ideal
$I$ in $R$ such that $N\cap X_\alpha(V_\alpha)=X_\alpha(IV_\alpha)$
for any $\alpha\in\Phi_P$.
\end{thm}

The corresponding result for a Chevalley--Demazure group $G$ is due to L. Vaserstein~\cite[Theorem 4]{Vas86};
see also~\cite[Theorem 3]{Abe89} or~\cite[Proposition 3.1]{RR} for alternative proofs.
More generally,~\cite[Theorem 4]{Vas86} establishes that for any $E(R)$-normalized subgroup $N$
of $G(R)$ there is an ideal $I$ of $R$ such that $N$ lies between the elementary congruence subgroup $E(R,I)$
and the inverse image $C(R,I)$ of the center of $G(R/I)$ in $G(R)$.
In our joint work with A. Stepanov~\cite{StSt}
we generalize this description of $E(R)$-normalized subgroups to isotropic reductive groups using Theorem~\ref{thm:E-normal}.

\section{Preliminaries}\label{sec:prel}

\subsection{Split tori, parabolic subgroups, and elementary subgroups}

Let $R$ be a commutative ring with 1.
Let $S=(\Gm_{,R})^N=\Spec(R[x_1^{\pm 1},\ldots,x_N^{\pm 1}])$
be a split $N$-dimensional torus over $R$. The character group
$X^*(S)=\Hom_R(S,\Gm_{,R})$ of $S$ is canonically isomorphic to $\ZZ^N$.
If $S$ acts $R$-linearly on an $R$-module $V$, this module has a natural $\ZZ^N$-grading
$$
V=\bigoplus_{\lambda\in X^*(S)}V_\lambda,
$$
where
$$
V_\lambda=\{v\in V\ |\ s\cdot v=\lambda(s)v\ \mbox{for any}\ s\in S(R)\}.
$$
Conversely, any $\ZZ^N$-graded $R$-module $V$ can be provided with an $S$-action by the same rule.

Let $G$ be a reductive group scheme over $R$ in the sense of~\cite{SGA3}. Assume that $S$ acts on $G$
by $R$-group automorphisms. The associated Lie algebra functor $\Lie(G)$ then acquires
a $\ZZ^N$-grading compatible with the Lie algebra structure,
$$
\Lie(G)=\bigoplus_{\lambda\in X^*(S)}\Lie(G)_\lambda.
$$

We will use the following version of~\cite[Exp. XXVI Prop. 6.1]{SGA3}.

\begin{lem}\label{lem:T-P}
Let $L=\Cent_G(S)$ be the subscheme of $G$ fixed by $S$. Let
$\Psi\subseteq X^*(S)$ be an $R$-subsheaf of sets closed under addition of characters.

(i) If $0\in\Psi$, then there exists
a unique smooth connected closed subgroup $U_\Psi$ of $G$ containing $L$ and satisfying
\begin{equation}\label{eq:LieUPsi}
\Lie(U_\Psi)=\bigoplus_{\lambda\in\Psi}\Lie(G)_\lambda.
\end{equation}
Moreover, if $\Psi=\{0\}$, then $U_\Psi=L$; if $\Psi=-\Psi$, then $U_\Psi$ is reductive; if $\Psi\cup(-\Psi)=X^*(S)$,
then $U_\Psi$ and $U_{-\Psi}$ are two opposite parabolic subgroups of $G$
in the sense of~\cite[Exp. XXVI D\'ef. 1.1]{SGA3}
with the common Levi subgroup
$U_{\Psi\cap(-\Psi)}$.

(ii) If $0\not\in\Psi$, then there exists a unique smooth connected unipotent closed subgroup $U_\Psi$ of $G$
normalized by $L$ and satisfying~\eqref{eq:LieUPsi}.
\end{lem}
\begin{proof}
The statement immediately follows by faithfully flat descent from the standard facts about the subgroups of
split reductive groups proved in~\cite[Exp. XXII]{SGA3}; see the proof of~\cite[Exp. XXVI Prop. 6.1]{SGA3}.
\end{proof}

\begin{dfn}\label{def:tor-roots}
The sheaf of sets
$$
\Phi=\Phi(S,G)=\{\lambda\in X^*(S)\setminus\{0\}\ |\ \Lie(G)_\lambda\neq 0\}
$$
is called the \emph{system of relative roots of $G$ with respect to $S$}.
\end{dfn}

Choosing a total order on the $\QQ$-space $\QQ\otimes_{\ZZ} X^*(S)\cong\QQ^n$, one defines the subsets
of positive and negative relative roots $\Phi^+$ and $\Phi^-$, so that $\Phi$ is a disjoint
union of $\Phi^+$, $\Phi^-$, and $\{0\}$. By Lemma~\ref{lem:T-P} the closed subgroups
$$
U_{\Phi^+\cup\{0\}}=P,\qquad U_{\Phi^-\cup\{0\}}=P^-
$$
are two opposite parabolic subgroups of $G$ with the common Levi subgroup $\Cent_G(S)$.
Thus, if a reductive group $G$ over $R$ admits a non-trivial action of a split torus,
then it has a proper parabolic subgroup. If $G$ has isotropic rank $\ge 1$ over $R$, then Lemma~\ref{lem:T-P} implies that $G$ contains a
strictly proper parabolic subgroup $P$.
The converse is true Zariski-locally, see~Lemma~\ref{lem:relroots} below.

Let $P$ be any parabolic subgroup of $G$. Since the base $\Spec R$ is affine, the group $P$ has a Levi subgroup $L_P$~\cite[Exp.~XXVI Cor.~2.3]{SGA3}.
There is a unique parabolic subgroup $P^-$ in $G$ which is opposite to $P$ with respect to $L_P$,
that is $P^-\cap P=L_P$, cf.~\cite[Exp. XXVI Th. 4.3.2]{SGA3}. We denote by $U_P$ and $U_{P^-}$ the unipotent
radicals of $P$ and $P^-$ respectively.

\begin{dfn}
\label{defn:E_P}
The \emph{elementary subgroup $E_P(R)$ corresponding to $P$} is the subgroup of $G(R)$
generated as an abstract group by $U_P(R)$ and $U_{P^-}(R)$.
\end{dfn}

Note that if $L'_P$ is another Levi subgroup of $P$,
then $L'_P$ and $L_P$ are conjugate by an element $u\in U_P(R)$~\cite[Exp. XXVI Cor. 1.8]{SGA3}, hence
$E_P(R)$ does not depend on the choice of a Levi subgroup or an opposite subgroup
$P^-$, so we do not include this datum in the notation.

\begin{thm}\label{th:PS-normality}\cite[Lemma 12, Theorem 1; SGA3]{PS}
Let $G$ be a reductive group over a commutative ring $R$, and let $A$ be a commutative $R$-algebra.

(i) If $R$ is a semilocal ring, then the subgroup $E_P(A)\le G(A)$ is the same for any
minimal parabolic $R$-subgroup $P$ of $G$. If, moreover, $G$ contains a strictly proper parabolic $R$-subgroup,
the subgroup $E_P(A)$ is the same for any strictly proper parabolic $R$-subgroup $P$.

(ii) If for any maximal ideal $m$ of $R$ the group $G_{R_m}$ has isotropic rank $\ge 2$,
then the subgroup $E_P(A)$ of $G(A)$ is the same for any
strictly proper parabolic $A$-subgroup $P$ of $G_A$.

\end{thm}
\begin{proof}
See~\cite[Theorem 2.1]{St-poly}.
\end{proof}

\begin{dfn}
Under the assumptions of Theorem~\ref{th:PS-normality},
we call the subgroup $E(A)=E_P(A)$, where $P$ is a strictly proper parabolic subgroup of $G$,
the \emph{elementary subgroup}  of $G(A)$.
\end{dfn}

\subsection{Abstract relative roots}
In order to better understand systems of relative roots $\Phi(S,G)$ of Definition~\ref{def:tor-roots},
we consider the abstract systems of relative roots derived from abstract root systems in the sense of~\cite{Bu}.
The corresponding notion was first introduced in~\cite{PS}.

Let $\Phi$ be a reduced root system in the sense of~\cite{Bu} with an inner product $(-,-)$.
Let $\Pi=\{a_1,\ldots,a_l\}$ be a fixed system of simple roots of $\Phi$; if $\Phi$ is irreducible,
we assume that the numbering follows Bourbaki~\cite{Bu}. Let $D$ be the Dynkin diagram of $\Phi$.
We identify nodes of $D$ with the corresponding simple roots in $\Pi$.

\begin{dfn}
Let $\Gamma\le\Aut(D)$ be a subgroup, and let
$J\subseteq\Pi$ be a $\Gamma$-invariant subset. Consider the projection
$$
\pi_{J,\Gamma}\colon\ZZ \Phi\longrightarrow \ZZ\Phi/\l<\Pi\setminus J;\ a-\sigma(a),\ a\in J,\ \sigma\in\Gamma\r>.
$$
The set $\Phi_{J,\Gamma}=\pi_{J,\Gamma}(\Phi)\setminus\{0\}$ is called the system of {\it relative roots} corresponding to
the pair $(J,\Gamma)$.
The {\it rank} of $\Phi_{J,\Gamma}$ is the rank of $\pi_{J,\Gamma}(\ZZ \Phi)$ as a free abelian group.
\end{dfn}

It is clear that any relative root $\alpha\in\Phi_{J,\Gamma}$ can be represented as a unique
linear combination of relative roots from $\pi_{J,\Gamma}(\Pi)\setminus\{0\}$. We call the elements of
$\pi_{J,\Gamma}(\Pi)\setminus\{0\}$ the \emph{simple relative roots}.
We say that $\alpha\in\Phi_{J,\Gamma}$ is a {\it positive} (resp. {\it negative}) relative root, if it
is a non-negative (respectively, a non-positive) linear combination of simple relative roots.
The sets of positive and negative relative roots will be denoted by
$\Phi_{J,\Gamma}^+$ and $\Phi_{J,\Gamma}^-$ respectively.
The \emph{height} $\height(\alpha)$ of a relative root $\alpha$ is the sum of coefficients in its decomposition
into a linear combination of simple relative roots (the same thing was called the {\it level} in~\cite{PS,LS}).

Observe that the action of $\Gamma$ on $\Pi$ naturally extends to an action on $\Phi$ respecting the
irreducible components of the root system. If the action on irreducible components is transitive, the system of relative roots
$\Phi_{J,\Gamma}$ is {\it irreducible}.
Clearly, any system of relative roots $\Phi_{J,\Gamma}$
is a disjoint union of irreducible ones; we call them the {\it irreducible components} of $\Phi_{J,\Gamma}$.

Let $S\subseteq \Phi$ be any subset.
We say that a root $a\in S$ is {\it $\Pi$-maximal}, if there is no simple root $b\in\Pi$
such that $a+b\in S$. Note that a root of maximal height in $S$ is automatically $\Pi$-maximal, but not vice versa.
In particular, $\Pi$-maximal roots of $\Phi^+$ are precisely the roots of maximal height (with respect to $\Pi$)
in irreducible components of $\Phi$. $\Pi$-minimal roots, and maximal and minimal relative roots in
$S\subseteq\Phi_{J,\Gamma}$ are defined similarly.
Each irreducible component of $\Phi_{J,\Gamma}$ contains a unique relative root $\tilde a$ of maximal height,
which is the image under $\pi_{J,\Gamma}$ of a maximal root in $\Phi$.

For any root $a\in\Phi$ we write
$$
a=\sum\limits_{b\in\Pi}m_b(a)b,\quad m_b(a)\in\ZZ,\quad\mbox{and}\quad a_J=\sum\limits_{b\in J}m_b(a)b.
$$
We will later need the following lemmas.

\begin{lem}\label{lem:max-min-roots}
Let $\Phi^\circ$ be an irreducible component of $\Phi$, and $\alpha\in\Phi_{J,\{\id\}}$ be a relative root.

(i) For any $\alpha\in\Phi_{J,\{\id\}}$ the set $\Phi^\circ\cap\pi_{J,\{\id\}}^{-1}(\alpha)$ contains a unique $\Pi$-maximal
and a unique $\Pi$-minimal root, which are automatically the unique roots of maximal and minimal height respectively.

(ii) If $i\alpha,j\alpha\in\Phi_{J,\{\id\}}$ for some $i,j>0$, then the difference of
maximal (respectively, minimal) roots in $\pi_{J,\{\id\}}^{-1}(i\alpha)\cap\Phi^\circ$ and
$\pi_{J,\{\id\}}^{-1}(j\alpha)\cap\Phi^\circ$ is also a root.

(iii) For any $\alpha\in\Phi_{J,\Gamma}$, the group $\Gamma$ acts transitively
on the set of all $\Pi$-maximal (respectively, $\Pi$-minimal) roots in $\pi_{J,\Gamma}^{-1}(\alpha)$.

(iv) There is an integer $m_\alpha\ge 1$ such that
$\ZZ\alpha\cap\Phi_{J,\Gamma}=\{\pm \alpha,\pm 2\alpha,\ldots,\pm m_\alpha\alpha\}$.
\end{lem}
\begin{proof}
The proof of (i) and (ii) for the case of maximal roots is literally the same as in~\cite[Lemma 1]{KazS}, which
deals with the case where $\alpha$ is a simple relative root. The case of minimal roots is symmetric.
To establish (iii), we observe that $\Psi=\Phi_{\Pi,\Gamma}$ is naturally a root system,
and $\Psi_{\pi_{\Pi,\Gamma}(J),\{\id\}}=\Phi_{J,\Gamma}$. Applying (i) to $\Psi_{\pi_{\Pi,\Gamma}(J),\{\id\}}$,
we conclude that the images under $\pi_{\Pi,\Gamma}$ of all $\Pi$-maximal (respectively, $\Pi$-minimal) roots in
$\pi_{J,\Gamma}^{-1}(\alpha)$ coincide. By~\cite[Lemma 3]{PS} this implies that $\Gamma$ acts
transitively on such roots.
The claim (iv) follows similarly from (ii).
\end{proof}

\begin{lem}\label{lem:beta-correct}
Assume that $\rank(\Phi_{J,\Gamma})\ge 2$ and it is irreducible. Let $\alpha,\beta\in\Phi_{J,\Gamma}$
be two relative roots.
Let $a_{max}\in\pi_{J,\Gamma}^{-1}(\alpha)$ be a $\Pi$-maximal root, and let $b_{min}\in\pi_{J,\Gamma}^{-1}(\beta)$ be
a $\Pi$-minimal root. Let $a_1,\ldots,a_n\in\Phi^+$
be a sequence of roots such that
$$
b_{min}=a_{max}-a_1-\ldots -a_n,$$ and for all $1\le i\le n$ one has
$a_{max}-\sum_{k=1}^i a_k\in\Phi$, $\pi_{J,\Gamma}(a_{max}-\sum_{k=1}^i a_k)\neq 0$, and $\pi_{J,\Gamma}(a_i)\neq 0$. Then for any
$a\in\pi_{J,\Gamma}^{-1}(\alpha)$
there is a $\Pi$-minimal root $b'_{min}\in\pi_{J,\Gamma}^{-1}(\beta)$ and a sequence of
roots $a'_i\in\pi_{J,\Gamma}^{-1}(\pi_{J,\Gamma}(a_i))$, $1\le i\le n$,
such that
$$
b'_{min}=a-a'_1-\ldots-a'_n
$$ and $a-\sum_{k=1}^i a'_k\in\Phi$ for all $1\le i\le n$.
\end{lem}
\begin{proof}
We write $\pi$ instead of $\pi_{J,\Gamma}$ for short.
By the definition of $\Pi$-maximal roots, there is a $\Pi$-maximal root
$a'_{max}\in\pi^{-1}(\alpha)$ and a sequence of simple roots
$b_1,\ldots,b_k\in\Pi\setminus J$ such that each sum $a+b_1+\ldots +b_i$ is a root and $a+b_1+\ldots+b_k=a'_{max}$.
Since by Lemma~\ref{lem:max-min-roots} all $\Pi$-maximal roots in $\pi^{-1}(\alpha)$ are permuted by $\Gamma$,
there is $\sigma\in\Gamma$ such that $\sigma(a_{max})=a'_{max}$. Clearly, $b'_{min}=\sigma(b_{min})$ is a $\Pi$-minimal
root in $\pi^{-1}(\beta)$, and the pair $a'_{max}$, $b'_{min}$ satisfies the same assumptions as the pair
$a_{max}$, $b_{min}$. Therefore, we can assume that $a'_{max}=a_{max}$ and $b'_{min}=b_{min}$ without loss of generality.

The equality $a_{max}=b_{min}+a_1+a_2+\ldots+a_n$ then rewrites as
\begin{equation}\label{eq:many-a-b}
a+b_1+\ldots+b_k=b_{min}+a_1+a_2+\ldots+a_n.
\end{equation}
We prove that this last equality implies the equality
$$
a+b_1+\ldots+b_{k-1}=b_{min}+a'_1+a'_2+\ldots+a'_n
$$
for some new positive roots $a'_i$ such that $\pi(a'_i)=\pi(a_i)$ and $b_{min}+a'_1+a'_2+\ldots+a'_i$
is a root for all $1\le i\le n$. The claim of the lemma then
follows by induction on $k$.

The equality~\eqref{eq:many-a-b} together with the definition of $b_i$'s implies that
$(b_{min}+a_1+\ldots+a_{n-1})+a_n-b_k$ is a root. Observe that none of the three roots
$b_{min}+a_1+\ldots+a_{n-1}$, $a_n$ and $-b_k$ is opposite to another. Indeed,
we know that $(b_{min}+a_1+\ldots+a_{n-1})+a_n\neq 0$ since it is a root. Since $b_k\in\Pi\setminus J$ and
$\pi(a_n)\neq 0$, $a_n$ and $b_k$ are linearly independent. Similarly, since
$\pi(b_{min}+a_1+\ldots+a_{n-1})\neq 0$, $b_{min}+a_1+\ldots+a_{n-1}$ and $b_k$ are also linearly independent.
Then we can apply~\cite[Lemma 1]{PS}, which tells that
at least one of the expressions $(b_{min}+a_1+\ldots+a_{n-1})-b_k$ and
$a_n-b_k$ is a root too. In the first case, applying the same lemma several more times, we eventually
conclude that $a_i-b_k$ is a root for some $1\le i\le n-1$, since $b_{min}-b_k$ is not a root for any
positive simple root $b_k\in\Pi\setminus J$ by the $\Pi$-minimality of $b_{min}$ in $\pi^{-1}(\beta)$. Summing up, we see that $a_i-b_k$ is a root for some $1\le i\le n$.
Therefore,
$$
a+b_1+\ldots+b_{k-1}=b_{min}+a'_1+\ldots+a'_n,
$$
where $a'_i=a_i-b_k\in\pi^{-1}(\pi(a_i))$
and $a'_j=a_j$ for all $j\neq i$. Note that $b_{min}+a'_1+\ldots+a'_l$ is still a root for any $1\le l\le n$
by the choice of the index $i$.
\end{proof}

\subsection{Relative roots and relative root subschemes}\label{ssec:rel}

Let $G$ be a reductive group scheme over a commutative ring $R$.
By~\cite[Exp. XXII, Prop. 2.8]{SGA3} the root system $\Phi$ of $G_{\overline{k(s)}}$, $s\in\Spec R$,
is constant locally in the Zariski topology on $\Spec R$.
Let $P$ be a parabolic subgroup scheme of $G$ over $R$, and let $L$ be a Levi subgroup of $P$.
The type of the root system of
$L_{\overline{k(s)}}$ is determined by a Dynkin subdiagram
of the Dynkin diagram of $\Phi$, which is also constant Zariski-locally on $\Spec R$
by~\cite[Exp. XXVI, Lemme 1.14 and Prop. 1.15]{SGA3}. In particular, if $\Spec R$ is connected,
all these data are constant on $\Spec R$.

We denote by
$\ad_G:G\to G^{\ad}$ the canonical homomorphism of $G$ onto the corresponding adjoint group $G^{\ad}=G/\Cent(G)$.
We consider $G^{\ad}$ as a subgroup scheme of the automorphism group scheme $\Aut(G)$ in the sense of~\cite[Exp. XXIV]{SGA3}.

\begin{lem}\label{lem:relroots}\cite[Lemma 3.6]{St-serr}
Let $G$ be a reductive group over a connected commutative ring $R$.
Let $D$ be the Dynkin diagram of the root system $\Phi$ of $G_{\overline{k(s)}}$ for any $s\in\Spec R$. There is a subgroup
$\Gamma\le\Aut(D)$ such that for any parabolic $R$-subgroup $P$ of $G$ with a Levi subgroup $L$, there exist a unique
$\Gamma$-invariant
subset $J$ of $D$ and a unique maximal split $R$-subtorus
$S\subseteq\Cent(\ad_G(L))$ such that
for any $s\in\Spec R$ and any split maximal torus $T\subseteq\ad_G(L)_{\overline{k(s)}}$,
one can find  a bijection of $D$ onto a system of simple roots of $\Phi(T,G_{\overline{k(s)}})\cong\Phi$ satisfying

(i)  $L_{\overline{k(s)}}=U_{\ZZ(D\setminus J)}$ and $P_{\overline{k(s)}}=U_{\ZZ(D\setminus J)\cup \NN J}$;

(ii) the kernel of the surjective restriction homomorphism
\begin{equation}\label{eq:T-S}
X^*(T)\cong\ZZ\Phi\xrightarrow{\ \pi_P\ } X^*(S_{\overline{k(s)}})\cong \ZZ\Phi(S,G)
\end{equation}
is generated by all roots $r\in D\setminus J$,
and by all differences $r-\sigma(r)$, $r\in J$, $\sigma\in\Gamma$.
\end{lem}

\begin{dfn}
In the setting of Lemma~\ref{lem:relroots} we call $\Phi(S,G)=\Phi_{J,\Gamma}$ a \emph{system
of relative roots with respect to the parabolic subgroup $P$ over $R$} and denote it by $\Phi_P$.
\end{dfn}

If $A$ is a field or a local ring, and $P$ is a minimal parabolic subgroup of $G$,
then $\Phi_P$ is nothing but the root system of $G$ with respect to a maximal split subtorus
in the sense of~\cite{BorelTits} or, respectively,~\cite[Exp. XXVI \S 7]{SGA3}.

\begin{rem}
The system of relative roots $\Phi_P$ with respect to a parabolic subgroup $P$ was first introduced in~\cite{PS}, however,
the construction used there implicitly required the base ring to be noetherian. Every noetherian ring $R$ admits
a decomposition $R=\prod\limits_{i=1}^mR_i$,
where each $R_i$ is connected. The construction of $\Phi_P$ given in~\cite[Lemma 3.6]{St-serr} naturally coincides with the
construction of~\cite{PS} over every $R_i$.

\end{rem}

For any finitely generated projective $R$-module $V$, we denote by $W(V)$ the affine $R$-scheme
$\Spec\Sym^*(V^\vee)$, where
$V^\vee$ denotes the dual $R$-module of $V$, and $\Sym^*$ denotes the symmetric algebra.
Any morphism of $R$-schemes $W(V_1)\to W(V_2)$
is determined by an element $f\in\Sym^*(V_1^\vee)\otimes_R V_2$. If $f\in\Sym^d(V_1^\vee)\otimes_R V_2$,
we say that the corresponding morphism is
homogeneous of degree $d$.
By abuse of notation, we also write $f:V_1\to V_2$ and call it {\it a degree $d$
homogeneous polynomial map from $V_1$ to $V_2$}. Such a map $f$ satisfies
$$
f(\lambda v)=\lambda^d f(v)
$$
for any $v\in V_1$ and $\lambda\in R$.

\begin{lem}\cite[Lemma 3.9]{St-serr}\label{lem:relschemes}
In the setting of Lemma~\ref{lem:relroots}, for any $\alpha\in\Phi_P=\Phi(S,G)$ there exists a closed
$S$-equivariant  embedding of $R$-schemes
$$
X_\alpha\colon W\bigl(\Lie(G)_\alpha\bigr)\to G
$$
satisfying the following condition. Let $R'/R$ be any ring extension such that $G_{R'}$ is split
with respect to a maximal split $R'$-torus $T\subseteq L_{R'}$. Let $\Phi=\Phi(T,G_{R'})$ be the corresponding
root system, let
$$
\pi_P:X^*(T)\cong\ZZ\Phi\to X^*(S_{R'})\cong\ZZ\Phi_P
$$
be the restriction homomorphism,
and let
$x_r\colon\Ga_{,R'}\to G_{R'}$ be the 1-parameter root subgroups of $G_{R'}$ corresponding to a Chevalley system
$e_r\in\Lie(G_{R'})_r$, $r\in\Phi$, in the sense
of~\cite[Exp. XXIII, D\'ef. 6.1]{SGA3}.
Then for any
$u=\hspace{-8pt}\sum\limits_{r\in\pi_P^{-1}(\alpha)}\hspace{-8pt}a_r e_r\in\Lie(G_{R'})_\alpha$
one has
\begin{equation}\label{eq:Xalpha-prod}
X_\alpha(u)=
\Bigl(\prod\limits_{r\in\pi_P^{-1}(\alpha)}\hspace{-8pt}x_{r}(a_r)\Bigr)\cdot
\prod\limits_{i\ge 2}\Bigl(\prod\limits_{
\st
r\in \pi_P^{-1}(i\alpha)
}
\hspace{-8pt}x_r(p^i_{r}(u))\Bigr),
\end{equation}
where each $p^i_{r}:\Lie(G_{R'})_\alpha\to R'$, $r\in\pi_P^{-1}(i\alpha)$, is a homogeneous polynomial map of degree $i$,
and the products over $r$ are taken in any prescribed order.

\end{lem}

\begin{dfn}
Closed embeddings $X_\alpha$, $\alpha\in\Phi_P$, satisfying the statement of Lemma~\ref{lem:relschemes},
are called \emph{relative root subschemes of $G$ with respect to the parabolic subgroup $P$}.
\end{dfn}

Relative root subschemes of $G$ with respect to $P$ a priori
depend on the choice of a Levi subgroup $L$ in $P$, but since all Levi subgroups are
conjugate under $U_P(R)$~\cite[Exp. XXVI Cor. 1.8]{SGA3}, our proofs do not depend on this choice,
so we usually omit $L$ from the notation.

We will use the following properties of relative root subschemes.

\begin{lem}\label{lem:rootels}\cite[Theorem 2, Lemma 6, Lemma 9]{PS}
Let $X_\alpha$, $\alpha\in\Phi_P$, be as in Lemma~\ref{lem:relschemes}.
Set $V_\alpha=\Lie(G)_\alpha$ for short. Then

(i) There exist degree $i$ homogeneous polynomial maps $q^i_\alpha:V_\alpha\oplus V_\alpha\to V_{i\alpha}$, i>1,
such that for any $R$-algebra $R'$ and for any
$v,w\in V_\alpha\otimes_R R'$ one has
\begin{equation}\label{eq:sum}
X_\alpha(v)X_\alpha(w)=X_\alpha(v+w)\prod_{i>1}X_{i\alpha}\left(q^i_\alpha(v,w)\right).
\end{equation}

(ii) For any $g\in L(R)$, there exist degree $i$ homogeneous polynomial maps
$\varphi^i_{g,\alpha}\colon V_\alpha\to V_{i\alpha}$, $i\ge 1$, such that for any $R$-algebra $R'$ and for any
$v\in V_\alpha\otimes_R R'$ one has
$$
gX_\alpha(v)g^{-1}=\prod_{i\ge 1}X_{i\alpha}\left(\varphi^i_{g,\alpha}(v)\right).
$$

(iii) \emph{(generalized Chevalley commutator formula)} For any $\alpha,\beta\in\Phi_P$
such that $m\alpha\neq -k\beta$ for all $m,k\ge 1$,
there exist polynomial maps
$$
N_{\alpha\beta ij}\colon V_\alpha\times V_\beta\to V_{i\alpha+j\beta},\ i,j>0,
$$
homogeneous of degree $i$ in the first variable and of degree $j$ in the second
variable, such that for any $R$-algebra $R'$ and for any
for any $u\in V_\alpha\otimes_R R'$, $v\in V_\beta\otimes_R R'$ one has
\begin{equation}\label{eq:Chev}
[X_\alpha(u),X_\beta(v)]=\prod_{i,j>0}X_{i\alpha+j\beta}\bigl(N_{\alpha\beta ij}(u,v)\bigr)
\end{equation}

(iv) For any subset $\Psi\subseteq X^*(S)\setminus\{0\}$ that is closed under addition,
the morphism
$$
X_\Psi\colon W\Bigl(\bigoplus_{\st \alpha\in\Psi\cap\Phi_P}\!\!V_\alpha\Bigr)\to U_\Psi,\qquad
(v_\alpha)_\alpha\mapsto\prod_\alpha X_\alpha(v_\alpha),
$$
where the product is taken in any fixed order,
is an isomorphism of $R$-schemes.
\end{lem}

Note that Lemma~\ref{lem:rootels} (iv) implies, in particular, that
for any $R$-algebra $R'$, one has
$$
U_{P^\pm}(R')=\left<X_\alpha(R'\otimes_R V_\alpha),\ \alpha\in\Phi_P^\pm\right>.
$$
Consequently,
$$
E_P(R')=\left<X_\alpha(R'\otimes_R V_\alpha),\ \alpha\in\Phi_P\right>.
$$
For any $\alpha\in\Phi_P$, set $(\alpha)=\NN\alpha\subseteq\ZZ\Phi_P$.
Then $U_{(\alpha)}=\prod\limits_{k\ge 1}X_{k\alpha}$ is a closed subscheme of $G$
satisfying
$$
U_{(\alpha)}(R')=\l<X_{k\alpha}(V_{k\alpha}\otimes_R R'),\ k\ge 1\r>
$$
for any $R$-algebra $R'$. This notation
coincides with that of~\cite{BorelTits} in case of isotropic reductive groups over a field.

\begin{lem}\label{lem:PQ-relroots}
Under the assumptions of Lemma~\ref{lem:relroots}, let $R'$ be a connected commutative $R$-algebra.
Let $Q\le P_{R'}$ be a parabolic $R'$-subgroup of $G_{R'}$, let $L_Q\le (L_P)_{R'}$ be a Levi subgroup of $Q$,
and let $J'$, $\Gamma'$, and $S'$ be the subset of $D$, the subgroup of
$\Aut(D)$, and the maximal split subtorus of $\Cent(\ad_G(L_Q))$
corresponding to $Q$. Then

(i) $J\subseteq J'$, $\Gamma'\le\Gamma$, $S_{R'}\le S'$, and the restriction of characters $\pi_{PQ}:X^*(S')\to X^*(S_{R'})$
completes the commutative triange
\begin{equation*}
\xymatrix@R=20pt@C=35pt{
\ZZ\Phi\ar[rd]_{\pi_P}\ar[r]^{\pi_{Q}\qquad\quad}&\ZZ\Phi_{Q}\cong\ZZ\Phi_{J',\Gamma'}\ar[d]^{\pi_{PQ}}\\
&\ZZ\Phi_P\cong\ZZ\Phi_{J,\Gamma}\\
}
\end{equation*}
(ii) for any $\alpha\in\Phi_P$ and any vector
$$
v=\sum\limits_{\hbox to 9pt{$\st \delta\in\pi_{PQ}^{-1}(\alpha)$}}v_\delta\,\in\, V_\alpha\otimes_R R'=
\Lie(G_{R'})_\alpha=\bigoplus\limits_{\hbox to 15pt{$\st \delta\in\pi_{PQ}^{-1}(\alpha)$}} V_\delta
$$
one has
$$
X_\alpha(v)=\prod\limits_{\hbox to 9pt{$\st \delta\in\pi_{PQ}^{-1}(\alpha)$}} X_\delta(v_\delta)\cdot\prod_{i\ge 2}X_{i\alpha}(f^i_\alpha(v)),
$$
where each $f^i_\alpha:V_\alpha\to V_{i\alpha}$ is a polynomial map of degree $i$.
\end{lem}
\begin{proof}
The claim (i) is clear except possibly for the inclusion $\Gamma'\le\Gamma$. For the latter,
recall~\cite[Lemma 3.6]{St-serr} that $\Gamma\le\Aut(D)$ represents the action of the Galois
group $\Gal(\tilde R/R)$ on $D$, where $\tilde R/R$ is a finite Galois ring extension that splits the Dynkin scheme
$\Dyn(G)$ of~~\cite[Exp. XXIV, \S 3.7]{SGA3}. The group $\Gamma'\le\Aut(D)$ represents the corresponding Galois group
for $\Dyn(G_{R'})=\Dyn(G)\times_R R'$. Since $\tilde R\otimes_R R'/R'$ is a finite Galois extension that splits $\Dyn(G_{R'})$,
the claim follows.

To prove (ii), first reduce to $v=v_\delta$ by Lemma~\ref{lem:rootels} (i).
Since $\Psi=\pi_{PQ}^{-1}(\NN\alpha)$ is an additively closed subset of $\ZZ\Phi_Q\setminus\{0\}$,
by Lemma~\ref{lem:rootels} (iv) one has $X_{\delta}(v_\delta)\in U_{\Psi}(R')$. Since
$\Lie(U_\Psi)=\Lie(U_{\NN\alpha}\times_R R')$, we have $U_\Psi=U_{\NN\alpha}\times_R R'$ by Lemma~\ref{lem:T-P}. Hence
$X_\delta(v_\delta)=\prod_{i\in\NN}X_{i\alpha}(u_i)$, $u_i\in V_{i\alpha}\otimes_R R'$.
Locally in the fpqc-topology on $\Spec(R')$, $G$ is split with respect
to a split maximal torus contained in $L$~\cite[Exp. XXII Cor. 2.3]{SGA3}, and one can choose a Chevalley basis of $\Lie(G)$
adapted to $Q$ and $L_Q$~\cite[Exp. XXVI Lemme 1.14]{SGA3}. Since $Q\le P_{R'}$ and $L_Q\le (L_P)_{R'}$, it is also
adapted to $P_{R'}$ and $(L_P)_{R'}$. Thus, fpqc-locally there are
presentations~\eqref{eq:Xalpha-prod} of Lemma~\ref{lem:relschemes} for $X_\delta(v_\delta)$ and $X_{i\alpha}(u_i)$.
Comparing these presentations and applying inverse induction on $i$ concludes the proof.
\end{proof}

Apart from the above properties of relative root subschemes, we will use the following lemma
which is a modification of~\cite[Lemma 10]{PS} and~\cite[Lemma 2]{LS}.
Note that both these earlier versions of the lemma are imprecise in claiming that the set-theoretic image $\im(N_{\alpha,\beta,1,1})$
(respectively, the sum of images in (2)) coincides with $V_{\alpha+\beta}$, while it only additively generates
$V_{\alpha+\beta}$. This correction does not affect the validity of any proofs in~\cite{PS,LS,St-poly}.

\begin{dfn}
Let $\Psi$ be a reduced irreducible root system in the sense of~\cite{Bu}. The \emph{structure constants} of $\Psi$
are the structure constants of the simple complex Lie algebra of type $\Psi$, i.e. the integers $\{\pm 1\}$
if $\Psi=A_l,D_l,E_6,E_7,E_8$, the integers $\{\pm 1,\pm 2\}$ if $\Psi=B_l,C_l,F_4$, and $\{\pm 1,\pm 2,\pm 3\}$
if $\Psi=G_2$.
\end{dfn}

\begin{lem}\label{lem:const}
Consider $\alpha,\beta\in\Phi_P$ satisfying $\alpha+\beta\in\Phi_P$ and $m\alpha\neq -k\beta$ for any $m,k\ge 1$.
Denote by $\Phi^0$ an irreducible component of $\Phi$ such that $\alpha,\beta\in\pi_P(\Phi^0)$.

{\rm (1)} In each of the following cases,
$\im( N_{\alpha\beta 11})$ generates $V_{\alpha+\beta}$ as an abelian group:

\quad {\rm (a)} structure constants of $\Phi^0$ are invertible in $R$ (for example, $\Phi^0$ is simply laced);

\quad {\rm (b)} $\alpha\neq \beta$ and $\alpha-\beta\not\in\Phi_P$;

\quad {\rm (c)} $\Phi^0$ is of type $B_l$, $C_l$, or $F_4$,
and $\pi_P^{-1}(\alpha+\beta)$ consists of short roots;

\quad {\rm (d)} $\Phi^0$ is of type $B_l$, $C_l$, or $F_4$, and there exist long roots
$a\in\pi_P^{-1}(\alpha)$, $b\in\pi_P^{-1}(\beta)$ such that $a+b$ is a root.

Moreover, in each of the cases {\rm (a)}, {\rm (b)}, or {\rm (c)},
for any $R$-algebra $R'$ and $0\neq u\in V_\beta\otimes_R R'$ one has $\im (N_{\alpha\beta 11}(-,u))\neq 0$.

{\rm (2)} If $\alpha-\beta\in\Phi_P$ and $\Phi^0\neq G_2$, then\;
$\im(N_{\alpha\beta 11})$,\, $\im( N_{\alpha-\beta,2\beta,1,1})$, and\,
$\im (N_{\alpha-\beta,\beta,1,2})$ together
generate $V_{\alpha+\beta}$ as an abelian group. Here we assume
$\im(N_{\alpha-\beta,2\beta,1,1})=0$ if $2\beta\not\in\Phi_P$.
\end{lem}
\begin{proof}
{\rm (1)} Since $N_{\alpha\beta11}$ is bilinear, in order to prove that its image generates $V_{\alpha+\beta}$ as an abelian
group, it is enough to show that the $R$-submodule generated by $\im( N_{\alpha\beta 11})$ coincides with $V_{\alpha+\beta}$.
The latter can be proved locally in the fpqc-topology on $\Spec R$. The second claim also can be checked fpqc-locally,
since it is equivalent to $\im(N_{\alpha\beta11}(-,u))\neq 0$.  Locally in the fpqc-topology, $G$ is split with respect
to a split maximal torus contained in $L$~\cite[Exp. XXII Cor. 2.3]{SGA3}, and one can choose a Chevalley basis of $\Lie(G)$
adapted to $P$ and $L$~\cite[Exp. XXVI Lemme 1.14]{SGA3}. To simplify the notation, assume that this is the case already
over $R$; then we have the presentation~\eqref{eq:Xalpha-prod} of Lemma~\ref{lem:relschemes}.

The $R$-module $V_{\alpha+\beta}$ is generated by vectors $e_c$, $c\in\pi_P^{-1}(\alpha+\beta)$.
By~\cite[Lemma 4]{PS} for every $c\in\pi_P^{-1}(\alpha+\beta)$ there are $a\in\pi_P^{-1}(\alpha)$ and $b\in\pi_P^{-1}(\beta)$
such that $c=a+b$. Let $u_a\in V_\alpha$, $u_b\in V_\beta$ be such that
$$
X_\alpha(u_a)=x_a(1)\cdot\prod\limits_{i\ge 2}\prod\limits_{r\in\pi_P^{-1}(i\alpha)}x_r(p_r^i(u_a))
\quad\mbox{and}\quad
X_\beta(u_b)=x_b(1)\cdot \prod\limits_{i\ge 2}\prod\limits_{r\in\pi_P^{-1}(i\beta)}x_r(p_r^i(u_b)).
$$
Then the (usual) Chevalley commutator formula implies that $N_{\alpha\beta 11}(u_a,u_b)=\lambda e_{a+b}$,
where $\lambda\in\{\pm 1,\pm 2,\pm 3\}$ is a structure constant
of $\Phi$.  Now if the assumption (a) holds, then $\lambda$ is invertible. If
(b), (c) or (d) holds, then $\lambda$ necessarily equals $\pm 1$. Indeed, in the
only dubious case (d) one should
note that, due to the transitive action of the Weyl group of $L$ on the roots of the same length in $\pi_P^{-1}(\alpha+\beta)$
(see~\cite[Lemma 1]{ABS}), {\it any} long root $c\in\pi_P^{-1}(\alpha+\beta)$
decomposes as a sum of long roots $a$ and $b$. Hence $\lambda$ is always invertible.
This shows that $\im(N_{\alpha\beta 11})$ generates $V_{\alpha+\beta}$
as an $R$-module.

To prove the second claim of {\rm (1)}, for a given $u\in V_\beta\otimes_R R'$ write
\begin{equation*}
X_\beta(u)=\prod_{\hbox to 15pt{$\st b\in\pi_P^{-1}(\beta)$}}x_{b}(\nu_b)\cdot
\prod_{i\ge 2}\prod_{r\in\pi_P^{-1}(i\beta)}x_r(p_r^i(u)),\quad \nu_b\in R'.
\end{equation*}
Since $X_\beta(u)\neq 0$, by Lemma~\ref{lem:rootels} (iv) there exists $\nu_b\neq 0$.
By~\cite[Lemma 4]{PS} there exists
a root $a\in\pi_P^{-1}(\alpha)$ such that $a+b\in\Phi$. Let $v\in V_\alpha$ be such that
\begin{equation*}
X_\alpha(v)=x_a(1)\cdot\prod\limits_{i\ge 2}\prod\limits_{r\in\pi_P^{-1}(i\alpha)}x_r(p_r^i(v)).
\end{equation*}
Then $N_{\alpha\beta 11}(v,u)$ is a sum of $\lambda e_{a+b}$ and a linear combination of basis vectors of $V_{\alpha+\beta}$
 of the form $e_{a+b'}$, $b'\in\pi_P^{-1}(\beta)$,
$a+b'\in\Phi$. By the same token as above $\lambda$ is invertible, and hence $N_{\alpha\beta 11}(v,u)\neq 0$.

{\rm (2)} The proof of~\cite[Lemma 10 2)]{PS} carries over verbatim.
\end{proof}

\begin{dfn}\label{def:E(R,I)}
For any $R$-algebra $R'$, any ideal $I$ of $R'$, and any $\alpha\in\Phi_P$ we denote
$$
\begin{array}{l}
G(R',I)=\ker \bigl(G(R')\to G(R'/I)\bigr),\\
U_{(\alpha)}(I)=\l<X_{k\alpha}(IV_\alpha),\ k\in\NN\r>=U_{(\alpha)}(R')\cap G(R',I),\\
E_P(I)=\l<X_\alpha(IV_\alpha),\ \alpha\in\Phi_P\r>,\\
E_P(R',I)=E_P(I)^{E_P(R')}=\mbox{ the normal closure of $E_P(I)$ in }E_P(R'),\\
E_\alpha(I)=\left<U_{(\alpha)}(I),\,U_{(-\alpha)}(I)\right>.
\end{array}
$$
\end{dfn}

\section{Steinberg groups associated to parabolic subgroups}\label{sec:K2}

\subsection{Basic properties of Steinberg groups}\label{sec:St}
Let $R$ be a connected commutative ring,
let $G$ be an isotropic reductive group over $R$, $P,P^-\subseteq G$ two opposite parabolic subgroups of $G$,
$L_P=P\cap P^-$ their common Levi subgroup. By Lemmas~\ref{lem:relroots} and~\ref{lem:relschemes} we are given a system
of relative roots $\Phi_P$ and relative root suschemes $X_\alpha(V_\alpha)$, $\alpha\in\Phi_P$, corresponding to
$P$.

\begin{dfn}\label{dfn:St}
For any commutative $R$-algebra $R'$, the \emph{Steinberg group} associated to $G$ and $P$ over $R'$ is the group
$St_P(R')$ generated as an
abstract group by elements $\tilde X_\alpha(u)$, $\alpha\in\Phi_P$, $u\in V_\alpha\otimes_R R'$,
subject to the relations
\begin{equation}\label{eq:sum-St}
\begin{array}{c}
\tilde X_\alpha(v)\tilde X_\alpha(w)=\tilde X_\alpha(v+w)\prod\limits_{i>1}\tilde X_{i\alpha}\l(q^i_\alpha(v,w)\r)\ \mbox{for all}
\ \alpha\in\Phi_P,\ v,w\in V_\alpha\otimes_R R',
\end{array}
\end{equation}
and
\begin{equation}\label{eq:Chev-St}
\begin{array}{c}
[\tilde X_\alpha(u),\tilde X_\beta(v)]=\prod\limits_{i,j>0}\tilde X_{i\alpha+j\beta}\bigl(N_{\alpha\beta ij}(u,v)\bigr),
\quad\mbox{for all}\ \alpha,\beta\in\Phi_P\ \mbox{such that}\\
m\alpha\neq-k\beta\ \mbox{for any}\ m,k>0,\ \mbox{and all }\ u\in V_\alpha\otimes_R R',\ v\in V_\beta\otimes_R R'.
\end{array}
\end{equation}
\end{dfn}

\begin{rem}
If $G$ is a Chevalley group of rank $\ge 2$ over $R$ and $P$ is a Borel subgroup, then our definition coincides with the standard definition of
the Steinberg group $St(\Phi,R)$ corresponding to $G$~\cite{Steinberg,Steinberg-book}. If $R$ is a field, $P$ is a minimal parabolic subgroup
and $\rank\Phi_P\ge 2$, it also coincides with the definition of V. Deodhar~\cite[1.9]{Deo}.
On the other hand, if $\rank(\Phi_P)=1$, then our definition does not include any relations of the commutator
type, so it is not the "right"{} definition. However, it is convenient for the purposes of the present paper.
\end{rem}

It is clear that $St_P(-)$ is a functor on the category of commutative $R$-algebras $R'$.
We denote by
$$
\pst_P=\pst_P(R'): St_P(R')\to E_P(R')
$$
the natural functorial surjection.

For any subgroup of $E_P(R)$ generated by a set of relative root elements $X_\alpha(v)$, we will
denote the subgroup of $St_P(R)$ generated by the respective liftings $\tilde X_\alpha(v)$ by the same
letter or combination of letters, but with $\tilde{}$ on top. In particular,
for any set $S\subseteq\ZZ\Phi_P$ and any ideal $I\subseteq R$ we write
$$
\begin{array}{c}
\tilde U_S(I)=\l<\tilde X_\alpha(IV_\alpha),\ \alpha\in S\r>,\\
\tilde E_\alpha(R)=\tilde U_{\ZZ\alpha}(R)=\l<\tilde U_{(\alpha)}(R),\tilde U_{(-\alpha)}(R)\r>.\\
\end{array}
$$

\begin{lem}\label{lem:mono}
Let $\Psi\subseteq \ZZ\Phi_P\setminus 0$ be an additively closed subset.
Then $\pst_P|_{\tilde U_{\Psi}(R)}:\tilde U_{\Psi}(R)\to U_{\Psi}(R)$ is a group isomorphism.
\end{lem}
\begin{proof}
Since $0\not\in\Psi$, there is a partial
order $\le$ on the finite set $S=\Psi\cap\Phi_P$ generated by $\alpha<\alpha+\beta$ for any $\alpha,\beta\in S$.
The relation~\eqref{eq:Chev-St}
then says that for any $\alpha,\beta\in S$ the commutator $[X_\alpha(u),X_\beta(v)]$ is a product of
elements in $S$ strictly bigger than $\alpha$ and $\beta$. This shows that
if $\alpha_1,\alpha_2,\ldots,\alpha_n$ are all elements in $S$ in any order,
then any $x\in \tilde U_{\Psi}(R)$ can be written
(possibly, non-uniquely) as $x=\prod\limits_{i=1}^n\tilde X_{\alpha_i}(u_i)$ for some  $u_i\in V_{\alpha_i}$.
By Lemma~\ref{lem:rootels} (iv) the corresponding
presentation for $\pst_P(x)\in U_{\Psi}(R)$ is unique, hence $\pst_P(x)=1$ implies
$\pst_P(\tilde X_{\alpha_i}(u_i))=X_\alpha(u_i)=1$ for all $i$. Hence $u_i=0$ for all $i$.
\end{proof}

Note that Lemma~\ref{lem:mono} applied to $\Psi=\Phi_P^\pm$ implies that the two subgroups
 $\tilde U_{\Phi_{P}^\pm}(R)$ are isomorphic to $U_{P^\pm}(R)$ via $\pst_P$; clearly, these
subgroups generate $St_P(R)$.

\begin{lem}\label{lem:phiPQ}
Let $Q\le P$ be two parabolic subgroups of $G$ over $R$, with Levi subgroups $L_Q\le L_P$ and opposite parabolic subgroups
$Q^-\le P^-$.
For any commutative $R$-algebra $R'$ the inclusions $U_{P^\pm}(R')\le U_{Q^\pm}(R')$ induce a natural group
homomorphism
$$
\phi_{PQ}=\phi_{PQ}(R'):St_P(R')\to St_Q(R')
$$
such that $\phi_{PQ}(\tilde X_\alpha(v))=\bigl(\pst_Q|_{\tilde U_{\Phi_Q^\pm}(R')}\bigr)^{-1}(X_\alpha(v))$ for all $\alpha\in\Phi_P^\pm$,
$v\in V_\alpha\otimes_R R'$.
\end{lem}
\begin{proof}
For any additively closed subset $\Psi\subseteq\ZZ\Phi_P\setminus\{0\}$ the set
$\pi_{PQ}^{-1}(\Psi)\subseteq\ZZ\Phi_Q\setminus\{0\}$ is also additively closed. By the unicity part of Lemma~\ref{lem:T-P}
together with Lemma~\ref{lem:rootels} (iv) we have
$U_{\Psi}=U_{\pi_{PQ}^{-1}(\Psi)}$ inside $G$. Then by Lemma~\ref{lem:mono} the groups $\tilde U_\Psi(R)\le St_P(R)$
and $\tilde U_{\pi_{PQ}^{-1}(\Psi)}(R)\le St_Q(R)$ are isomorphic by means of two mutually inverse isomorphisms
$\pst_P^{-1}\circ\pst_Q$ and $\pst_Q^{-1}\circ\pst_P$. Taking $\Psi=\NN\alpha$ for the relation~\eqref{eq:sum-St}
and
$$
\Psi=(\NN\alpha+\NN\beta)\cup\{\alpha,\beta\}
$$
for the relation~\eqref{eq:Chev-St}, we conclude that
the elements
$$
\phi_{PQ}(\tilde X_\alpha(u))=\bigl(\pst_Q|_{\tilde U_{\pi_{PQ}^{-1}(\Psi)}(R)}\bigr)^{-1}(X_\alpha(u))\in St_Q(R),\quad
\alpha\in\Phi_P,\ u\in V_\alpha,
$$
satisfy the same relations as the original elements of $St_P(R)$. Thus, there is a correctly defined homomorphism
$\phi_{PQ}:St_P(R)\to St_Q(R)$ such that $\pst_Q\circ\phi_{PQ}=\pst_P$.
\end{proof}

The following lemma slightly extends~\cite[Lemma 11]{PS}.

\begin{lem}\label{lem:lemma11}
Let $R[Y,Z]$ be a polynomial ring in two variables over $R$.
Suppose that $\alpha\in\Phi_P$ lies in an irreducible component of rank $\ \ge 2$.
Then there exist relative roots
$\beta_i,\gamma_i\in\Phi_P$, non-collinear to $\alpha$,
and integers $k_i,l_i>0$, $n_i\ge 0$ {\rm ( $1\le i\le m$ )},
such that for any $v\in V_\alpha$ there are $v_i\in V_{\beta_i}$ and  $u_i\in V_{\gamma_i}$ {\rm ( $1\le i\le m$ )}
satisfying 
\begin{equation}\label{eq:lemma11}
\tilde X_\alpha(YZ^2 v)=\prod\limits_{i=1}^m \tilde X_{\beta_i}(Y^{k_i}Z^{n_i} v_{i})^{\ds \tilde X_{\gamma_i}(Z^{l_i} u_{i})}
\end{equation}
in $St_P(R[Y,Z])$. If every module $V_\beta$, $\beta\in\Phi_P$, is generated by $\le M$ elements,
then there is a presentation~\eqref{eq:lemma11} with $m\le 4m_\alpha^2\cdot M^3\cdot|\Phi_P|$.
\end{lem}
\begin{proof}
We assume that $\alpha\in\Phi_P^+$ and prove the claim
by inverse induction on the height of $\alpha$.
By~\cite[Lemma 5]{PS} there are non-collinear relative roots $\beta,\gamma\in\Phi_P$ such that
$\alpha=\beta+\gamma$, and moreover $\beta-\gamma\not\in\Phi_P$ if $\alpha$ belongs to the image of an irreducible
component of $\Phi$ of type $G_2$.
Then by Lemma~\ref{lem:const} the vector $v$ is a finite sum of elements of
$\im(N_{\beta\gamma 11})$, together with elements of $\im (N_{\beta-\gamma,\gamma,1,2})$, if $\beta-\gamma\in\Phi_P$,
and elements in $\im( N_{\beta-\gamma,2\gamma,1,1})$, if $2\gamma\in\Phi_P$.  Since $N_{\beta\gamma 11}$ and
$N_{\beta-\gamma,2\gamma,1,1}$ are bilinear, and $N_{\beta-\gamma,\gamma,1,2}$ is $(1,2)$-homogeneous,
 $v$ is a sum of no more than
$$
\mu(V_\beta)\cdot\mu(V_\gamma)+\mu(V_{\beta-\gamma})\cdot\bigl(\mu(V_{2\gamma})+
\mu(V_\gamma)+\mu(V_\gamma)^2\bigr)\le 4M^3
$$
elements of the respective images, where $\mu(V)$ denotes the minimum number of generators of $V$.
Write $v=u+v'$, where $u\in V_\alpha$ belongs
to $\im(N_{\beta\gamma 11})$, $\im (N_{\beta-\gamma,\gamma,1,2})$, or $\im( N_{\beta-\gamma,2\gamma,1,1})$.
By~\eqref{eq:sum-St} one has
\begin{equation}\label{eq:lemma11-sum}
\tilde X_\alpha(YZ^2v)\in \tilde X_\alpha(YZ^2u)\cdot \tilde X_\alpha(YZ^2v')\cdot U,
\end{equation}
where $U=\prod\limits_{i=2}^{m_\alpha}\tilde X_{i\alpha}(Y^iZ^{2i}V_{i\alpha})$ is a subgroup of $\tilde U_{(\alpha)}(R[Y,Z])$.
By the relation~\eqref{eq:Chev-St} $\tilde X_\alpha(YZ^2u)$
is a product \hbox{of $\le|\Phi_P|$ factors} of the form $\tilde X_{\beta_i}(Y^{k_i}Z^{n_i} v_{i})^{\ds \tilde X_{\gamma_i}(Z^{l_i} u_{i})}$
as in~\eqref{eq:lemma11} (where $\beta_i$ and $\gamma_i$ are pairwise distinct roots of the form
$\lambda\beta+\mu\gamma$ with $\lambda\in\NN$, $\mu\in\ZZ$,
$\lambda\neq\mu$), and of an element lying in $U$.
By~\eqref{eq:Chev-St} $\tilde X_\alpha(YZ^2v')$ normalizes $U$.
Hence, iterating the decomposition~\eqref{eq:lemma11-sum} for $v'$, we conclude that $\tilde X_\alpha(YZ^2v)$ is a product
of $\le 4M^3\cdot|\Phi_P|$ factors of the required form
$\tilde X_{\beta_i}(Y^{k_i}Z^{n_i} v_{i})^{\ds \tilde X_{\gamma_i}(Z^{l_i} u_{i})}$, and of an element in $U$.
By the induction hypothesis,
any element of $\tilde X_{i\alpha}(Y^iZ^{2i}V_{i\alpha})$, $2\le i\le m_\alpha$, admits a decomposition as in~\eqref{eq:lemma11} with
$\le 4m_{i\alpha}^2\cdot M^3\cdot|\Phi_P|$ factors. Since $m_{i\alpha}\le \frac { m_\alpha}i$,
the total number of factors for $\tilde X_\alpha(YZ^2v)$ is no more than
$$
\Bigl(1+\sum_{i=2}^{m_\alpha}\left(\frac {m_\alpha}i\right)^2\Bigr)\cdot 4M^3\cdot|\Phi_P|=
\Bigl(1+m_\alpha^2\sum_{i=2}^{m_\alpha}\frac 1{i^2}\Bigr)\cdot 4M^3\cdot|\Phi_P|\le 4 m_\alpha^2\cdot M^3\cdot|\Phi_P|.
$$
\end{proof}

The following lemma is an extension of~\cite[Lemma 12]{PS}.

\begin{lem}\label{lem:lemma12}
In the setting of Lemma~\ref{lem:phiPQ},
let $\alpha\in\Phi_Q$ be a root that belongs to an irreducible
component $\Phi^0$ of $\Phi_Q$ such that $\pi_{PQ}(\Phi^0)\neq 0$. Set
$N=\rank\Phi_Q-\rank\Phi_P$. Then

(i) there exist relative roots
$\beta_i,\gamma_{ij}\in\Phi_{P}$,  and integers $k_i,n_i,l_{ij}>0$ {\rm ($1\le i,j\le m$)}
such that for any $v\in V_\alpha$ there are elements $v_i\in V_{\beta_i}$,  $u_{ij}\in V_{\gamma_{ij}}$
satisfying
\begin{equation}\label{eq:X-PQ}
\tilde X_\alpha(YZ^k v)=
\prod\limits_{i=1}^m \phi_{PQ}\bigl(\tilde X_{\beta_i}(Y^{k_i}Z^{n_i} v_{i})\bigr)^{\ds\mbox{$\prod\limits_{j=1}^{m}$}
\phi_{PQ}\bigl(\tilde X_{\gamma_{ij}}(Z^{l_{ij}} u_{ij})\bigr)},
\end{equation}
where $k=3^N$, and if, moreover,
every module $V_\beta$, $\beta\in\Phi_Q$, is generated by $\le M$ elements, then
$m\le |\Phi_Q|^{4N^2}M^{3N}$.

(ii) there exists $\alpha'\in\Phi_P$ such that $\tilde E_\alpha(R)\le \phi_{PQ}(\tilde E_{\alpha'}(R))$.

In particular,
if $\Phi_Q$ is irreducible, then $\phi_{PQ}$ is surjective.
\end{lem}
\begin{proof}

(i)  Let $J'\subseteq J\subseteq\Pi$ be the two sets of simple roots in $\Phi$ corresponding to $P$ and $Q$
respectively. Both sets $J$ and $J'$
are invariant under the group of automorphisms $\Gamma_{\tau}=\Gamma$,
that is, are unions of some $\Gamma$-orbits of simple roots. Then $J\setminus J'$ is a union of
$N=\rank\Phi_Q-\rank\Phi_P$ distinct $\Gamma$-orbits. We prove the claim by induction on $N$. If $N=0$, the claim is clear.
Now if $N\ge 1$, take any $\Gamma$-orbit $O\subseteq J\setminus J'$ and let $\alpha_r\in O$ be a simple root.
By Lemma~\ref{lem:T-P} there exists a parabolic
subgroup $Q\le Q'\le P$ of $G$ such that $\Phi_{Q'}=\Phi_{J\setminus O,\Gamma}$.
By the induction assumption the claim of the lemma is satisfied for the pair $Q'\le P$.
Note that for any $\beta\in\Phi_{Q'}$ one has $V_\beta=\bigoplus_{\gamma\in\pi_{Q'Q}^{-1}(\beta)}V_\gamma$,
hence if all modules $V_\gamma$ are generated by $\le M$ elements, then
$V_\beta$ is generated by $\le |\Phi_Q|\cdot M$ elements.

If $\alpha\in\Phi_Q$ satisfies $\pi_{Q'Q}(\alpha)\neq 0$, then Lemma~\ref{lem:PQ-relroots} together with Lemma~\ref{lem:mono}
imply that
$$
\tilde X_\alpha(Yv)=\prod_{i\in\NN}\tilde X_{i\pi_{Q'Q}(\alpha)}(Y^iv_i)
$$
for some $v_i\in V_{i\pi_{Q'Q}(\alpha)}$. Replacing $Y$ by $YZ^{3^{N-1}}$ and applying the induction assumption
we obtain~\eqref{eq:X-PQ} with $k=3^{N-1}$ and
$$
m\le m_{\pi_{Q'Q}(\alpha)}\cdot |\Phi_Q|^{4(N-1)^2}\cdot (|\Phi_Q|\cdot M)^{3(N-1)}\le M^{3(N-1)}\cdot
|\Phi_Q|^{4(N-1)^2+3(N-1)+1}.
$$
Assume $\pi_{Q'Q}(\alpha)=0$. Then $\alpha\in\ZZ\pi_Q(\alpha_r)\cap\Phi_Q$, and any root $\beta\in\Phi_Q$ that is non-collinear to
$\alpha$ satisfies $\pi_{Q'Q}(\beta)\neq 0$.
Apply to $\tilde X_\alpha$ the decomposition~\eqref{eq:lemma11} of Lemma~\ref{lem:lemma11} replacing
$Y$ by $YZ^{3^{N-1}}$ and $Z$ by $Z^{3^{N-1}}$.
Then by the previous case we obtain a decomposition~\eqref{eq:X-PQ} for $\tilde X_\alpha(YZ^{3^N}v)$ satisfying
$$
\begin{array}{rcl}
m&\le& 4m_\alpha^2\cdot M^3\cdot |\Phi_Q|\cdot M^{3(N-1)}\cdot |\Phi_Q|^{4(N-1)^2+3(N-1)+1}\\
&\le&M^{3N}\cdot |\Phi_Q|^{3+4(N-1)^2+3(N-1)+1}\\
&=&M^{3N}\cdot |\Phi_Q|^{4N^2-5N+5}\le M^{3N}\cdot |\Phi_Q|^{4N^2}.
\end{array}
$$

(ii) If $\pi_{PQ}(\alpha)\neq 0$, then $\tilde E_\alpha(R)\le\phi_{PQ}(\tilde E_{\pi_{PQ}(\alpha)}(R))$ by Lemma~\ref{lem:mono}.
Assume that $\pi_{PQ}(\alpha)=0$, and pick any root $a\in\pi_Q^{-1}(\alpha)$. Let
$S\subseteq J$ be the set of simple roots that occur with non-zero coefficients in the
decomposition of $a$. Then $\pi_{PQ}(\alpha)=0$ implies $S\cap J'=\emptyset$. Since $\alpha$ belongs to an irreducible
component of $\Phi_Q$ whose image under $\phi_{PQ}$ is non-zero, there is a simple root
$b\in J'$ adjacent to a root in $S$, or connected to such root by a chain of simple roots in $\Pi\setminus J$.  Consider the reductive subgroup
$G'=U_{\ZZ \pi_Q(S\cup \{b\})}$ of $G$, and its parabolic subgroups
$$
P'=G'\cap P=U_{\NN\pi_Q(b)\cup\ZZ\pi_Q(S)}\quad\mbox{and}\quad Q'=G'\cap Q=U_{\Psi},
$$ where $\Psi=\ZZ\pi_Q(S\cup\{b\})\cap\Phi_Q^+$. Clearly, we can identify
$\Phi_{P'}$ with $\ZZ\pi_P(b)\cap\Phi_P$ and $\Phi_{Q'}$ with $\ZZ \pi_Q(S\cup \{b\})\cap\Phi_Q$, and
by the very Definition~\ref{dfn:St} we obtain a commutative diagram of Steinberg groups
\begin{equation*}
\xymatrix@R=20pt@C=35pt{
St_{P'}(R)\ar[d]^{\phi_{P'Q'}}\ar[r]^{}&St_{P}(R)\ar[d]^{\phi_{PQ}}\\
St_{Q'}(R)\ar[r]^{}&St_Q(R)\\
}
\end{equation*}
where the horizontal arrows send each $\tilde X_\beta(u)$, where $\beta\in \Phi_{P'}$ or, respectively,
$\beta\in\Phi_{Q'}$, to the corresponding element of the right hand side Steinberg group.
Since $\alpha\in\ZZ S\cap\Phi_Q\subseteq\Phi_{Q'}$, by the surjectivity of $\phi_{P'Q'}$ we conclude that
$\tilde E_\alpha(R)\le\phi_{PQ}\bigl(\tilde E_{\pi_P(b)}(R)\bigr)$.
\end{proof}

For any ideal $I\subseteq R$ we set
$$
\tilde E_P(R,I)=\l<\tilde X_{\alpha}(IV_\alpha),\ \alpha\in\Phi_P\r>^{St_P(R)}.
$$

\begin{lem}\label{lem:st-ker}
For any ideal $I\subseteq R$ there is a short exact sequence
$$
\xymatrix{
1\ar[r]&\tilde E_P(R,I)\ar[r]&St_P(R)\ar[r]^{\tilde\rho}& St_P(R/I)\ar[r]& 1,
}
$$
where $\tilde\rho$ is induced by the residue map $\rho:R\to R/I$.
\end{lem}
\begin{proof}
Since the maps $V_\alpha\to V_\alpha\otimes_R R/I$ are surjective, the homomorphism $\tilde\rho$ is surjective.
Clearly, $\tilde E_P(R,I)\le \ker\tilde\rho$,
so $\tilde\rho$ factors through the surjective homomorphism
$$
\bar\rho:St_P(R)/\tilde E_P(R,I)\to St_P(R/I).
$$
The group $St_P(R)/\tilde E_P(R,I)$ is generated by all subgroups
$\tilde U_{(\alpha)}(R)/\bigl(\tilde U_{(\alpha)}(R)\cap \tilde E_P(R,I)\bigr)$, and Lemma~\ref{lem:mono} implies that
such a subgroup is isomorphic to
$$
U_{(\alpha)}(R)/\bigl(U_{(\alpha)}(R)\cap E_P(R,I)\bigr)
=U_{(\alpha)}(R)/U_{(\alpha)}(I)\cong U_{(\alpha)}(R/I).
$$
Clearly, the images
of $\tilde X_\alpha(v)$, $\alpha\in\Phi_P$, $v\in V_\alpha$ in
$St_P(R)/\tilde E_P(R,I)$ satisfy the same relations~\eqref{eq:sum-St} and~\eqref{eq:Chev-St}
as the corresponding elements $\tilde X_\alpha(\rho(v))$ in
$St_P(R/I)$. Therefore, the
canonical map $St_P(R/I)\to E(R/I)$ factors through the group $St_P(R)/\tilde E_P(R,I)$,
producing an inverse for $\bar\rho$.
\end{proof}

Recall that by Lemma~\ref{lem:max-min-roots} for any $\alpha\in\Phi_P$ one has
$$
\Phi_P\cap\ZZ\alpha=\{\pm\alpha,\pm2\alpha,\ldots,\pm m_\alpha\alpha\}.
$$
For any $a\in \tilde E_\alpha(A)$,
$u_i\in V_{i\alpha}$, $1\le i\le m_\alpha$, we set
$$
\tilde Z_\alpha(a,u_1,\ldots,u_{m_\alpha})=a\left(\prod_{i=1}^{m_\alpha} \tilde X_{i\alpha}(u_i)\right)a^{-1}.
$$
The elements $Z_\alpha(a,u_1,\ldots,u_{m_\alpha})\in E_P(R)$ are defined
in the same way.

\begin{lem}\label{lem:EI-gen}
For any ideal $I\subseteq R$ the group $\tilde E_P(R,I)$ is generated by all
$\tilde Z_\alpha(a,u)$, $a\in \tilde E_\alpha(A)$, $u=(u_1,\ldots,u_{m_\alpha})$, $u_i\in IV_{i\alpha}$, $1\le i\le m_\alpha$.
\end{lem}
\begin{proof}
The corresponding statement for $E_P(R,I)\le G(R)$ was proved in~\cite[Lemma 4.3]{St-poly}. Since the proof only used
the relations~\eqref{eq:sum} and~\eqref{eq:Chev}, the same proof goes for $\tilde E_P(R,I)$.
\end{proof}

\subsection{Lifting the Levi subgroup to the Steinberg group}

\begin{dfn}\label{def:tilde-L}
Define the group $\tilde L_P(R)$ to be the subgroup of $St_P(R)$ generated by all elements $\tilde h$ such that
$\tilde h\in\tilde E_\alpha(R)$ for some
$\alpha\in\Phi_P$, and  $\pst_P(\tilde h)\in L_P(R)$.
\end{dfn}

\begin{lem}\label{lem:diag-action}
Assume that $\alpha\in\Phi_P$ belongs to an irreducible component of rank $\ge 2$.
Take $\tilde h\in\tilde L_P(R)$ and set $h=\pst_P(\tilde h)$. Then
for any $v\in V_\alpha$ one has
\begin{equation}\label{eq:h-action}
\tilde h\tilde X_\alpha(v)\tilde h^{-1}=(\pst_P|_{\tilde U_{(\alpha)}})^{-1}(h X_\alpha(v) h^{-1}).
\end{equation}
\end{lem}
\begin{proof}
It is enough to prove the claim for all $\tilde h\in\tilde E_\beta(R)$, $\beta\in\Phi_P$, such that $h=\pst_P(\tilde h)\in L_P(R)$.
Assume first that $\alpha$ and $\beta$ are linearly independent.
By Lemma~\ref{lem:EI-gen} $\tilde h$ as a product of elements of the form
$\tilde Z_{\beta}(b,u_1,\ldots,u_{m_\beta})$, $b\in \tilde E_\beta(R)$, $u_i\in V_{i\beta}$, $1\le i\le m_\beta$.
Then by~\cite[Lemma 4.4]{St-poly}, which was proved for elements in $E_P(R)$ but used only the relations~\eqref{eq:sum-St}
and~\eqref{eq:Chev-St} that hold in $St_P(R)$, one has
$\tilde h\tilde X_\alpha(v)\tilde h^{-1}\in \tilde U_{S}(R)$, where
$S=\ZZ\beta+\NN\alpha\subseteq\ZZ\Phi_P$. Since $S$ is subject to Lemma~\ref{lem:mono}, we have
$\tilde h\tilde X_\alpha(v)\tilde h^{-1}\in\tilde U_{(\alpha)}(R)$, and the claim follows. If $\alpha$ and $\beta$ are linearly dependent,
we reduce to the previous case by Lemma~\ref{lem:lemma11}.
\end{proof}

\begin{lem}\label{lem:st-bigcell}

(i) If for some $x,x'\in \tilde U_P(R)$, $y,y'\in \tilde U_{P^-}(R)$ and $t,t'\in \tilde L_P(R)$
one has $\pst_P(xty)=\pst_P(x't'y')$, then $x=x'$, $y=y'$.

(ii) The natural map $\tilde U_P(R)\times \tilde L_P(R)\times \tilde U_{P^-}(R)\to St_P(R)$ induced by multiplication in
$St_P(R)$ is injective.
\end{lem}
\begin{proof}
The natural map $U_P(R)\times L_P(R)\times U_{P^-}(R)\to G(R)$ is injective, hence $\pst_P(xty)=\pst_P(x't'y')$
implies $\pst_P(x)=\pst_P(x')$ and $\pst_P(y)=\pst_P(y')$. By Lemma~\ref{lem:mono} this implies
$x=x'$ and $y=y'$, which proves $(i)$. The claim $(ii)$ follows since $xty=x't'y'$ implies $\pst_P(xty)=\pst_P(x't'y')$.
\end{proof}

\begin{lem}\label{lem:ker-PQ}
In the setting of Lemma~\ref{lem:phiPQ}, one has $\ker\phi_{PQ}\subseteq\tilde L_P(R)$.
If, moreover, every irreducible component of $\Phi_P$ has rank $\ge 2$, then
$$
\ker\phi_{PQ}\subseteq\Cent(St_P(R)).
$$
\end{lem}
\begin{proof}
Let $\pi_{PQ}:\Phi_Q\to\Phi_P\cup\{0\}$ be the natural surjection. By Lemma~\ref{lem:lemma12} (ii) for any $\alpha\in\Phi_Q$ there is $\alpha'\in\Phi_P$ such that
$\tilde E_\alpha(R)\le\phi_{PQ}(\tilde E_{\alpha'}(R))$. Moreover, if $\pi_{PQ}(\alpha)\neq 0$, by Lemma~\ref{lem:mono}
we can choose $\alpha'=\pi_{PQ}(\alpha)$. For any $u\in V_\alpha$ let
$\hat X_\alpha(u)$ be an element of $\tilde E_{\alpha'}(R)$ such that
$$
\phi_{PQ}(\hat X_\alpha(u))=\tilde X_\alpha(u);
$$
if $\pi_{PQ}(\alpha)\neq 0$, we impose $\hat X_\alpha(u)\in \tilde U_{(\alpha')}(R)$.
Then for any $\beta\in\Phi_P$ the set $\tilde X_\beta(V_\beta)$ is
contained in the subgroup of $St_P(R)$ generated by the elements $\hat X_\alpha(u)$,
$\alpha\in\pi_{PQ}^{-1}(\beta)$, $u\in V_\alpha$. Therefore, the group freely generated by all
$\tilde X_\alpha(u)$, $\alpha\in\Phi_Q$, $u\in V_\alpha$, surjects onto $St_P(R)$ via the map
$\tilde X_\alpha(u)\mapsto\hat X_\alpha(u)$. By the definition of $St_Q(R)$
this implies that $\ker_{PQ}$ is contained in the subgroup of $St_P(R)$ generated by
the elements
\begin{equation*}
\sigma(\alpha,v,w)=\hat X_\alpha(v)\hat X_\alpha(w)\Bigl(\hat X_\alpha(v+w)\prod_{i>1}\hat X_{i\alpha}\l(q^i_\alpha(v,w)\r)\Bigr)^{-1}
\end{equation*}
for all $\alpha\in\Phi_Q,\ v,w\in V_\alpha\otimes_R R'$,
and by the elements
\begin{equation*}
\begin{array}{c}
\tau(\alpha,\beta,u,v)=[\hat X_\alpha(u),\hat X_\beta(v)]\Bigl(\prod\limits_{i,j>0}\hat X_{i\alpha+j\beta}\bigl(N_{\alpha\beta ij}(u,v)\bigr)\Bigr)^{-1}
\end{array}
\end{equation*}
for all $\alpha,\beta\in\Phi_Q$ such that
$m\alpha\neq-k\beta$ for any $m,k>0$ and all $u\in V_\alpha\otimes_R R',\ v\in V_\beta\otimes_R R'$.
We show that all $\sigma(\alpha,v,w)$ and $\tau(\alpha,\beta,u,v)$ belong to $\tilde L_P(R)$.

Since $\pst_P=\pst_Q\circ\phi_{PQ}$, we have $\pst_P(\sigma(\alpha,v,w))=\pst_P(\tau(\alpha,\beta,u,v))=1$
for all $\alpha,\beta,u,v,w$. Then, according to Definition~\ref{def:tilde-L}, it is enough to find $\alpha'\in\Phi_P$ such that
$\sigma(\alpha,v,w)$, or, respectively, $\tau(\alpha,\beta,u,v)$ belongs to $\tilde E_{\alpha'}(R)$.
This holds automatically for all elements $\sigma(\alpha,v,w)$, and for all $\tau(\alpha,\beta,u,v)$ such that $\pi_{PQ}(\alpha)\neq 0$ and
$\pi_{PQ}(\beta)\neq 0$ are collinear elements of $\Phi_P$. If $\pi_{PQ}(\alpha)\neq 0$ and
$\pi_{PQ}(\beta)\neq 0$ are non-collinear elements of $\Phi_P$, then they span a unipotent closed subset of
$\Phi_P$, and by applying Lemma~\ref{lem:mono}, we readily deduce that $\tau(\alpha,\beta,u,v)=1$.

It remains to consider the elements $\tau(\alpha,\beta,u,v)$, where $\alpha\in\Phi_Q$ satisfies
$\pi_{PQ}(\alpha)=0$. In this case $\pst_P(\hat X_{\alpha}(u))=\pst_Q(\tilde X_{\alpha}(u))\in L_P(R)$.
Then $\hat X_\alpha(u)\in \tilde L_P(R)$ by Definition~\ref{def:tilde-L}.
Now if $\pi_{PQ}(\beta)=0$ as well, the same argument shows that all factors in $\tau(\alpha,\beta,u,v)$
belong to $\tilde L_P(R)$, hence $\tau(\alpha,\beta,u,v)\in\tilde L_P(R)$. If $\pi_{PQ}(\beta)\neq 0$,
then we have $\hat X_{i\alpha+j\beta}(V_{i\alpha+j\beta})\subseteq \tilde U_{(\beta')}(R)$ for all
$i,j\in\ZZ$ with $j>0$, where $\beta'=\pi_{PQ}(\beta)$. Let $\alpha'\in\Phi_P$ be such that
$\hat X_\alpha(u)\in \phi_{PQ}(\tilde E_{\alpha'}(R))$. If $\alpha'$ and $\beta'$ are collinear, then
there is $\gamma\in\Phi_P$ such that both $\alpha'$ and $\beta'$ are multiples of $\gamma$, and then
$\tau(\alpha,\beta,u,v)\le\tilde E_\gamma(R)$ as required. If $\alpha'$ and $\beta'$ are non-collinear, then
either they belong to different irreducible components, or they belong to a component of rank $\ge 2$. In the first case
the relation~\eqref{eq:Chev-St} implies that $\tau(\alpha,\beta,u,v)\in\tilde U_{(\beta')}(R)$. In the second case
Lemma~\ref{lem:diag-action} implies $\tau(\alpha,\beta,u,v)=1$.

We have proved that $\ker\phi_{PQ}\subseteq\tilde L_P(R)$. If all irreducible components of $\Phi_P$ have rank $\ge 2$,
then Lemma~\ref{lem:mono} together with Lemma~\ref{lem:diag-action} imply
$\ker\phi_{PQ}\subseteq\Cent(St_P(R))$.
\end{proof}

\subsection{Proof of Theorem~\ref{thm:K2}}

Throughout this subsection, we always assume that $R$ is a local ring, $I\subseteq R$ is its maximal ideal,
$$
\rho:R\to R/I=l
$$
is the residue map, $G$ is
an simply connected simple reductive group over $R$, $P$ is a parabolic $R$-subgroup of $G$,
$L_P\le P$ a Levi subgroup, and all irreducible components of $\Phi_P$ have $\rank\ge 2$.

\begin{lem}\label{lem:E-cell}
One has $\tilde E_P(R,I)\subseteq \tilde U_P(I)\tilde L_P(R)\tilde U_{P^-}(I)$.
\end{lem}
\begin{proof}
Set $Z=\tilde U_P(I)\tilde L_P(R)\tilde U_{P^-}(I)$.
By Lemma~\ref{lem:EI-gen} the group $\tilde E_P(R,I)$ is generated by elements $\tilde Z_\alpha(a,u_1,\ldots,u_{m_\alpha})$,
$\alpha\in\Phi_P$, $a\in\tilde E_\alpha(R)$, $u_i\in IV_{i\alpha}$. Hence it is enough to show that
$\tilde Z_\alpha(a,u_1,\ldots,u_{m_\alpha})Z\subseteq Z$.

Denote by $G_\alpha$ the reductive subgroup $U_{\ZZ\alpha}$ of $G$. Then
$L_P\le G_\alpha$ is the common Levi subgroup of two opposite parabolic subgroups
$L_PU_{(\alpha)}$ and $L_PU_{(-\alpha)}$ of $G_\alpha$, and $U_{(\alpha)}L_PU_{(-\alpha)}$ is an open
subscheme of $G_\alpha$. Then, since
$$\rho\l(\pst_P(\tilde Z_\alpha(a,u_1,\ldots,u_{m_\alpha}))\r)=1,
$$ we have
$$
\pst_P(\tilde Z_\alpha(a,u_1,\ldots,u_{m_\alpha}))\in U_{(\alpha)}(I)L_P(R,I)U_{(-\alpha)}(I).
$$
Therefore, by the definition of $\tilde L_P(R)$, we have
$$
\tilde Z_\alpha(a,u_1,\ldots,u_{m_\alpha})\in \tilde U_{(\alpha)}(I)\tilde L_P(R)\tilde U_{(-\alpha)}(I).
$$

By Lemma~\ref{lem:diag-action} we have $\tilde L_P(R)Z\subseteq Z$. Therefore, in order to establish
the inclusion $\tilde Z_\alpha(a,u_1,\ldots,u_{m_\alpha})Z\subseteq Z$, it is enough to show
that
\begin{equation}\label{eq:comm-cell-claim}
\tilde U_{(-\alpha)}(I)\tilde U_P(I)\subseteq Z
\end{equation}
for any non-divisible relative root $\alpha\in\Phi_P^+$.
By Lemma~\ref{lem:mono} the group $\tilde U_P(I)$ can be written as a product of $\tilde U_{(\beta)}(I)$ with
$\beta$ running over all non-divisible elements of $\Phi_P^+$ in any fixed order.
One shows exactly as above that
\begin{equation}\label{eq:comm-cell-1}
\tilde U_{(-\alpha)}(I)\tilde U_{(\alpha)}(I)\subseteq U_{(\alpha)}(I)\tilde L_P(R) U_{(-\tilde\alpha)}(I).
\end{equation}
Let $\beta\in\Phi_P^+$ be non-collinear to $\alpha$. By the Chevalley commutator formula~\eqref{eq:Chev-St} one has
\begin{equation}\label{eq:comm-cell-2}
\tilde U_{(-\alpha)}(I)\tilde U_{(\beta)}(I)\subseteq
\hspace{-15pt}\prod\limits_{\begin{array}{c}
\vspace{-3pt}\st i\ge 0,\, j>0,\\
\st -i\alpha+j\beta\in\Phi_P^+
\end{array}}\hspace{-15pt}
\tilde U_{(-i\alpha+j\beta)}(I)
\ \cdot
\hspace{-15pt}\prod\limits_{\begin{array}{c}
\vspace{-3pt}\st i>0,\, j\ge 0,\\
\st -i\alpha+j\beta\in\Phi_P^-
\end{array}}\hspace{-15pt}
\tilde U_{(-i\alpha+j\beta)}(I).
\end{equation}
Then formulas~\eqref{eq:comm-cell-1} and~\eqref{eq:comm-cell-2} together imply~\eqref{eq:comm-cell-claim}.
\end{proof}

\begin{lem}\label{lem:rho-L}
The map $\tilde\rho|_{\tilde L_P(R)}:\tilde L_P(R)\to\tilde L_P(l)$ is surjective.
\end{lem}
\begin{proof}
Let $\tilde h\in St_P(l)$ be a standard generator of $\tilde L_P(l)$, i.e.
$\tilde h\in\tilde E_\alpha(l)$, $\alpha\in\Phi_P$, and $h=\pst_P(\tilde h)\in L_P(l)$.
Since $h\in L_P(l)$, we conclude as in the proof of Lemma~\ref{lem:E-cell} that
$$
\pst_P(\tilde\rho^{-1}(\tilde h))\subseteq U_{(\alpha)}(I)L_P(R)U_{(-\alpha)}(I).
$$
Since $\tilde h\in\tilde E_\alpha(l)$, it has a preimage in $\tilde E_\alpha(R)$. Multiplying
this preimage, if necessary, by the corresponding elements in $\tilde U_{(\pm\alpha)}(I)$,
we obtain an element in $\tilde L_P(R)$.
\end{proof}

The following lemma is established using the results of V. V. Deodhar~\cite{Deo}.

\begin{lem}\label{lem:ker-in-cell}
Assume that $g\in St_P(R)$ satisfies $\pst_P(g)=1$. Then $g\in\tilde L_P(R)\tilde E_P(R,I)$.
\end{lem}
\begin{proof}
Let $Q$ be a minimal parabolic subgroup of $G_l$ contained in $P_l$, $L_Q$ a Levi subgroup of $Q$
contained in $(L_P)_l$. The system of relative roots $\Phi_Q$ and the subgroups $U_{(\alpha)}$, $\alpha\in\Phi_Q$,
in the sense of~\cite{PS}, are identified in this case with the relative root system of $G_l$
and respective root subgroups in the sense of~\cite{BorelTits}, and, accordingly, in~\cite{Deo}. This
readily follows by descent from a field extension that splits $G_l$. Since all irreducible
components of $\Phi_P$ have rank $\ge 2$, by Theorem~\ref{th:PS-normality} one has
$$
E_P(l)=E_{P_l}(l)=E_Q(l)=G_l(l)^+.
$$
If $G_l$ is an absolutely almost simple group,
the construction~\cite[1.9]{Deo} introduces a covering group $\widetilde{G_l}$ of $G_l(l)^+=E_Q(l)$, which is
precisely $St_Q(l)$ in our notation. If $G$ is not absolutely almost simple, then it is a direct product of absolutely almost simple groups,
and we define $\widetilde{G_l}$ to be the direct product of covering groups of the factors.
Thus, we obtain the following commutative diagram consisting of surjective group homomorphisms:
\begin{equation}\label{eq:diag-RI}
\xymatrix@R=20pt@C=40pt{
St_P(R)\ar[r]^{\tilde\rho}\ar[d]^{\pst_P}&St_P(R/I)\ar[r]^{\phi_{PQ}}\ar[d]^{\pst_P}&St_Q(R/I)\ar[d]^{\pst_Q}\\
E(R)=E_P(R)\ar[r]^{\rho}&E_P(R/I)\ar@{=}[r]&E_Q(R/I)\\
}
\end{equation}
Here the map $\phi_{PQ}$ is the one constructed in Lemma~\ref{lem:phiPQ}, the map $\tilde\rho$
is the one from Lemma~\ref{lem:st-ker}.

Deodhar shows in~\cite[Prop. 1.16]{Deo} that $St_Q(R/I)=St_Q(l)$ is a central
extension of $E_Q(R/I)=E_Q(l)$, and moreover
$\ker\pst_Q\subseteq \tilde H$, where $\tilde H$ is a certain subgroup of $St_Q(l)$. Note that
this result does not use the assumption $|l|\ge 16$ present in~\cite[Theorem 1.9]{Deo}; that assumption was required to prove that the central extension
is universal. By definition, the subgroup $\tilde H$ is generated by some elements $\tilde h_\alpha(u,u')$,
where $\alpha\in\Phi_Q$, $1\neq u,u'\in \tilde U_{(\alpha)}(l)$, such that
$\tilde h_\alpha(u,u')\in \tilde U_{(\alpha)}(l)$ and $\pst_Q(\tilde h(u,u'))\in L_Q(l)$.
This implies that $\ker\pst_Q\subseteq \tilde L_Q(l)$.

Let $g\in St_P(R)$ be such that $\pst_P(g)=1$. Then $\phi_{PQ}(\tilde\rho(g))\in\ker\pst_Q\subseteq\tilde L_Q(l)$.
By Lemma~\ref{lem:ker-PQ} this implies that $\tilde\rho(g)\in\tilde L_P(R)$.
Since $\ker\tilde\rho=\tilde E_P(R,I)$ by Lemma~\ref{lem:st-ker}, by Lemma~\ref{lem:rho-L} we have
$g\in\tilde L_P(R)\tilde E_P(R,I)$.
\end{proof}

\begin{proof}[Proof of Theorem~\ref{thm:K2}]
By Lemma~\ref{lem:ker-in-cell} combined with Lemmas~\ref{lem:diag-action} and~\ref{lem:E-cell} we have
$$
\ker(\pst_P:St_P(R)\to E(R))\subseteq \tilde U_P(R)\tilde L_P(R)\tilde U_{P^-}(R).
$$
By Lemma~\ref{lem:st-bigcell} this implies that actually
$$
\ker(\pst_P:St_P(R)\to E(R))\subseteq\tilde L_P(R).
$$
Applying Lemma~\ref{lem:diag-action}, we conclude that
$\ker(\pst_P:St_P(R)\to E(R))$ is contained in the the center of $St_P(R)$.
\end{proof}

\section{Intersection of normal subgroups with relative root subschemes}\label{sec:normal}

\subsection{Root chains in a root system}

Let $\Psi$ be an irreducible, but not necessarily reduced root system in the sense of~\cite{Bu} with an
inner product $(\ ,\ )$. 
We denote by $\NN$ the set of positive integers, and by $\NN_0$ the set of non-negative integers.
For any subsets $S_1,S_2\subseteq\Psi$ we set
\begin{equation*}
\begin{array}{r|l}
[S_1,S_2]=\Psi\,\cap\,\Bigl\{\sum_{i=1}^nc_i\alpha_i+\sum_{j=1}^md_j\beta_j\ &\
\parbox{7.3cm}{$\textstyle m,n\in\NN,\
\alpha_i\in S_1,\ \beta_j\in S_2,\ c_i,d_j\in\NN_0,\\
\sum_{i=1}^nc_i> 0,\ \sum_{j=1}^md_j>0$}\Bigr\}.
\end{array}
\end{equation*}
Note that for any $S\subseteq\Psi$ and any $\alpha\in\Psi$ one has
$$
[S,\alpha]=[S,(\alpha)],
$$
where $(\alpha)=\NN\alpha\cap\Psi$.

\begin{dfn}\label{dfn:good}
Let $\delta,\gamma\in\Psi$ be two roots, $2\gamma\not\in\Psi$.
For a sequence of roots $\beta_1,\ldots,\beta_n\in\Psi$, set
$$
D_0=\{\delta\},\quad D_i=[\ldots[[\delta,\beta_1],\beta_2],\ldots,\beta_i],\quad 1\le i\le n.
$$
The sequence $\beta_1,\ldots,\beta_n$ is called a \emph{good chain between $\delta$ and $\gamma$},
if
\begin{itemize}
\item $\delta_i=\delta+\beta_1+\ldots+\beta_i$ belongs to $\Psi$ for any $1\le i\le n$;
\item $D_n=\{\delta_n\}=\{\gamma\}$;
\item there are indices $0=i_0< i_1< i_2<\ldots < i_k=n$ such that for every $0\le j\le k-1$
 the additive closure of the set $D_{i_j}\cup\{\beta_{i_j+1},\beta_{i_j+2},\ldots,\beta_{i_{j+1}}\}$
inside $\Psi$ is a unipotent closed set of roots, and
\begin{equation}\label{eq:Dij}
\bigl[D_{i_j}\setminus\{\delta_{i_j}\},\
\beta_{i_j+1}\cup [\delta_{i_j},\beta_{i_j+1}]\bigr]=\emptyset.
\end{equation}
In particular, one has $D_{i_j+1}=[D_{i_j},\beta_{i_j+1}]=[\delta_{i_j},\beta_{i_j+1}]$.
\end{itemize}
\end{dfn}

Note that good chains can be concatenated: if
$\beta_1,\ldots,\beta_n$ is a good chain between
$\delta$ and $\gamma$, and $\alpha_1,\ldots,\alpha_m$ is a good chain between
$\gamma$ and $\epsilon$, then
$\beta_1,\ldots,\beta_n,\alpha_1,\ldots,\alpha_m$ is a good chain between $\delta$ and $\epsilon$.

\begin{lem}\label{lem:positive-good}
Let $\Psi$ be an irreducible root system with a system of simple roots $\Pi$, let
$\Psi^+$ be the set
of positive roots with respect to $\Pi$, and $\tilde\alpha\in\Psi^+$ be the unique root
of maximal height with respect to $\Pi$. For any root $\delta\in\Psi^+$ there is a good chain
$\beta_1,\ldots,\beta_n\in\Psi^+$ between $\delta$ and $\tilde\alpha$.
\end{lem}
\begin{proof}
It is well-known that for any positive
root $\delta$ there are simple roots $\beta_1,\ldots,\beta_n\in\Pi$ such that $\delta+\beta_1+\ldots+\beta_i\in\Psi$
for any $1\le i\le n$, and $\delta+\beta_1+\ldots+\beta_n=\tilde\alpha$. Then $\beta_1,\ldots,\beta_n$
is a good chain between $\alpha$ and $\tilde\alpha$ with $k=1$, $i_k=i_1=n$. Indeed,
the additive closure of $D_{i_0}=\{\delta\}$ and $\beta_1,\ldots,\beta_n$
is unipotent, since all these roots are positive, and
$$
D_n=[\ldots[[\delta,\beta_1],\beta_2],\ldots,\beta_n]=\{\tilde\alpha\},
$$
since $\tilde\alpha$ is the root of maximal height with respect to $\Pi$.
\end{proof}

\begin{lem}\label{lem:medium-good}
Let $\Psi$ be a root system of type $BC_l$, $l\ge 2$, and let $\gamma\in\Psi$ be a root of medium length. There exists
a root of maximal length $\delta\in\Psi$ such that there is a good chain between $\delta$ and $\gamma$.
\end{lem}
\begin{proof}
Let $\alpha\in\Psi$ be a root of minimal length (i.e. extra-short), such that $\alpha$ and $\gamma$
consitute a base of simple roots in a root subsystem of type $BC_2$ inside $\Psi$. Take $\delta=2\alpha$.
We claim that $(\beta_1,\beta_2,\beta_3)=(\gamma,-\alpha,-\alpha)$ is a good chain between $\delta$ and $\gamma$ with $k=2$,
$i_1=1$, $i_2=n=3$. Indeed, one has
$$
\begin{array}{l}
{}D_{i_1}=D_1=[2\alpha,\gamma]=\{2\alpha+\gamma,2\alpha+2\gamma\};\\
{}[D_1,\beta_2]=[\{2\alpha+\gamma,2\alpha+2\gamma\},-\alpha]=\{\alpha+\gamma,\gamma,2(\alpha+\gamma)\};\\
{}[D_1\setminus\{\delta+\beta_1\},\beta_2]=[2\alpha+2\gamma,-\alpha]=\emptyset;\\
{}\bigl[D_1\setminus\{\delta+\beta_1\},[\delta+\beta_1,\beta_2]\bigr]=
  \bigl[2\alpha+2\gamma,\{\alpha+\gamma,\gamma,2(\alpha+\gamma)\}\bigr]=\emptyset;\\
{}D_{i_2}=D_3=[\{\alpha+\gamma,\gamma,2(\alpha+\gamma)\},-\alpha]=\{\gamma\}.\\
\end{array}
$$
Note that the roots of $D_{i_0}=D_0=\{2\alpha\}$ and $\beta_1=\gamma$ are positive with respect to the system of simple
roots $\{\alpha,\gamma\}$ of $B_2$, and the roots of $D_{i_1}=D_1$ and $\beta_2=\beta_3=-\alpha$
are positive with respect to the system $\{-\alpha,2\alpha+\gamma\}$. Hence the corresponding additive closures are unipotent.
\end{proof}

\begin{lem}\label{lem:max-good}
Let $\Phi$ be a reduced irreducible root system with a set of simple roots $\Pi$ and Dynkin diagram $D$. Let $J\subseteq\Pi$ and
$\Gamma\le\Aut(D)$ be such that $\Psi=\Phi_{J,\Gamma}$ is isomorphic as a set with partially defined addition
to an irreducible root system, and $\rank\Psi\ge 2$. Let $\tilde a$ and $\tilde\alpha$ denote the roots of maximal
height in $\Phi$ and $\Psi$ respectively.
Let $\sigma\in\pi_{J,\Gamma}(\Pi)\setminus\{0\}$ be a simple root of $\Psi$ such that
$\{-\tilde\alpha,\sigma\}$ is a base of a root subsystem of $\Psi$ of type \emph{(a)} $A_2$ or \emph{(b)} $B_2$.

(i) Depending on whether \emph{(a)} or \emph{(b)} takes place, let
$\beta_1,\ldots,\beta_n$ be the sequence $-\sigma,-\tilde\alpha,-\tilde\alpha+\sigma$ or,
respectively, $-\sigma,-\sigma,-\tilde\alpha+\sigma,-\tilde\alpha+\sigma$.
Then $\beta_1,\ldots,\beta_n$ is a good chain between
$\tilde\alpha$ and $-\tilde\alpha$.

(ii) There are roots
$a_i\in\pi_{J,\Gamma}^{-1}(\beta_i)$, $1\le i\le n$, such that $\tilde a+a_1+\ldots+a_i$ is a root for any
$1\le i\le n$, and $\tilde a+a_1+\ldots+a_n=-\tilde a$.
\end{lem}
\begin{proof}
Let $\tilde D$ be the extended Dynkin digram of $\Phi$. Let $s_0\in J$ be a simple root that is connected to $-\tilde a$
in the graph $\tilde D$ by a (possibly, empty) chain of simple roots from $D\setminus J$. Let $s$ be the sum of $s_0$ and all simple roots
in this chain.
Clearly, $s\in\Phi$ and $(\tilde a,-s)>0$, hence $\tilde a-s\in\Phi$.
Since $|J|\ge 2$, not all simple roots of $\Phi$ are involved in $s$, and therefore $\tilde a-s\in\Phi^+$.
Set $\sigma=\pi_{J,\Gamma}(s)=\pi_{J,\Gamma}(s_0)$. Since $\pi_{J,\Gamma}(\tilde a-s)=\tilde\alpha-\sigma$,
we have $\tilde\alpha-\sigma\in\Psi^+$.
Since $\sigma$ is a simple root in $\Psi$,  one readily sees that
$\tilde\alpha$ and $-\sigma$ generate a root subsystem of type $A_2$ or $B_2$.
Conversely, since $\tilde a\in\pi_{J,\Gamma}^{-1}(\tilde\alpha)$, any simple root $\sigma$ with this property can be constructed
as above.

Case (a). Assume that $\tilde\alpha$ and $-\sigma$ generate a root subsystem of $\Psi$ of type $A_2$.
Then
$$
(\beta_1,\beta_2,\beta_3)=(-\sigma,-\tilde\alpha,-\tilde\alpha+\sigma)
$$
is a good chain between $\tilde\alpha$ and $-\tilde\alpha$ with $k=3$, $i_1=1$, $i_2=2$, $i_3=n=3$. Indeed,
one has
$$
\begin{array}{l}
D_1=[\tilde\alpha,-\sigma]=\{\tilde\alpha-\sigma\};\\
D_2=[\tilde\alpha-\sigma,-\tilde\alpha]=\{-\sigma\};\\
D_3=[-\sigma,-\tilde\alpha+\sigma]=\{-\tilde\alpha\}.
\end{array}
$$
One readily sees that the additive closures of $D_0\cup\{\beta_1\}=\{\tilde\alpha,-\sigma\}$,
$D_1\cup\{\beta_2\}$, and $D_2\cup\{\beta_3\}$ are unipotent sets, namely, the corresponding 3-element
subsets of $A_2$.

Let $c$ be a $\Pi$-minimal root in $\pi_{J,\Gamma}^{-1}(-\sigma)$ that exists by Lemma~\ref{lem:max-min-roots}.
If $S\subseteq D\setminus\{J\setminus s_0\}$ is the connected component containing $s_0$, then
$c$ is a negative linear combination of some simple roots in $S$. Since $-\tilde a$ is adjacent to
a root of $S$ in $\tilde D$, we have $(-\tilde a,c)>0$, or, equivalently, $(\tilde a,c)<0$. Hence
$\tilde a+c\in\pi_{J,\Gamma}^{-1}(\tilde\alpha-\sigma)=\pi_{J,\Gamma}^{-1}(\tilde\alpha+\beta_1)$ is a root, choose $a_1=c$ and
$a_2=-\tilde a$, so that
$\tilde a+a_1+a_2=c$, and then $a_3=-\tilde a-c$.

Case (b). Now assume that $\tilde\alpha$ and $-\sigma$ generate a root subsystem of $\Psi$ of type $B_2$, and
set
$$
(\beta_1,\beta_2,\beta_3,\beta_4)=(-\sigma,-\sigma,-\tilde\alpha+\sigma,-\tilde\alpha+\sigma).
$$
This is a good chain between $\delta=\tilde\alpha$ and $\gamma=-\tilde\alpha$ with $k=3$, $i_1=2$, $i_2=3$, $i_3=n=4$.
Indeed, we have
\begin{equation}\label{eq:spchain(b)}
\begin{array}{l}
{}D_1=[D_0,\beta_1]=[\tilde\alpha,-\sigma]=\{\tilde\alpha-\sigma,\tilde\alpha-2\sigma\};\\
{}D_{i_1}=D_2=[D_1,\beta_2]=[[\tilde\alpha,-\sigma],-\sigma]=\{\tilde\alpha-2\sigma\};\\
{}D_{i_2}=D_3=[D_2,\beta_3]=[\tilde\alpha-2\sigma,-\tilde \alpha+\sigma]=\{-\sigma,-\tilde\alpha\};\\
{}D_n=D_4=[\{-\sigma,-\tilde\alpha\},-\tilde\alpha+\sigma]=[-\sigma,-\tilde\alpha+\sigma]=\{-\tilde\alpha\};\\
{}[D_{i_2}\setminus\{\delta+\beta_1+\beta_2+\beta_3\},\beta_4]=[-\tilde\alpha,-\tilde\alpha+\sigma]=\emptyset;\\
{}\bigl[D_{i_2}\setminus\{\delta+\beta_1+\beta_2+\beta_3\},[\delta+\beta_1+\beta_2+\beta_3,\beta_4]\bigr]
  =[-\tilde\alpha,[-\sigma,-\tilde\alpha+\sigma]]=\emptyset.\\
\end{array}
\end{equation}
The additive closures of $D_0\cup\{\beta_1,\beta_2\}=\{\tilde\alpha,-\sigma\}$ and
$D_2\cup\{\beta_3\}=\{\tilde\alpha-2\sigma,-\tilde\alpha+\sigma\}$ are unipotent, since these sets
are two systems of simple roots in $B_2$. The additive closure of
$D_3\cup\beta_4=\{-\sigma,-\tilde\alpha,-\tilde\alpha+\sigma\}$ is unipotent, since this set is already additively
closed.

Let $S\subseteq\tilde D\setminus\{J\setminus s_0\}$
be the connected component containing $-\tilde a$ and $s_0$, and let $\Delta$ be the root subsystem
of $\Phi$ generated by $S$. Let
$c\in\pi_{J,\Gamma}^{-1}(\tilde\alpha-2\sigma)$ be the root of maximal height in $\Delta$ with respect
to $-S$.
Choose $b\in\Phi^+$ to be the sum of all roots in $-S$, then $\pi_{J,\Gamma}(b)=\pi_{J,\Gamma}(\tilde a-s_0)=
\tilde\alpha-\sigma$.
We set $a_1=b-\tilde a$, $a_2=c-b$,
$a_3=-b$, $a_4=-\tilde a+b-c$. It remains show that all $a_i$ are roots.

Note that the root system $\Phi$ is not of type $A_l$,
because it is only possible in case (a). Then $-\tilde a$ is a leaf of $\tilde D$ and $S$.
Then $b-\tilde a$ is a root.
Considering the extended Dynkin diagram of $\Delta$, we conclude that $(c,b)>0$, and hence $c-b\in\Phi$.

Finally, $-a_4=c-(b-\tilde a)$ is also a root. Indeed, since $c-(b-\tilde a)-\tilde a=c-b$ is a root,
by~\cite[Lemma 1]{PS} we have that if $c-(b-\tilde a)\not\in\Phi$, then necessarily $c-\tilde a\in\Phi$.
Both $\tilde a$ and $c$ are long roots in $\Phi$, hence
if $c-\tilde a\in\Phi$, we have $(\tilde a,c)=1/2$, assuming both $\tilde a$ and $c$ have length 1.
However, $(\tilde a,c)=(c-x,c)=1-(x,c)$, where $x$ is a
positive linear combination of roots in $-S\setminus\{\tilde a\}$. Therefore, $(c,d)>0$ for at least one
root $d$ from $-S\setminus\{\tilde a\}$, which implies $(c,b-\tilde a)>0$ and $c-(b-\tilde a)\in\Phi$.
\end{proof}

\begin{lem}\label{lem:long-good}
Let $\Psi$ be an irreducible root system, $\rank\Psi\ge 2$. Let $\gamma\in\Psi$ be a root of maximal length, or a root
of medium length if $\Psi$ is of type $BC_l$. For any root $\delta\in\Psi$ there is a good chain between $\delta$
and $\gamma$.
\end{lem}
\begin{proof}
Assume first that $\gamma$ is a root of maximal length. Then
there is a choice of a system of simple roots $\Pi\subseteq\Psi$
such that $\gamma=\tilde\alpha$ is the root of maximal height with respect to $\Pi$.
If $\delta$ is positive with respect to $\Pi$, the claim follows from Lemma~\ref{lem:positive-good}.
If $\delta$ is negative with respect to $\Pi$, Lemma~\ref{lem:positive-good} implies that there is a good chain
$\beta_1,\ldots,\beta_n$ between $\delta$ and $-\tilde\alpha$. By Lemma~\ref{lem:max-good} there is a good chain
between $-\tilde\alpha$ and $\tilde\alpha$. The concatenation of $\beta_1,\ldots,\beta_n$ and the latter
chain is a good chain between $\delta$ and $\tilde\alpha$.

Now assume that $\Psi=BC_l$, $l\ge 2$, and $\gamma$ is a root of medium length. By Lemma~\ref{lem:medium-good}
there is a root of maximal length $\alpha\in\Psi$ and a good chain between $\alpha$ and $\gamma$; we concatenate this chain
with a good chain between $\delta$ and $\alpha$ that exists by the above.
\end{proof}

\subsection{Root module homomorphisms over a local ring}\label{sec:local-Erepr}

Throughout this subsection we assume that $G$ is a reductive group
over a local commutative ring $R$. Assume that the structure constants
of the absolute root system $\Phi$ of $G$ are invertible in $R$.
Let $Q$ is a minimal parabolic subgroup of $G$ over $R$, and let $L_Q$ be a Levi subgroup of $Q$. Since $R$ is local,
it is connected, and hence $G$, $Q$ and $L_Q$ are subject to Lemmas~\ref{lem:relroots} and~\ref{lem:relschemes}.
We denote by $\Phi_Q=\Phi(S,G)$ the corresponding system of relative roots, and by
$$
\pi_Q=\pi_{J,\Gamma}:\ZZ\Phi\to\ZZ\Phi_Q
$$
the corresponding surjection. Observe that in this context $\Phi_Q$ is the root system of $G$ with respect to $S$
in the sense of~\cite[Exp. XXVI, \S 7]{SGA3}. In particular, $\Phi_Q$ is a root system in the sense of~\cite{Bu}.

\begin{dfn}\label{def:N-map}
For any $\alpha_1,\ldots,\alpha_m\in\Phi_Q$ define the multilinear map
$$
N_{\alpha_1,\ldots,\alpha_m}:V_{\alpha_1}\times V_{\alpha_2}\times\ldots\times V_{\alpha_m}\to V_{\alpha_1+\ldots+\alpha_m}
$$
by the following inductive formula:
$N_{\alpha_1}=\id$, and for any
$2\le i\le m$
$$
N_{\alpha_1,\ldots,\alpha_i}(v_1,\ldots,v_i)=N_{\alpha_1+\ldots+\alpha_{i-1},\alpha_i,1,1}\bigl(
N_{\alpha_1,\ldots,\alpha_{i-1}}(v_1,\ldots,v_{i-1}),v_i\bigr),
$$
where $N_{\alpha,\beta,1,1}:V_\alpha\times V_\beta\to V_{\alpha+\beta}$, $\alpha,\beta\in\Phi_Q$,
is the bilinear map of~\eqref{eq:Chev}.
\end{dfn}

\begin{lem}\label{lem:chain-comm}
Let $\delta,\gamma\in\Phi_Q$ be two roots such that $2\gamma\not\in\Phi_Q$, and let $\beta_1,\ldots,\beta_n$ be a good chain between $\delta$ and $\gamma$.
Then for any $(u,v_1,\ldots,v_n)\in V_\delta\times V_{\beta_1}\times\ldots\times V_{\beta_n}$
one has
$$
\bigl[\ldots\bigl[\bigl[X_{\delta}(u),X_{\beta_1}(v_1)\bigr],
X_{\beta_2}(v_2)\bigr],
\ldots,X_{\beta_n}(v_n)\bigr]=X_{\gamma}\bigl(N_{\delta,\beta_1,\ldots,\beta_n}(u,v_1,\ldots,v_n)\bigr).
$$
\end{lem}
\begin{proof}
Note that if $A_1,A_2\subseteq\Phi_Q$ are two unipotent closed sets of roots both
contained in a larger unipotent closed set, then $[A_1,A_2]$
is also a unipotent closed set, and the Chevalley
commutator formula~\eqref{eq:Chev} implies that
\begin{equation}\label{eq:commset}
[U_{A_1},U_{A_2}]\subseteq U_{[A_1,A_2]}.
\end{equation}
We will use the notation of Definition~\ref{dfn:good} for the chain $\beta_1,\ldots,\beta_n$. In particular,
set $\delta_0=\delta$, $\delta_n=\gamma$.
Fix an index $j$, $0\le j\le k-1$. The additive closure of $D_{i_j}$,
$\beta_{i_j+1},\ldots,\beta_{i_{j+1}}$ is a unipotent set of roots. Therefore, there is a system of simple
roots $\Pi_j$ in $\Phi_Q$ such that all these roots are positive with respect to $\Pi_j$. An additively closed
subset of a unipotent set is also unipotent, hence $D_{i_j}$ is a unipotent closed set, and
for any $i_j< p\le i_{j+1}$ the set
$$
D_p=[\ldots[D_{i_j},\beta_{i_j+1}],\ldots,\beta_p]=[\ldots[D_{i_j},\beta_{i_j+1}],\ldots,\beta_p]=
[\ldots[\delta_{i_j},\beta_{i_j+1}],\ldots,\beta_p]
$$
is also a unipotent closed set of roots.
Note that the root $\delta_p=\delta_{i_j}+\beta_{i_j+1}+\ldots+\beta_p$ is a root of minimal height
with respect to $\Pi_j$ in the set $D_p$, since height is
consistent with addition of roots. In particular, $D_p\setminus\delta_p$ is additively closed, and
$$
[D_{p-1}\setminus\delta_{p-1},\beta_p]\subseteq D_p\setminus\{\delta_{p-1}+\beta_p\}=D_p\setminus\{\delta_p\}.
$$

Any element $x\in U_{D_{i_j}}$ satisfies
$x\in X_{\delta_{i_j}}(v)\cdot U_{D_{i_j}\setminus\{\delta_{i_j}\}}(R)$,
where $v\in V_{\delta_{i_j}}$.
By the equality~\eqref{eq:Dij} of Definition~\ref{dfn:good} and the generalized Chevalley commutator formula~\eqref{eq:Chev}
we have
$$
\begin{array}{rcl}
[x,X_{\beta_{i_j+1}}(v_{i_j+1})]&=&[X_{\delta_{i_j}}(v),X_{\beta_{i_j+1}}(v_{i_j+1}))]\\
&\in& X_{\delta_{i_j}+\beta_{i_j+1}}\bigl(N_{\delta_{i_j},\beta_{i_j+1},1,1}(v,v_{i_j+1})\bigr)\cdot U_{D_{i_j+1}
\setminus\{\delta_{i_j}+\beta_{i_j+1}\}}(R).
\end{array}
$$
Then, applying~\eqref{eq:commset}, we deduce that for all $i_j< p\le i_{j+1}$
\begin{equation*}
[\ldots[x,X_{\beta_{i_j+1}}(v_{i_j+1})],\ldots],X_{\beta_p}(v_p)]
\in
X_{\delta_p}\bigl(N_{\delta_{i_j},\beta_{i_j+1},\ldots,\beta_p}(v,v_{i_j+1},\ldots,v_p)\bigr)\cdot
U_{D_p\setminus\{\delta_p\}}(R).
\end{equation*}
Combining these inclusions for all $0\le j\le k-1$ and $p=i_{j+1}$, and using
the equality $D_{i_k}=D_n=\{\delta_n\}=\{\gamma\}$,
we immediately obtain the claim of the lemma.
\end{proof}

\begin{dfn}
For any $\alpha,\beta\in\Phi_Q$, we say that an $R$-linear homomorphism of $R$-modules $\phi:V_\alpha\to V_\beta$
is \emph{$E$-representable}, if there exist a good chain $\beta_1,\ldots,\beta_n$ between $\alpha$ and $\beta$
in the sense of Definition~\ref{dfn:good},
and elements $v_i\in V_{\beta_i}$,
$1\le i\le n$, such that
$$
\phi(-)=N_{\alpha,\beta_1,\ldots,\beta_n}(-,v_1,\ldots,v_n):V_\alpha\to V_\beta.
$$
\end{dfn}

Note that a composition of $E$-representable homomorphisms is also $E$-representable, since a concatenation of
good chains is a good chain.

\begin{lem}\label{lem:E-repr-long}
Let $\gamma$ be a root of maximal length in an irreducible component of $\Phi_Q$ of rank $\ge 2$. Then any $R$-linear endomorphism
of the $R$-module $V_\gamma$ is a finite sum of $E$-representable $R$-linear endomorphisms.
\end{lem}
\begin{proof}
First we prove the claim for the case where $\gamma=\tilde\alpha$ is a  positive root of maximal height in an irreducible
component of $\Phi_Q$.
Let $\alpha_1,\ldots,\alpha_n\in\Phi_Q$
be the good chain from $\tilde\alpha$ to $-\tilde\alpha$ constructed in Lemma~\ref{lem:max-good}.
Then, clearly, $-\alpha_1,\ldots,-\alpha_n$ is a good chain between $-\tilde\alpha$ and $\tilde\alpha$,
and $\alpha_1,\ldots,\alpha_n,-\alpha_1,\ldots,-\alpha_n$ is a good chain between $\tilde\alpha$
and $\tilde\alpha$. Let
$$
N_{\tilde\alpha,\alpha_1,\ldots,\alpha_n,-\alpha_1,\ldots,-\alpha_n}:
V_{\tilde\alpha}\times\prod_{i=1}^n V_{\alpha_i}\times\prod_{i=1}^n V_{-\alpha_i}\to V_{\tilde\alpha}
$$
be the multilinear map of Definition~\ref{def:N-map}.
It induces an $R$-linear map
$$
F:\bigotimes_{i=1}^n V_{\alpha_i}\otimes\bigotimes_{i=1}^n V_{-\alpha_i}\to \End_R(V_{\tilde\alpha}).
$$
The images of elementary tensors under $F$ are $E$-representable by definition.
We are going to show that
$F$ is surjective. In order to do that, by faithfully flat descent we can assume that $G$ is split
and $L_Q$ contains a split maximal subtorus as in~Lemma~\ref{lem:relschemes}, although $R$ is no longer a local ring.
Let $x_a(\xi)$, $a\in\Phi$, $\xi\in R$, denote the elementary root unipotents in $G(R)$.
For any $\alpha\in\Phi_Q$ and any $a\in\pi_Q^{-1}(\alpha)$ by Lemma~\ref{lem:relschemes} there is a unique $v_a\in V_\alpha$
such that
$$
x_a(1)\in X_\alpha(v_a)X_{2\alpha}(V_{2\alpha}).
$$
Here, as usual, we assume that $X_{2\alpha}(V_{2\alpha})=1$ if $2\alpha\not\in\Phi_Q$.
Moreover, for all $\alpha,\beta\in\Phi_Q$ such that $n\alpha\neq-m\beta$ for all $n,m\in\NN$,
combining the usual Chevalley commutator formula with the definition of $N_{\alpha\beta 11}$ we conclude that
\begin{equation}\label{eq:vavb}
N_{\alpha\beta 11}(v_a,v_b)=N_{ab}\cdot v_{a+b},
\end{equation}
where $N_{ab}\in R^\times$ is the corresponding structure constant of $\Phi$.

By Lemma~\ref{lem:max-good} (ii) combined with Lemma~\ref{lem:beta-correct}, for any $a\in\pi_Q^{-1}(\tilde\alpha)$ there
is a sequence of roots $a_i\in\pi_Q^{-1}(\alpha_i)$ such that all sums $a+a_1+\ldots+a_i$ are roots
and $a+a_1+\ldots+a_n=-\tilde a$.
Fix a pair of roots $a,a'\in\pi_Q^{-1}(\tilde\alpha)$ and consider the corresponding
sequences of roots $a_1,\ldots,a_n$ and $a'_1,\ldots,a'_n$.
By~\eqref{eq:vavb} we conclude that
$$
N_{\tilde\alpha,\alpha_1,\ldots,\alpha_n,-\alpha_1,\ldots,-\alpha_n}
\bigl(v_a,v_{a_1},\ldots,v_{a_n},v_{-a'_1},\ldots,v_{-a'_n}\bigr)=n(a,a')\cdot v_{a'},
$$
where $n(a,a')\in R$ is a product of structure constants $\pm 1,\pm 2,\pm 3$ of $\Phi$. By our assumption
the structure constants are invertible, hence $n(a,a')\in R^\times$.

Note that for any $b\in\pi_Q^{-1}(\tilde\alpha)$ such that
$\height(b)\le\height(a)$, we have $\height(b+a_1+\ldots+a_n)\le\height(-\tilde a)$, which implies
that either $b=a$, or $b+a_1+\ldots+a_n$ is not a root. Hence for any $b\neq a$ satisfying $\height(b)\le\height(a)$,
we have
$$
N_{\tilde\alpha,\alpha_1,\ldots,\alpha_n,-\alpha_1,\ldots,-\alpha_n}
\bigl(v_b,v_{a_1},\ldots,v_{a_n},v_{-a'_1},\ldots,v_{-a'_n}\bigr)=0.
$$
On the other hand, for any $b\in\pi_Q^{-1}(\tilde\alpha)$ such that $\height(b)>\height(a)$, if
$$
b+a_1+\ldots+a_n-a'_1-\ldots-a'_n=b-a+a'
$$ is a root, then
$$
N_{\tilde\alpha,\alpha_1,\ldots,\alpha_n,-\alpha_1,\ldots,-\alpha_n}
\bigl(v_b,v_{a_1},\ldots,v_{a_n},v_{-a'_1},\ldots,v_{-a'_n}\bigr)=m(b,a,a')\cdot v_{b-a+a'}
$$
for some $m(b,a,a')\in R^\times$, and, clearly, $\height(b-a+a')>\height(a')$.

The vectors $v_a$, $a\in\pi_Q^{-1}(\tilde\alpha)$, constitute  a basis of $V_{\tilde\alpha}$. Let
$e_1,\ldots,e_k$, $k=|\pi_Q^{-1}(\tilde\alpha)|$, be the list of these vectors in any order
coherent with the height of roots. Observations in the previous paragraph imply that
the matrix of $F\bigl(v_{a_1},\ldots,v_{a_n},v_{-a'_1},\ldots,v_{-a'_n}\bigr)$ in this basis has an invertible entry
in the position $(i_0,j_0)$, where $e_{i_0}=v_{a'}$ and $e_{j_0}=v_a$, while all other non-zero entries of
this matrix
are in positions $(i,j)$ with $j>j_0$ and $i>i_0$. This readily implies that the $R$-module
$F(V_{\alpha_1},\ldots,V_{\alpha_n},V_{-\alpha_1},\ldots,V_{-\alpha_n})$ contains all
elementary endomorphisms of $V_{\tilde\alpha}$ whose matrices have just one entry equal to 1, and all other entries equal to $0$.
Hence $F(V_{\alpha_1},\ldots,V_{\alpha_n},V_{-\alpha_1},\ldots,V_{-\alpha_n})=\End_R(V_{\tilde\alpha})$.

It remains to observe that the above proof holds verbatim for an arbitrary root $\gamma\in\Phi_Q$ of
maximal length lying in the same irreducible component as $\tilde\alpha$.
Indeed, the root $\gamma$ is mapped to $\tilde\alpha$ by an element $w$ of the Weyl group of $\Phi_Q$.
Since $Q$ is a minimal parabolic subgroup of $G$,
any element of this Weyl group is induced by an element $w'$ of the Weyl group of $\Phi$ via the
projection $\pi_Q$ by~\cite[Th\'eor\`eme 7.13]{SGA3}. Then $w$ maps the good chain between $\tilde\alpha$ and $-\tilde\alpha$
employed above to a good chain between $\gamma$ and $-\gamma$, and, thanks to the existence of $w'$,
the latter chain would satisfy all the same properties stated in Lemma~\ref{lem:max-good}.
\end{proof}

\begin{lem}\label{lem:E-repr-alltomedlong}
Let $\delta,\gamma\in\Phi_Q$ be two roots in the same irreducible component of rank $\ge 2$, and let $\phi:V_\delta\to V_\gamma$ be any homomorphism of $R$-modules.

(i) For any good chain between $\delta$ and $\gamma$,
$\phi$ is a finite sum
of $R$-linear homomorphisms that are compositions of an $E$-representable homomorphism represented by this chain,
and an $R$-linear endomorphism of $V_\gamma$.

(ii) For any good chain between $\delta$ and $\gamma$,
$\phi$ is a finite sum
of $R$-linear homomorphisms that are compositions of an $R$-linear endomorphism of $V_\delta$
and an $E$-representable homomorphism represented by this chain.

(iii) Assume that $\gamma$ is a root of maximal length in $\Phi_Q$, or a root of medium length if
it belongs to an irreducible component of $\Phi_Q$ of type $BC_l$.
Then $\phi$ is a finite sum of $E$-representable
homomorphisms.
\end{lem}
\begin{proof}
All three statements can be rephrased in terms of surjectivity of certain homomorphisms of $R$-modules as in the proof of
Lemma~\ref{lem:E-repr-long}, and hence it is enough to prove them locally in the faithfully flat topology on $\Spec R$. Thus,
we can assume that $G$ is split
and $L_Q$ contains a split maximal subtorus as in~Lemma~\ref{lem:relschemes}, although $R$ is no longer a local ring.
Let $x_a(\xi)$, $a\in\Phi$, $\xi\in R$, denote the elementary root unipotents in $G(R)$.
We also use the same notation $v_a\in V_\alpha$, $a\in\pi^{-1}_Q(\alpha)$, for vectors
satisfying $x_a(1)\in X_\alpha(v_a)X_{2\alpha}(V_{2\alpha}).$

(i) Let $\beta_1,\ldots,\beta_n\in\Phi_Q$ be a good chain between
$\delta$ and $\gamma$. By~\cite[Lemma 4]{PS} for any fixed root $a_0\in\pi_Q^{-1}(\delta)$ there
are roots $b_i\in\pi_Q^{-1}(\beta_i)$, $1\le i\le n$, such that each sum $a_0+b_1+b_2+\ldots+b_i$
is a root.
Clearly, the vectors $v_a$, $a\in\pi_Q^{-1}(\delta)$, and $v_b$, $b\in\pi_Q^{-1}(\gamma)$,
are basis vectors for $V_\delta$ and $V_{\gamma}$ respectively.
For any $a\in \pi_Q^{-1}(\delta)$ such that $a+b_1+\ldots+b_i$ is a root for all $i$, we deduce from~\eqref{eq:vavb} that
$$
N_{\delta,\beta_1,\ldots,\beta_n}(v_a,v_{b_1},\ldots,v_{b_n})=m(a)\cdot v_{a+b_1+\ldots+b_n}
$$
 for some $m(a)\in R^\times$. If $a+b_1+\ldots+b_i$ is not a root for some $i$, then
$$
N_{\delta,\beta_1,\ldots,\beta_n}(v_a,v_{b_1},\ldots,v_{b_n})=0.
$$
Clearly, all non-zero vectors $v_{a+b_1+\ldots+b_n}$, $a\in\pi_Q^{-1}(\delta)$, are distinct
and belong to a basis of $V_{\alpha+\beta_1+\ldots+\beta_n}=V_\gamma$. Therefore, composing $N_{\delta,\beta_1,\ldots,\beta_n}(-,v_{b_1},\ldots,v_{b_n})$
with a suitable endomorphism of $V_\gamma$, we may obtain
any $R$-linear homomorphism
$\phi=\phi_{a_0}:V_\delta\to V_{\gamma}$ that takes a prescribed value on $v_{a_0}$, and maps all other basis vectors
$v_a\in V_\delta$, $a\neq a_0$, to $0$. This settles (i).

(ii) Let $b_0\in\pi_Q^{-1}(\gamma)$ be any root. By~\cite[Lemma 4]{PS} there are roots $a\in\pi_Q^{-1}(\delta)$,
$b_i\in\pi_Q^{-1}(\beta_i)$ such that $a+b_1+\ldots+b_i$ is a root for all $1\le i\le n$,
and $a+b_1+\ldots+b_n=b_0$.
Again, by the Chevalley commutator formula we have
$$
N_{\delta,\alpha_1,\ldots,\alpha_n}(v_a,v_{b_1},\ldots,v_{b_n})=
m(b_0)\cdot v_{b_0}
$$
for some $m(b_0)\in R^\times$. Therefore, for any fixed vector of the form $v_{b_0}$, $b_0\in\pi_Q^{-1}(\gamma)$,
there is a  homomorphism $V_\delta\to V_\gamma$ represented by the chain $\beta_1,\ldots,\beta_n$,
that takes some basis vector $v_a$ of $V_\delta$ to $v_{b_0}$. Precomposing with an endomorphism of $V_\delta$,
we can obtain a homomorphism that sends all other basis vectors to $0$.
Since such vectors $v_{b_0}$ generate $V_\gamma$, (ii) is settled.

(iii) By Lemma~\ref{lem:long-good} there is
a good chain from $\delta$ to $\gamma$.
If $\delta$ is a root of maximal length, then all elements of $\End_R(V_\delta)$ are finite sums of
$E$-representable homomorphisms by Lemma~\ref{lem:E-repr-long}. Since composition of $E$-representable
homomorphisms is $E$-representable, the claim of (iii) then follows from (ii). Now let $\delta$ be arbitrary.
By Lemma~\ref{lem:long-good} there is a good chain between $\delta$ and any root $\alpha$ of maximal length.
Since all $R$-modules involved are free, any $R$-linear homomorphism $V_\delta\to V_\gamma$ is a sum of homomorphisms
that factor through $V_\alpha$. Now the claim of (iii) for such $\delta$ follows from (i) and the previous case.
\end{proof}

\begin{lem}\label{lem:E-repr-short}
Let $\gamma$ be a short root in an irreducible component $\Psi$ of $\Phi_Q$ of type $B_l$, $C_l$ {\rm ($l\ge 2$)}, or $F_4$.
 Let $\beta\in\Phi_Q$ be a long root such that $\gamma,\beta$
form a basis of a root subsystem of $\Phi_Q$ of type $B_2$.
For any root $\delta\in\Psi$, any $R$-linear homomorphism $V_\delta\to V_\gamma$
is a finite sum of homomorphisms $\phi:V_\delta\to V_\gamma$ satisfying the following:
there are $u\in V_{-(\gamma+\beta)}$ and $R$-linear homomorphisms $\psi_1:V_\delta\to V_{\beta+2\gamma}$,
$\psi_2:V_\delta\to V_{-\beta}$ such that for any $v\in V_\delta$ one has
$$
X_\gamma(\phi(v))=X_\gamma\bigl(N_{\beta+2\gamma,-(\beta+\gamma),1,1}(\psi_1(v),u)\bigr)=[X_{\beta+2\gamma}(\psi_1(v)),X_{-(\beta+\gamma)}(u)]\cdot X_{-\beta}(\psi_2(v)),
$$
and $\psi_1$, $\psi_2$ are finite sums of $E$-representable homomorphisms.
\end{lem}
\begin{proof}
It is clear that any $R$-linear homomorphism $V_\delta\to V_\gamma$ is a finite sum of $R$-linear homomorphisms
that factor through $V_{\beta+2\gamma}$, since the latter is a free $R$-module of non-zero dimension.
By Lemma~\ref{lem:E-repr-alltomedlong} (iii) any $R$-linear homomorphism $V_\delta\to V_{\beta+2\gamma}$ is
a finite sum of $E$-representable ones, hence we can assume that $\delta=\beta+2\gamma$ from the start.
By the generalized Chevalley commutator formula~\eqref{eq:Chev} for any $w\in V_{\beta+2\gamma}$, $u\in V_{-(\beta+\gamma)}$
one has
$$
X_{\gamma}\bigl(N_{\beta+2\gamma,-(\beta+\gamma),1,1}(w,u)\bigr)=[X_{\beta+2\gamma}(w),X_{-(\beta+\gamma)}(u)]\cdot
X_{-\beta}\bigl(-N_{\beta+2\gamma,-(\beta+\gamma),1,2}(w,u)\bigr).
$$
By Lemma~\ref{lem:const} the image of the bilinear map $N_{\beta+2\gamma,-(\beta+\gamma),1,1}$ generates
$V_\gamma$. Since
by Lemma~\ref{lem:E-repr-long} any element of $\End_R(V_{\beta+2\gamma})$ is a finite sum of
$E$-representable homomorphisms, we conclude that any $R$-linear homomorphism $V_{\beta+2\gamma}\to V_\gamma$
is a finite sum of homomorphisms of the form $\phi(-)=N_{\beta+2\gamma,-(\beta+\gamma),1,1}(\psi_1(-),u)$,
for some $E$-representable $\psi_1\in\End_R(V_{2\gamma+\beta})$ and some $u\in V_{-(\gamma+\beta)}$.
The proof is now finished by observing that $\psi_2(-)=N_{\beta+2\gamma,-(\beta+\gamma),1,2}(\psi_1(-),u)$ is
an $R$-linear homomorphism $V_{\beta+2\gamma}\to V_{-\beta}$, and hence a finite sum of $E$-representable  homomorphisms
by Lemma~\ref{lem:E-repr-alltomedlong} (iii).
\end{proof}

\subsection{Localization lemmas}
Next we prove several statements that allow to pass from the local ring case to the non-local ring case.
Throughout this subsection, $G$ is a reductive group scheme over a connected
commutative ring $R$, and $P$ is a parabolic $R$-subgroup
of $G$. We do not impose any invertibility conditions on $R$. We denote by
$\Phi_P$ be the system of relative roots for $P$, and by $X_{\alpha}(V_\alpha)$, $\alpha\in\Phi_P$,
the relative root subschemes of $G$ with respect to $P$.

For any $\kappa\in R$ we denote by $R_\kappa$ the localization of $R$ at $\kappa$, and by
$$
F_\kappa:R\to R_\kappa
$$
the localization homomorphism, as well
as the induced homomorphism $G(R)\to G(R_\kappa)$. We also work with elements of $G(R[Y])$
and $G(R[Y,Z])$, where $R[Y]$ and $R[Y,Z]$ are the rings of polynomials over $R$
in one and two variables respectively. We identify $G(R)$ with a subset
of $G(R[Y])$ or $G(R[Y,Z])$ via the natural ring embeddings. The homomorphism $F_\kappa$ naturally extends to these groups.
Similarly, for any maximal ideal $m$ of $R$ we denote by
$$
F_m:R\to R_m
$$ the localization homomorphism, as well
as the induced homomorphism $G(R)\to G(R_m)$.

\begin{lem}\label{lem:PS-12-cor}
Let $m$ be a maximal ideal of $R$, let $Q\le P_{R_m}$ be a parabolic $R_m$-subgroup of $G_{R_m}$ with a system
of relative roots $\Phi_Q$ and relative root subschemes $X_\alpha(V_\alpha)$, $\alpha\in\Phi_Q$.
For any $\beta\in\Phi_Q$ and any $u\in V_\beta$ there is $\lambda\in R\setminus m$ such that
$X_\beta(\lambda YRu)\subseteq F_m\bigl(E_P(YR[Y])^{E_P(R)}\bigr)$.
\end{lem}
\begin{proof}
By Lemma~\ref{lem:lemma12} (i) combined with Lemma~\ref{lem:PQ-relroots} there are  $k>0$,
relative roots
$\alpha_i,\alpha_{ij}\in\Phi_P$, elements $u_i\in V_{\alpha_i}\otimes R_m$,  $u_{ij}\in V_{\alpha_{ij}}\otimes R_m$, and integers
$k_i,n_i,l_{ij}>0$ {\rm ($1\le i\le l$, $1\le j\le l_i$)},
which satisfy the equality
$$
X_\beta(YZ^k u)=\prod\limits_{i=1}^l X_{\alpha_i}(Y^{k_i}Z^{n_i} u_{i})^{\ds\mbox{$\prod\limits_{j=1}^{l_i}$}
X_{\alpha_{ij}}(Z^{l_{ij}} u_{ij})}
$$
as elements of $G(R_m[Y,Z])$.
Pick $\eta\in R\setminus m$ such that $\eta u_i\in V_{\alpha_i}$ and $\eta u_{ij}\in V_{\alpha_{ij}}$
for all $i$ and $j$, then $X_\beta(\eta^kYu)\in F_m\bigl(E_P(YR[Y])^{E(R)}\bigr)$. Replacing $Y$ by an arbitrary $R$-multiple
finishes the proof.
\end{proof}

The following standard trick allows to make preimages under localization more tractable.

\begin{lem}\label{lem:loc-inj}
Let $A$ be any commutative ring, let $H$ be a reductive group scheme over $A$, and
let $\kappa\in A$ be any element. For any $g(Y),h(Y)\in H(A[Y],YA[Y])$ such that $F_\kappa(g(Y))=F_\kappa(h(Y))$ there
is $n\ge 1$ such that $g(\kappa^nY)=h(\kappa^nY)$.
\end{lem}
\begin{proof}
Since $A$ is an inductive limit of finitely generated subrings, and $H$ commutes with inductive limits,
we can assume that $A$ is noetherian.
Since $A$ is noetherian, there is an integer $n\ge 0$ such that
$F_\kappa:\kappa^n A\to A_\kappa$ is injective. By~\cite[Theorem 3.1]{Thomason} $H$ has an exact linear representation
over $A$. Hence $F_\kappa:H(A[Y],\kappa^n A[Y])\to H(A_\kappa[Y])$
is injective as well. Then $g(\kappa^nY)=h(\kappa^nY)$.
\end{proof}

Sometimes we need the following corollary of Lemma~\ref{lem:loc-inj}.

\begin{lem}\label{lem:local-polyinclusion}
Let $m$ be a maximal ideal of $R$. Fix any
$\alpha\in \Phi_P$, $u\in V_\alpha$, and a subset $S\subseteq G(R)$. If
$F_m\bigl(X_\alpha(Yu)\bigr)\in F_m\bigl(\bigl[S^{E(R[Y])},E_P(R[Y],YR[Y])\bigr]\bigr)$ inside $G(R_m[Y])$,
then there is $\lambda\in R\setminus m$ such that
$X_\alpha(\lambda Ru)\subseteq S^{E(R)}$ inside $G(R)$.
\end{lem}
\begin{proof}

Clearly, there is $\kappa\in R\setminus m$ such that
\begin{equation*}
F_\kappa\bigl(X_{\alpha}(Y\cdot u)\bigr)\in F_\kappa\bigl(\bigl[S^{E(R[Y])},E_P(R[Y],YR[Y])\bigr]\bigr).
\end{equation*}
By Lemma~\ref{lem:loc-inj} there is $n\ge 1$ such that
\begin{equation*}
X_{\alpha}(\kappa^n Y u)\in \bigl[S^{E(R[\kappa^n Y])},E_P(R[\kappa^n Y],\kappa^n YR[\kappa^n Y])\bigr]\subseteq S^{E(R[Y])}.
\end{equation*}
Once we set $\lambda=\kappa^n$ and make $Y$ run over $R$, we obtain the claim of the lemma.
\end{proof}

The following generalization of Lemma~\ref{lem:PS-12-cor} will be used in~\S~\ref{sec:main-proof}.

\begin{lem}\label{lem:nu-lift}
Assume the setting of Lemma~\ref{lem:PS-12-cor}, and assume moreover that $R$ is noetherian.

(i) For any finite set of elements $g_1(Y),\ldots,g_k(Y)\in E_Q(YR_m[Y])$
there are $\nu\in R\setminus m$, and elements $h_1(Y),\ldots,h_k(Y)\in E_P(YR[Y])^{E(R)}$
such that $F_m(h_i(Y))=g_i(\nu Y)$ for all $1\le i\le k$.

(ii) If $\nu'\in R\setminus m$ and $h_1'(Y),\ldots,h'_k(Y)$ is another such datum, then there are $\nu_1,\nu_2\in R\setminus m$ such that
$h_i(\nu_1 Y)=h_i'(\nu_2 Y)$ for all $1\le i\le k$.

(iii) If, moreover, $g_1(Y)g_2(Y)\ldots g_k(Y)=1$,
then one can choose $\nu$ and $h_i(Y)$ so that $h_1(Y)h_2(Y)\ldots h_k(Y)=1$.
\end{lem}
\begin{proof}
Since $E_Q(YR_m[Y])$ is generated by $X_\alpha(Y^Nv)$, $\alpha\in\Phi_Q$, $N\ge 1$, $v\in V_\alpha$, clearly, we can assume that
each $g_i(Y)$ is of this form.
By Lemma~\ref{lem:PS-12-cor} there is $\lambda\in R\setminus m$ such that
$X_\alpha(\lambda Y v)\in F_m\bigl(E_P(YR[Y])^{E_P(R)}\bigr)$ inside $G(R_m[Y])$.
Since $W(V_\alpha)$ is a finitely presented $R_m$-scheme,
there is $\mu\in R\setminus m$ such that $V_\alpha$ and $X_\alpha:W(V_\alpha)\to G_{R_m}$ are defined already over
$R_\mu$, and $X_\alpha(\lambda Y v)$ lifts to an element of $F_\mu\bigl(E_P(YR[Y])^{E_P(R)}\bigr)\subseteq G(R_\mu[Y])$.
Now let $\kappa\in R\setminus m$ be any element divided by $\lambda\mu$, i.e. $\kappa=\lambda\mu\kappa'$. Replace $Y$
by $\mu\kappa' Y$,  then, since $R_\mu\to R_m$ factors through $R_\kappa\to R_m$, we have
$$
X_\alpha(\kappa Y v)\in F_\kappa\bigl(E_P(\mu\kappa' YR[Y])^{E_P(R)}\bigr)\subseteq
F_\kappa\bigl(E_P( YR[Y])^{E_P(R)}\bigr)
$$
inside $G(R_\kappa[Y])$. Since $R$ is noetherian, there is $n\ge 1$
such that $F_\kappa:\kappa^n R\to R_\kappa$ is injective, and hence $F_\kappa:G(R[Y],\kappa^n R[Y])\to G(R_\kappa[Y])$ is injective as well.
Set $F_{\kappa,n}=F_\kappa|_{G(R[Y],\kappa^n R[Y])}$ for short.
Once we replace $Y$ by $\kappa^n Y^N$, we obtain a correctly defined element
$$
h(Y)=(F_{\kappa,n})^{-1}\bigl(X_\alpha(\kappa^{n+1} Y^N v)\bigr)\in E_P(\kappa^nY^NR[Y])^{E(R)}.
$$
Set $\nu=\kappa^{n+1}$, then
$$
F_m(h(Y))=F_{mR_\kappa}\circ F_\kappa(h(Y))=X_\alpha(\nu Y^N v),
$$
as required. Observe that by construction $\kappa$ can be replaced by an arbitrary multiple, hence we can choose
an element $\nu$ suitable for all elements $g_i(Y)$, $1\le i\le k$.
  The
injectivity of $F_{\kappa,n}$ implies that if $g_1(Y)g_2(Y)\ldots g_k(Y)=1$, then $h_1(Y)h_2(Y)\ldots h_k(Y)=1$ as well.

It remains to prove (ii). Assume that $h_i'(Y)$ is another lifting of $g_i(Y)$ corresponding to $\nu'$. Then
$$
F_m(h_i'(\nu Y))=g_i(\nu\nu' v)=F_m(h_i(\nu' Y)).
$$
Then there is $\mu\in R\setminus m$ such that $F_\mu(h_i'(\nu Y))=F_\mu(h_i(\nu' Y))$ for all $1\le i\le k$.
By Lemma~\ref{lem:loc-inj} there is $n\ge 1$ such that
$h_i'(\nu\mu^n Y)=h_i(\nu'\mu^n Y)$, as required.
\end{proof}

\subsection{Proof of Theorem~\ref{thm:E-normal}}\label{sec:proof-th-normal}

We assume that we are in the setting of Theorem~\ref{thm:E-normal}. In particular, $R$ is a connected ring,
the absolute root system type $\Phi$ of $G$ is irreducible,
and the structure constants of $\Phi$ are invertible in $R$. We
fix a parabolic subgroup $P$ of $G$ over $R$, and hence we are given a system of relative roots
$\Phi_P$ and relative root subschemes $X_{\alpha}(V_\alpha)$, $\alpha\in\Phi_P$, over $R$.

We also  fix an $E(R)$-normalized
subgroup $N$ of $G(R)$.
For any $\alpha\in\Phi_P$ we define subsets $M_\alpha\subseteq V_\alpha$ by the equality
$$
N\cap X_\alpha(V_\alpha)=X_\alpha(M_\alpha).
$$

In the following lemmas $m$ always denotes an arbitrary maximal ideal of $R$, $Q\le P_{R_m}$ denotes
a minimal parabolic $R_m$-subgroup of $G_{R_m}$, provided with a system
of relative roots $\Phi_Q$ and relative root subschemes $X_\alpha(V_\alpha)$, $\alpha\in\Phi_Q$.
Note that since $Q$ is minimal, $\Phi_Q$ is a root system in the sense of Bourbaki. Since $\Phi$ is irreducible,
$\Phi_Q$ is also irreducible. Summing up, $Q$ and $X_\delta(V_\delta)$, $\delta\in\Phi_Q$, are subject
to~\S~\ref{sec:local-Erepr}.

\begin{lem}\label{lem:Erepr-PQ}
For any $\delta,\gamma\in\Phi_Q$ such that $\gamma$ is not of minimal length if $\Phi_Q$ is of type $G_2$ or $BC_l$,
and any $R_m$-linear homomorphism $\phi:V_\delta\to V_\gamma$, there is
$\lambda\in R\setminus m$ such that for any $v\in V_\delta$
$$
X_\gamma(\lambda\cdot ZY\cdot\phi(v))\in \bigl[X_\delta(Zv)^{F_m(E_P(R))},\, F_m(E_P(R[Y],YR[Y]))\bigr]
$$ inside $G(R_m[Y,Z])$.
\end{lem}
\begin{proof}
Set $H=\bigl[X_\delta(Zv)^{F_m(E_P(R))},\, F_m(E_P(R[Y],YR[Y]))\bigr]$.
Assume first that $\gamma$ is of maximal or medium length in $\Phi_Q$. Then by Lemma~\ref{lem:E-repr-alltomedlong} (iii)
$\phi$ is a finite sum of $E$-representable homomorphisms. By the addition formula~\eqref{eq:sum}
it is enough to prove the claim for each summand. So we can assume that $\phi$ is $E$-representable.
By Lemma~\ref{lem:chain-comm} this means that
there is a good chain $\beta_1,\ldots,\beta_n$ between $\delta$ and $\gamma$
in the sense of Definition~\ref{dfn:good},
as well as vectors $v_i\in V_{\beta_i}$,
$1\le i\le n$, such that
\begin{equation*}
\bigl[\ldots\bigl[\bigl[X_{\delta}(v),X_{\beta_1}(v_1)\bigr],
X_{\beta_2}(v_2)\bigr],\ldots,X_{\beta_n}(v_n)\bigr]
=X_{\gamma}\bigl(N_{\delta,\beta_1,\ldots,\beta_n}(v,v_1,\ldots,v_n)\bigr)
=X_{\gamma}(\phi(v))
\end{equation*}
for any $v\in V_\delta$.
By Lemma~\ref{lem:PS-12-cor} there is a single $\lambda\in R\setminus m$ such that
$X_{\beta_i}(\lambda Y v_i)\in F_m\bigl(E_P(R[Y],YR[Y])\bigr)$ for all $1\le i\le n$. Then
for any $v\in V_\delta$ one has
$$
X_{\gamma}\bigl(\lambda^{n} YZ\cdot \phi(v)\bigr)=
X_{\gamma}\bigl(N_{\delta,\beta_1,\ldots,\beta_n}(Zv,\lambda Yv_1,\lambda v_2,\ldots,\lambda v_n)
\bigr)\in H.
$$

Now assume that $\Phi_Q$ is of type $B_l$, $C_l$ or $F_4$, and $\gamma$ is a short root. Let $\beta$
be a long root such that $\gamma,\beta$ form a basis of a root system of type $B_2$. By Lemma~\ref{lem:E-repr-short}
any $R$-linear homomorphism $V_\delta\to V_\gamma$ is a finite sum of homomorphisms
$\phi:V_\delta\to V_\gamma$ satisfying the following:
there are $u\in V_{-(\gamma+\beta)}$ and $R$-linear homomorphisms $\psi_1:V_\delta\to V_{\beta+2\gamma}$,
$\psi_2:V_\delta\to V_{-\beta}$ such that for any $v\in V_\delta$ one has
$$
X_\gamma(\phi(v))=X_\gamma\bigl(N_{\beta+2\gamma,-(\beta+\gamma),1,1}(\psi_1(v),u)\bigr)=
\bigl[X_{\beta+2\gamma}(\psi_1(v)),X_{-(\beta+\gamma)}(u)\bigr]\cdot X_{-\beta}(\psi_2(v)),
$$
where
$$
\psi_2(v)=-N_{\beta+2\gamma,-(\beta+\gamma),1,2}(\psi_1(v),u).
$$
By the previous case there are $\lambda,\mu\in R\setminus m$ such that
$$
X_{\beta+2\gamma}(\lambda ZY\cdot \psi_1(v)),\ X_{-\beta}(\mu ZY\cdot \psi_2(v))\in
H.
$$
Also by Lemma~\ref{lem:PS-12-cor} there is $\nu\in R\setminus m$ such that
$X_{-(\beta+\gamma)}(\nu u)\in F_m(E_P(R))$. Then
$$
\begin{array}{rcl}
X_\gamma\bigl(\lambda\mu\nu ZY\phi(v)\bigr)&=&
\bigl[X_{\beta+2\gamma}\bigl(\lambda Z(\mu Y)\cdot\psi_1(v)\bigr),X_{-(\beta+\gamma)}(\nu u)\bigr]\cdot
X_{-\beta}\bigl(\mu Z(\lambda\nu^2Y)\cdot \psi_2(v)\bigr)\\
&\in& [H,F_m(E_P[R])]\cdot H\subseteq  H,\\
\end{array}
$$
as required.
\end{proof}

\begin{lem}\label{lem:ideal-locally}
Let $J$ be any ideal of $R_m$.
For any $\delta,\gamma\in\Phi_Q$ one has
\begin{equation}\label{eq:ideal-locally}
X_\gamma(ZY\cdot JV_\gamma) 
\subseteq \bigl[X_\delta(JV_\delta\otimes_{R_m}ZR_m[Z])^{F_m(E_P(R))},\, F_m(E_P(R[Y],YR[Y]))\bigr]
\end{equation}
inside $G(R_m[Z,Y])$.
\end{lem}
\begin{proof}
The claim follows immediately from Lemma~\ref{lem:Erepr-PQ}, unless
$\gamma$ is a root of minimal length in a root subsystem of type $G_2$ or $BC_l$. Assume the latter. Let
$\alpha\in\Phi_Q$ be a root such that $\gamma,\alpha$ is a system of simple roots in a root subsystem
of type $G_2$ or $BC_2$ in $\Phi_Q$. In the $BC_2$ case, one has
\begin{equation}\label{eq:bc2-extra}
[X_{\alpha+\gamma}(u),X_{-\alpha}(v)]=X_{\gamma}(N_{\alpha+\gamma,-\alpha,1,1}(u,v))\cdot
X_{2\gamma}(N_{\alpha+\gamma,-\alpha,2,2}(u,v))\cdot
X_{2\gamma+\alpha}(N_{\alpha+\gamma,-\alpha,2,1}(u,v))
\end{equation}
for any $u\in V_{\alpha+\gamma}\otimes_{R_m} R_m[Y,Z]$, $v\in V_{\alpha+\gamma}\otimes_{R_m} R_m[Y,Z]$.
Since $\im (N_{\alpha+\gamma,-\alpha,1,1})$ additively generates $V_{\gamma}$ by~Lemma~\ref{lem:const},
for any $w\in JV_\gamma$ one has
\begin{equation*}
X_\gamma(ZY\cdot w)\in \left<\left[
X_{-\alpha}(Z\cdot JV_{-\alpha}),\,X_{\alpha+\gamma}(YV_{\alpha+\gamma})\right],\;
X_{2\gamma}(Z^2Y^2\cdot JV_{2\gamma}),\; X_{2\gamma+\alpha}(Z^2Y\cdot JV_{2\gamma+\alpha})
\right>.
\end{equation*}
Since $V_\gamma$ is a finitely generated free $R_m$-module, the module
$JV_\gamma$ is also finitely generated; let $e_1,\ldots,e_k$ be a generating
system of $JV_\gamma$. By Lemma~\ref{lem:PS-12-cor} (applied to $\beta=\alpha+\gamma$)
one finds $\mu\in R\setminus m$ such that for all $1\le i\le k$
\begin{equation*}
X_\gamma(\mu ZY\cdot e_i)\in \left<
\left[X_{-\alpha}(Z\cdot JV_{-\alpha}),\,F_m(E_P(YR[Y]))\right],\;
X_{2\gamma}(Z^2Y^2\cdot JV_{2\gamma}),\; X_{2\gamma+\alpha}(Z^2Y\cdot JV_{2\gamma+\alpha})
\right>.
\end{equation*}
Since any $w\in JV_\gamma$ is an
$R_m$-linear combination of $e_i$, replacing $Z$ by its $R_m$-multiples and applying the addition formula~\eqref{eq:sum},
we conclude  that
\begin{equation*}
X_\gamma(ZY\cdot JV_\gamma)\subseteq \left<
\left[X_{-\alpha}(Z\cdot JV_{-\alpha}),\,F_m(E_P(YR[Y]))\right],\;
X_{2\gamma}(Z^2Y^2\cdot JV_{2\gamma}),\; X_{2\gamma+\alpha}(Z^2Y\cdot JV_{2\gamma+\alpha})
\right>.
\end{equation*}
Now~\eqref{eq:ideal-locally} follows by applying Lemma~\ref{lem:Erepr-PQ} to the roots $-\alpha$, $2\gamma$, and $2\gamma+\alpha$
instead of $\gamma$.

In the $G_2$ case the proof is the same except that one uses the following equality instead of~\eqref{eq:bc2-extra}:
\begin{equation}\label{eq:g2-extra}
\begin{array}{rcl}
[[X_{-2\alpha-3\gamma}(u),X_{\alpha+2\gamma}(v)],X_{\alpha+2\gamma}(w)]&=&
X_{\gamma}(N_{-2\alpha-3\gamma,\alpha+2\gamma,\alpha+2\gamma}(u,v,w))\cdot\\
&&\cdot X_{-\alpha}(\psi_1(u,v,w))\cdot X_{\alpha+3\gamma}(\psi_2(u,v,w)),
\end{array}
\end{equation}
where $\psi_1(u,v,w)=N_{-\alpha-\gamma,\alpha+2\gamma,2,1}\bigl(N_{-2\alpha-3\gamma,\alpha+2\gamma,1,1}(u,v),w\bigr)$
and
$$
\psi_2(u,v,w)=N_{-\alpha-\gamma,\alpha+2\gamma,1,2}\bigl(N_{-2\alpha-3\gamma,\alpha+2\gamma,1,1}(u,v),w\bigr)+
N_{\gamma,\alpha+2\gamma,1,1}\bigl(N_{-2\alpha-3\gamma,\alpha+2\gamma,1,2}(u,v),w\bigr)
$$
by the generalized Chevalley commutator formula~\eqref{eq:Chev}.
\end{proof}

\begin{lem}\label{lem:prod-extract}
Let $\tilde\alpha$ be the highest root of $\Phi_P$.
Consider $\mathbf{x}=\prod\limits_{\alpha\in\pi_{PQ}^{-1}(\tilde\alpha)}X_\alpha(v_\alpha)\in G(R_m)$, where
 $v_\alpha\in V_\alpha$ and the product is taken in any order compatible with the height of roots $\alpha\in\Phi_Q$.
Let $\gamma\in\Phi_Q$ be a root such that $\gamma$ is not of minimal length if $\Phi_Q$
is of type $G_2$ or $BC_l$. For any $\alpha\in\pi_{PQ}^{-1}(\tilde\alpha)$ and any $R_m$-linear
homomorphism $\phi:V_\alpha\to V_\gamma$
there is $\lambda\in R\setminus m$ such that
\begin{equation}\label{eq:prod-extract-claim}
X_\gamma(\lambda\cdot Y\cdot \phi(v_\alpha))\in \bigl[\mathbf{x}^{F_m(E_P(R))},\, F_m(E_P(R[Y],YR[Y]))\bigr]
\end{equation}
inside $G(R_m[Y])$.
\end{lem}
\begin{proof}
Fix an total order on $\Phi_Q^+$ compatible with height.
Write
$$
\mathbf{x}=\prod\limits_{\alpha\in\pi^{-1}_{PQ}(\tilde\alpha)}X_\alpha(v_\alpha),
$$
where the product is taken in the respective order, and let $S\subseteq\pi^{-1}_{PQ}(\tilde\alpha)$ be
the set of all $\alpha$ such that $v_\alpha\neq 0$. Note that if $\Phi$ is of type $G_2$ or $BC_l$, then then no root
in $\pi_{PQ}^{-1}(\tilde\alpha)$ is of minimal length.
Let $\alpha_0$ be a root of minimal height in $S$.  We prove the claim of the lemma
by inverse induction on the height of $\alpha_0$ (the first induction).

If $\alpha_0$ is the unique highest root of $\Phi_Q$, the claim of the lemma holds by Lemma~\ref{lem:Erepr-PQ} with $Z=1$.
Assume that $\alpha_0$ is not the highest root of $\Phi_Q$. We prove the claim~\eqref{eq:prod-extract-claim}
for all $\alpha\in S$ by induction with respect to the fixed total order on $\Phi_Q^+$, starting with $\alpha=\alpha_0$
(the second induction). Thus, the inductive step of the second induction consists in proving~\eqref{eq:prod-extract-claim}
for the fixed $\mathbf{x}$ and a particular $\alpha\in S$, given that~\eqref{eq:prod-extract-claim} is known for
all elements of $S$ which are strictly smaller than $\alpha$ with respect to the total order; we denote the set
of such elements of $S$ by $S'$. Note that we may have $S'=\emptyset$.
For all $\delta\in S'$ we let $\lambda_\delta\in R\setminus m$ be
the corresponding constant appearing in~\eqref{eq:prod-extract-claim}.

If $\alpha$ is the highest root of $\Phi_Q^+$, it is the only element of $S\setminus S'$,
and~\eqref{eq:prod-extract-claim} for $\alpha$ easily follows from~Lemma~\ref{lem:Erepr-PQ}. Otherwise
there is $\beta\in\Phi_Q^+$ such that $\alpha+\beta\in\Phi_Q$. Clearly, $i\alpha+j\beta$, where $i,j>0$, can be a root
only if $i=1$, since all such roots belong to $\pi_{PQ}^{-1}(\tilde\alpha)$. If $m_{\alpha\beta}$ is the maximal
positive integer such that $\alpha+m_{\alpha\beta}\beta\in\Phi_Q$, then the sequence $\beta,\ldots,\beta$, where
$\beta$ is repeated $m_{\alpha\beta}$ times, is a good chain between $\alpha$ and $\alpha+m_{\alpha\beta}\beta$.
By Lemma~\ref{lem:PS-12-cor} for any $u_1,\ldots,u_{m_{\alpha\beta}}\in V_\beta$ one can find $\mu\in R\setminus m$ such that
$X_\beta(\mu\cdot u_i)\in F_m(E(R))$ for all $1\le i\le m_{\alpha\beta}$, and moreover $\mu$ is divisible by
$\prod_{\delta\in S'}\lambda_\delta$. Clearly, one has
$$
\begin{array}{rcl}
[[\ldots[\mathbf{x},X_\beta(\mu u_1)],\ldots],X_\beta(\mu u_{m_{\alpha\beta}})]&=&
\prod\limits_{\delta\in S}[[\ldots[X_\delta( v_\delta),X_\beta(\mu u_1)],\ldots],X_\beta(\mu u_{m_{\alpha\beta}})]\\
&=&\prod\limits_{\delta\in S}\prod\limits_{i\ge m_{\alpha\beta}}X_{\delta+i\beta}(\mu^i\cdot v_{\delta,i}),\\
\end{array}
$$
where each $v_{\delta,i}\in V_{\delta+i\beta}$ depends linearly on $v_\delta$ by the generalized Chevalley commutator
formula~\eqref{eq:Chev}. Then,
since $\prod_{\delta\in S'}\lambda_\delta$ divides $\mu$,  for all $\delta\in S'$ one has
$$
X_{\delta+i\beta}(\mu^i\cdot v_{\delta,i})\in\bigl[\mathbf{x}^{F_m(E_P(R))},\, F_m(E_P(R))\bigr]
\subseteq\mathbf{x}^{F_m(E_P(R))}.
$$
 Set
$$
\mathbf{y}=\prod\limits_{\delta\in S\setminus S'}\prod\limits_{i\ge m_{\alpha\beta}}X_{\delta+i\beta}(\mu^i\cdot v_{\delta,i}).
$$
Then $\mathbf{y}\in \mathbf{x}^{F_m(E_P(R))}$.
Clearly, $\alpha+m_{\alpha\beta}\beta$ is a  root of minimal height among all roots
appearing in $\mathbf{y}$, and it appears only once. In particular, $\mathbf{y}$
is subject to the inductive assumption of the first induction.
Hence for any $R_m$-linear homomorphism $\psi:V_{\alpha+m_{\alpha\beta}\beta}\to V_\gamma$, where $\gamma\in\Phi_Q$
is not of minimal length if $\Phi_Q$ is of type $G_2$ or $BC_l$, there is $\kappa\in R\setminus m$
such that
\begin{equation}\label{eq:Xgamma-y}
\begin{array}{rcl}
X_{\gamma}\bigl(\kappa Y\mu^{m_{\alpha\beta}}\cdot
\psi(v_{\alpha,m_{\alpha\beta}})\bigr)&\in&
\bigl[\mathbf{y}^{F_m(E_P(R))},\, F_m(E_P(R[Y],YR[Y]))\bigr]\\
&\subseteq&
\bigl[\mathbf{x}^{F_m(E_P(R))},\, F_m(E_P(R[Y],YR[Y]))\bigr].\\
\end{array}
\end{equation}
One has $v_{\alpha,m_{\alpha\beta}}=N_{\alpha,\beta,\ldots,\beta}(v_\alpha,u_1,\ldots,u_{m_{\alpha\beta}})$
by the definition of $m_{\alpha\beta}$
and Lemma~\ref{lem:chain-comm}.
Thus,~\eqref{eq:Xgamma-y} would imply that~\eqref{eq:prod-extract-claim} holds for $\alpha$, once we
show that any $R_m$-linear homomorphism $V_\alpha\to V_\gamma$ is a finite sum of $R_m$-linear homomorphisms
that factor through
$$
N_{\alpha,\beta,\ldots,\beta}(-,u_1,\ldots,u_{m_{\alpha\beta}}):V_\alpha\to V_{\alpha+m_{\alpha\beta}\beta}.
$$
Since $\beta,\ldots,\beta$ ($m_{\alpha\beta}$ times) is a good chain between $\alpha$ and
$\alpha+m_{\alpha\beta}\beta$, by Lemma~\ref{lem:E-repr-alltomedlong} (i)
any $R_m$-linear homomorphism $V_\alpha\to V_{\alpha+m_{\alpha\beta}\beta}$ is a finite sum of
homomorphisms that factor through $N_{\alpha,\beta,\ldots,\beta}(-,u_1,\ldots,u_{m_{\alpha\beta}})$. On the other hand,
any homomorphism $V_\alpha\to V_\gamma$ is a finite sum of homomorphisms that factor through $V_{\alpha+m_{\alpha\beta}\beta}$,
since all these modules are finitely generated projective, and hence free, $R_m$-modules.
\end{proof}

\begin{lem}\label{lem:highest-M}
Let $\tilde\alpha\in\Phi_P$ be the root of maximal height in $\Phi_P$. There is an ideal $I$ in $R$
such that $M_{\tilde\alpha}=IV_{\tilde\alpha}$ and $M_{-\tilde\alpha}=IV_{-\tilde\alpha}$.
\end{lem}
\begin{proof}
By~\eqref{eq:sum} $M_{\tilde\alpha}$ is an additive subgroup of $V_{\tilde\alpha}$. Let $\hat\alpha$ denote one of $\tilde\alpha$,
$-\tilde\alpha$. We will show that for any $R$-linear
homomorphism $\phi:V_{\tilde\alpha}\to V_{\hat\alpha}$ one has $\phi(M_{\tilde\alpha})\subseteq M_{\hat\alpha}$.
First, this will imply that
$M_{\tilde\alpha}$ is an $R$-submodule of $V_{\tilde\alpha}$, and, moreover,
since $V_{\tilde\alpha}$ is a faithfully (recall that the type of $P$ is constant)
projective $R$-module, that there is an ideal $I$ of $R$ such that $IV_{\tilde\alpha}=M_{\tilde\alpha}$.
Then it follows that $IV_{-\tilde\alpha}\subseteq M_{-\tilde\alpha}$. Since $\tilde\alpha$ and
$-\tilde\alpha$ are symmetric, we conclude that $M_{-\tilde\alpha}=IV_{-\tilde\alpha}$ as well.

Now we prove that for any $R$-linear
homomorphism $\phi:V_{\tilde\alpha}\to V_{\hat\alpha}$ one has $\phi(M_{\tilde\alpha})\subseteq M_{\hat\alpha}$.
In order to do that,
clearly, it is enough to show that for any $v\in M_{\tilde\alpha}$ and any maximal ideal $m$ of $R$ there is
$\lambda\in R\setminus m$ such that $\lambda\cdot R\cdot \phi(v)\subseteq M_{\hat\alpha}$,
and then apply decomposition of unity in $R$.

Fix $v\in M_{\tilde\alpha}$ and a maximal ideal $m$ of $R$. We denote by the same letter $\phi$ the $R_m$-linear extension
of this map to a map $V_{\tilde\alpha}\otimes_R R_m\to V_{\hat\alpha}\otimes_R R_m$.
Let $Q$ be a minimal parabolic subroup of $G_{R_m}$ contained in $P_{R_m}$, so that $\Phi_Q$ is an abstract root
system of rank $\ge 2$.
Clearly, $V_{\tilde\alpha}\otimes_R R_m=\bigoplus_{\beta\in\pi_{PQ}^{-1}(\tilde\alpha)} V_\beta$,
so that $F_m(v)=\sum_{\beta\in\pi_{PQ}^{-1}(\tilde\alpha)} v_\beta$ for some $v_\beta\in V_\beta$.
Set
$$
\mathbf{x}=F_m(X_{\tilde\alpha}(v))=X_{\tilde\alpha}(F_m(v)).
$$
Then by Lemma~\ref{lem:PQ-relroots}  we have
$$
\mathbf{x}=\prod\limits_{\beta\in \pi_{PQ}^{-1}(\tilde\alpha)}X_\beta( v_\beta).
$$
Now for every $\beta\in\pi_{PQ}^{-1}(\tilde\alpha)$ one has
$$
\phi(v_\beta)=\sum\limits_{\gamma\in\pi_{PQ}^{-1}(\hat\alpha)}\phi_{\beta\gamma}(v_\beta),
$$
where $\phi_{\beta\gamma}:V_\beta\to V_\gamma$ is the $R_m$-linear homomorphism obtained by composing
$\phi|_{V_\beta}$
with the projection of $V_{\hat\alpha}\otimes_R R_m$ to $V_\gamma$.
Observe that if $\Phi_Q$ is of type $G_2$ or $BC_l$, then $\pi_{PQ}^{-1}(\pm\tilde\alpha)$ contains
no roots of $\Phi_Q$ of minimal length, since all coefficients in their decomposition in terms of simple roots
are strictly smaller than those of the highest root. Then by Lemma~\ref{lem:prod-extract}
for any $\beta\in\pi_{PQ}^{-1}(\tilde\alpha)$, $\gamma\in\pi_{PQ}^{-1}(\hat\alpha)$ there exists
$\mu_{\beta\gamma}\in R\setminus m$  such that
$$
X_\gamma(\mu_{\beta\gamma}Y\cdot\phi_{\beta\gamma}(v_\beta))\in \bigl[\mathbf{x}^{F_m(E_P(R))},\, F_m(E_P(R[Y],YR[Y]))\bigr].
$$
Set $\mu=\prod_{\beta,\gamma}\mu_{\beta\gamma}$ and let $\mu'_{\beta\gamma}$ denote the same product
without the factor $\mu_{\beta\gamma}$. Then one has
\begin{equation*}
\begin{array}{rcl}
F_m\bigl(X_{\hat\alpha}(\mu Y\cdot \phi(v))\bigr)
&=&\prod\limits_{\hbox to 9pt{$\st\gamma\in\pi_{PQ}^{-1}(\hat\alpha)$}}
\ X_\gamma\Bigl(\mu Y\cdot \sum\limits_{\hbox to 9pt{$\st\beta\in\pi_{PQ}^{-1}(\tilde\alpha)$}}\ \phi_{\beta\gamma}(v_\beta)\Bigr)\\
&=&
\prod\limits_{\gamma,\beta}
X_\gamma\Bigl(\mu_{\beta\gamma}  (\mu'_{\beta\gamma}Y)\cdot \phi_{\beta\gamma}(v_\beta)\Bigr)\\
&\in& F_m\left(\bigl[X_{\tilde\alpha}(v)^{E_P(R)},\, E_P(R[Y],YR[Y])\bigr]\right).
\end{array}
\end{equation*}
Then by Lemma~\ref{lem:local-polyinclusion} there is $\nu\in R\setminus m$ such that
$$
X_{\hat\alpha}(\nu\mu\cdot R\cdot \phi(v))\subseteq X_{\tilde\alpha}(v)^{E_P(R)}.
$$
This implies that
$\nu\mu\cdot R\cdot\phi(v)\subseteq M_{\hat\alpha}$, as required.
\end{proof}

\begin{lem}\label{lem:inclusion-M}
Let $\tilde\alpha\in\Phi_P$ be the root of maximal height in $\Phi_P$, and let $I$ be an ideal in $R$.
Then $X_\alpha(IV_\alpha)\subseteq X_{\tilde\alpha}(IV_{\tilde\alpha})^{E(R)}$
for all $\alpha\in\Phi_P^+$.
\end{lem}
\begin{proof}
We prove the claim by induction on the height of $\alpha$.
Fix a maximal localization $R_m$ of $R$, and
let $Q$ be a minimal parabolic subroup of $G_{R_m}$ contained in $P_{R_m}$, so that $\Phi_Q$ is an abstract root
system of rank $\ge 2$.
Take any $v\in IV_\alpha$.
By Lemma~\ref{lem:PQ-relroots} $X_\alpha(ZY\cdot F_m(v))$ is a product of factors $X_\delta(Z^iY^iv_\delta)$
for some $\delta\in\Phi_Q$, $v_\delta\in IV_\delta$, and $i\ge 1$.
Then by Lemma~\ref{lem:ideal-locally} one has
\begin{equation*}
X_\alpha(ZY\cdot F_m(v)) \in
\bigl[X_{\tilde\alpha}(ZR_m[Z]\otimes_R IV_{\tilde\alpha})^{F_m(E_P(R))},\, F_m(E_P(R[Y],YR[Y]))\bigr].
\end{equation*}
Replacing $Z$ by a suitable $\nu\in R\setminus m$, we can secure that
\begin{equation*}
F_m(X_\alpha(\nu Yv))\in
F_m\bigl(\bigl[X_{\tilde\alpha}(IV_{\tilde\alpha})^{E_P(R)},\, E_P(R[Y],YR[Y])\bigr]\bigr).
\end{equation*}
Applying Lemma~\ref{lem:local-polyinclusion}, we conclude that there is $\lambda\in R\setminus m$
such that
\begin{equation*}
X_\alpha(\lambda\nu Rv)\subseteq X_{\tilde\alpha}(IV_{\tilde\alpha})^{E(R)}.
\end{equation*}
Using the decomposition of unity in $R$, the addition formula~\eqref{eq:sum} and the induction hypothesis on the height of $\alpha$, we conclude that
$X_\alpha(v)\in X_{\tilde\alpha}(IV_{\tilde\alpha})^{E(R)}$.
\end{proof}

\begin{proof}[Proof of Theorem~\ref{thm:E-normal}]
By Lemma~\ref{lem:highest-M} there is an ideal $I\subseteq R$ such that $M_{\tilde\alpha}=IV_{\tilde\alpha}$
and $M_{-\tilde\alpha}=IV_{-\tilde\alpha}$. By Lemma~\ref{lem:inclusion-M} one has $IV_\alpha\subseteq M_\alpha$
for any $\alpha\in\Phi_P$. We show that $M_\alpha\subseteq IV_\alpha$.
Without loss of generality assume that
$\alpha\in\Phi_P^+$. After passing from the ring $R$ to $R/I$, we can assume that $M_{\tilde\alpha}=0$, and we need to show that $M_\alpha=0$.
 There is a sequence of relative roots $\beta_1,\beta_2,\ldots,\beta_k\in\Phi_P^+$
such that $\gamma_i=\alpha+\beta_1+\ldots+\beta_i\in\Phi_P$ for all $1\le i\le k$, and
$\gamma_k=\alpha+\beta_1+\ldots+\beta_k=\tilde\alpha$.
Indeed, the corresponding statement for usual root systems is well-known, and the statement for relative roots readily follows
by applying $\pi_P$ to the respective chain and using height induction.
Assume that $v_0\neq 0$ is an element of $M_\alpha$.
By Lemma~\ref{lem:const} there are elements
 $u_i\in V_{\beta_i}$, $1\le i\le k$, such that all elements
$$
v_{i+1}=N_{\gamma_i,\beta_{i+1},1,1}(v_i,u_{i+1}),\quad 0\le i\le k-1,
$$
are non-zero. By the generalized Chevalley commutator formula~\eqref{eq:Chev} one has
$X_{\tilde\alpha}(v_k)\in X_{\alpha}(v_0)^{E_P(R)}$, hence $M_{\tilde\alpha}\neq 0$, a contradiction.

\end{proof}

\section{Proof of the main theorem}\label{sec:main-proof}

\subsection{Two general lemmas}
To prove our main result Theorem~\ref{thm:main},
we still need two more statements that do not use the invertibility of structure constants.

\begin{lem}\label{lem:YZ^3}
Let $R$ be a commutative ring, and let $G$ be a reductive group scheme over $R$.
Assume that $G$ has isotropic rank $\ge 1$ and $G_{R_m}$ has isotropic rank $\ge 2$ for any maximal ideal $m\subseteq R$.
Then  for any strictly proper parabolic subgroup $P$ of $G$ one has
$$
\bigl[E_P(R[Y,Z],Y\cdot R[Y,Z]),\,E_P(R[Y,Z],Z^3\cdot R[Y,Z])\bigr]\le E_P\bigl(R[Y,Z],YZ\cdot R[Y,Z]\bigr).
$$
\end{lem}
\begin{proof}
Since all group schemes involved are finitely presented, it is enough to prove the statement for any noetherian ring $R$.
Since any noetherian ring is a finite direct product of connected rings, we can assume that $R$ is connected and noetherian.

Using~\cite[Lemma 4.1]{St-poly}, one deduces
\begin{multline*}
[E_P(R[Y,Z],Y\cdot R[Y,Z]),\,E_P(R[Y,Z],Z^3\cdot R[Y,Z])]\\
\le [E_P(R[Y,Z],Y\cdot R[Y,Z]),\,E_P(Z^3\cdot R[Y,Z])]^{E_P(R[Y,Z])}\\
=[E_P(Y\cdot R[Y,Z])^{E_P(R)},\,E_P(Z^3\cdot R[Y,Z])]^{E_P(R[Y,Z])}.
\end{multline*}
Therefore, since $E_P(R[Y,Z],YZ\cdot R[Y,Z])$ is normal in $E_P(R[Y,Z])$,
it is enough to prove that
$$
g(Y,Z)=[X_\gamma(Yu)^h,X_\delta(Z^3v)]\in E_P(R[Y,Z],YZ\cdot R[Y,Z])
$$
for all $\gamma,\delta\in\Phi_P$, $u\in V_\gamma$, $v\in V_\delta$, $h\in E_P(R)$. We show that for any maximal
ideal $m$ of $R$ there are $\lambda,\mu\in R\setminus m$ such that $g(\lambda Y,\mu Z)\in E_P(R[Y,Z],YZ\cdot R[Y,Z])$.
Then the decomposition of unity in $R$ together with the height induction on $\gamma,\delta$ would imply the claim.

Let $m$ be a maximal ideal of $R$, and let $Q\le P_{R_m}$ be a minimal parabolic subgroup of $G$.
Our assumptions imply that every irreducible component $\Psi$ of $\Phi_Q$ has isotropic rank $\ge 2$, and
satisfies $\pi_{PQ}(\Psi)\neq 0$.
Since $U_{P_{R_m}}\le U_Q$, one has
$$
F_m(g(Y,Z))\in
[E_Q(R_m[Y],Y\cdot R_m[Y]),\,E_Q(Z^3\cdot R_m[Z])].\\
$$
The group $E_Q(Z^3\cdot R_m[Z])$ is generated by elements $X_\alpha(Z^3 u)$,
$\alpha\in\Phi_Q$ and $u\in V_\alpha\otimes_{R_m}R_m[Z]$.
By Lemma~\ref{lem:EI-gen} $E_Q(R_m[Y],Y\cdot R_m[Y])$ is generated by elements $Z_\beta(a,u_1,\ldots,u_{m_\beta})$,
$a\in E_\beta(R_m[Y])$, $u_i\in V_{i\beta}\otimes_{R_m} Y\cdot R_m[Y]$. If $\alpha$ and $\beta$ are non-collinear,
then by~\cite[Lemma 4.4]{St-poly} one has
$$
[Z_\beta(a,u_1,\ldots,u_{m_\beta}),\, X_\alpha(Z^3 u)]\in E_Q(YZ^3\cdot R_m[Y,Z]).
$$
On the other hand, by Lemma~\ref{lem:lemma11} there is a presentation
$$
X_\alpha(Z^3u)=\prod_{i=1}^k X_{\alpha_i}(Zv_i),
$$
where all $\alpha_i\in\Phi_Q$ are non-collinear to $\alpha$ and $v_i\in V_{\alpha_i}\otimes_{R_m}R_m[Z]$.
Hence if $\beta$ is collinear to $\alpha$, then by~\cite[Lemma 4.4]{St-poly}
$$
[Z_\beta(a,u_1,\ldots,u_{m_\beta}),\, X_\alpha(Z^3 u)]\in E_Q(YZ\cdot R_m[Y,Z])^{E_Q(Z\cdot R_m[Z])}.
$$
Summing up, we have
$$
[Z_\beta(a,u_1,\ldots,u_{m_\beta}),E_Q(Z^3\cdot R_m[Z])]\le E_Q(YZ\cdot R_m[Y,Z])^{E_Q(Z\cdot R_m[Z])}.
$$
Hence
$$
F_m(g(Y,Z))\in \bigl(E_Q(YZ\cdot R_m[Y,Z])^{E_Q(Z\cdot R_m[Z])}\bigr)^{E_Q(R_m[Y],Y\cdot R_m[Y])}.
$$
Let $R[Y,Z,T]$ be the polynomial ring in 3 variables over $R$. Then, clearly, there is
$$
h(Y,Z,T)\in \bigl(E_Q(T\cdot R_m[Y,Z])^{E_Q(Z\cdot R_m[Z])}\bigr)^{E_Q(R_m[Y],Y\cdot R_m[Y])}\le G(R_m[Y,Z,T])
$$
such that $F_m(g(Y,Z))=h(Y,Z,YZ)$. By~\cite[Lemma 15]{PS} there is $\mu\in R\setminus m$ such that
$$
h(Y,Z,\mu T)\in F_m\bigl(E_P(R[Y,Z,T],T\cdot R[Y,Z,T])\bigr).
$$
Then $F_m(g(\mu Y,Z))=h(\mu Y,Z,\mu YZ)\in F_m\bigl(E_P(R[Y,Z],YZ\cdot R[Y,Z])\bigr)$. By~\cite[Lemma 14]{PS}
there is $\kappa\in R\setminus m$ such that $g(\kappa\mu Y,Z)\in E_P(R[Y,Z],YZ\cdot R[Y,Z])$. This finishes the proof.
\end{proof}

The proof of the main result of~\cite{RR} uses bounded generation of the elementary subgroup
of a Chevalley group over the profinite completion $\hat R$ of a commutative ring $R$ with respect to elementary root
generators. In our case, we need it for quasi-split groups.

\begin{lem}\label{lem:bounded-gen}
Let $R$ be a connected semilocal ring, $G$ a quasi-split simply connected reductive group over $R$ such that
the absolute root system $\Phi$ of $G$ is irreducible,
and let $P$ be a strictly proper parabolic
subgroup of $G$. Then $G(R)=E_P(R)$, and there is an integer $N>0$, depending only on $\Phi$, such that each element of
$G(R)$ is a product
of $\le N$ elements of $U_P(R)$ and $U_{P^-}(R)$.
\end{lem}
\begin{proof}
Let $B$ be a Borel subgroup of $G$ contained in $P$. Since $\Phi$ is irreducible,
$\Phi_B$ is irreducible. Then by Lemma~\ref{lem:lemma12} (i) we can assume $P=B$.
The group $G$ is either split, or quasi-split of outer type ${}^2A_n$, $n\ge 2$; ${}^2D_n$, $n\ge 4$;
${}^{3(6)}D_4$ or ${}^2E_6$. Let
$T$ be the maximal torus $B\cap B^-$ of $G$. By~\cite[Exp. XXVI, Th\'eor\`eme 5.1]{SGA3} we have
$$
G(R)=U_B(R)U_{B^-}(R)T(R)U_B(R).
$$
Therefore, it is enough to show that $T(R)$ is boundedly generated by the elements of $U_B(R)$ and $U_{B^-}(R)$.
In what follows we use the terminology and notation of~\cite{PS-tind} and~\cite[Exp. XXIV]{SGA3}.
Let $R\to S$ is a connected Galois ring extension splitting the Dynkin scheme $\Dyn(G)$ of $G$.
Then $T$ is a product of
maximal tori of the standard subgroups $H$ of $G$ of the form $\SL_{2,R}$, $\SU_{3,S/R}$ or
$R_{S/R}(\SL_{2,S})$, corresponding to the distinct $\Gal(S/R)$-orbits of $\Dyn(G)$,
or, in other words, to the simple relative roots $\alpha=\alpha(H)\in\Phi_B$ (see
e.g.~\cite[Proposition 1]{PS-tind}). Here $R_{S/R}(\SL_{2,S})$ denotes
the Weil restriction of $\SL_{2,S}$ from $S$ to $R$, and $\SU_{3,S/R}$ stands for the group of type ${}^2A_2$
split by the extension $S/R$. For all three types of groups one readily sees that every element of $H(R)$ is a product
of $\le 9$ elements of $U_{(\pm\alpha)}(R)$, where $U_{(\pm\alpha)}$ are the unipotent radicals
of opposite parabolic subgroups $B^\pm\cap H$.
The corresponding formulas for $\SU_{3,S/R}$ can be found, for example, in~\cite[p. 196]{Abe-twisted}.
\end{proof}

\subsection{Proof of Theorem~\ref{thm:main}}

From now on, assume the setting of Theorem~\ref{thm:main}. We fix a strictly proper parabolic $R$-subgroup
$P$ of $G$ with a Levi subgroup $L_P$, a system of relative roots $\Phi_P$ and relative root subschemes
$X_\alpha(V_\alpha)$, $\alpha\in\Phi_P$.
We also denote by
$$
p:\widehat{E(R)}\to \overline{E(R)}
$$
the natural homomorphism between completions.

We denote by $R_m$ the localization of $R$ at a maximal ideal $m$, and by $\hat R_m$ the respective
$m$-adic completion.
Let $\mathcal{I}$ denote the set of all ideals of finite index in $R$, and let $\mathcal{M}$ denote the
subset of all maximal ideals of finite index. Let
$$
\hat R=\colim\limits_{I\in\mathcal{I}} R/I
$$
be the profinite completion of $R$.
By~\cite[Lemma 2.1]{RR} (note that in~\cite{RR} $R_m$ denotes the completion of $R$ at $m$) there is a natural isomorphism of topological rings
$$
\hat R\cong\prod_{m\in\mathcal{M}}\hat R_m.
$$
Note that the localization
homomorphism $F_m:\hat R\to \hat R_m$ is the canonical projection.

\begin{lem}\label{lem:easy-completion}
(i) One has $\overline{E(R)}=\overline{G(R)}$, and this group is naturally isomorphic to
$$G(\hat R)=E(\hat R)=\prod\limits_{m\in\mathcal{M}} E(\hat R_m).$$

(ii) For any $\alpha\in\Phi_P$ the homomorphism $p$ maps
the closure $\widehat{X_\alpha(V_\alpha)}$ of $X_\alpha(V_\alpha)$ in $\widehat{E(R)}$ homeomorphically onto $X_\alpha(V_\alpha\otimes_R \hat R)\subseteq
E(\hat R)$. Consequently, $p$ maps the closure $\widehat{U_{(\alpha)}(R)}$ of $U_{(\alpha)}(R)$ homeomorphically
onto $U_{(\alpha)}(\hat R)$.

(iii) For any $\mu,\nu\in\hat R$ let
$$
ev_{\mu,\nu}:G(R[Y,Z])\to G(\hat R)
$$
denote the evaluation at $Y=\mu$, $Z=\nu$.
There is a group homomorphism
$$
\hat{ev}_{\mu,\nu}:E(R[Y,Z])\to \widehat{E(R)}
$$
such that
$p\circ \hat{ev}_{\mu,\nu}=ev_{\mu,\nu}|_{E(R[Y,Z])}$,
and for any $g(Y,Z)\in E(R[Y,Z])$ the induced map
$$
g\colon \hat R\times\hat R\to \widehat{E(R)},\qquad (\mu,\nu)\mapsto \hat{ev}_{\mu,\nu}(g(Y,Z)),
$$ is continuous.
\end{lem}
\begin{proof}
(i) Each ring $\hat R_m$ is a local complete
ring with a finite residue field, hence $G$ is quasi-split over $\hat R_m$ by Lang's theorem
and~\cite[Exp. XXVI, \S 7.15]{SGA3}.
Then by Lemma~\ref{lem:bounded-gen} the group $G(\hat R)=E(\hat R)$ is boundedly generated by $X_\alpha(v)$,
$\alpha\in\Phi_P$, $v\in V_\alpha\otimes_R\hat R$. Then one shows exactly as in~\cite[Prop. 2.5]{RR} that
$\overline{E(R)}=\overline{G(R)}\cong E(\hat R)$.

(ii) By Theorem~\ref{thm:E-normal} for every normal subgroup $N$ of $E(R)$ there exists an ideal $I$ such that
$N\cap X_\alpha(V_\alpha)=X_\alpha(V_\alpha\otimes_R I)$ for all $\alpha\in\Phi_P$.
For any non-multipliable root $\alpha$ the subscheme $X_\alpha(V_\alpha)=U_{(\alpha)}(R)$ is a subgroup of
$E(R)$, and hence $X_\alpha(V_\alpha\otimes_R I)$ is of finite index in $X_\alpha(V_\alpha)$. Then
$I$ is of finite index in $R$. Then the restrictions of the profinite and congruence topologies on $E(R)$ to
$X_\alpha(V_\alpha)$
are the same for any $\alpha\in\Phi_P$. This implies the first claim. The second one follows by Lemma~\ref{lem:rootels} (iv).

(iii) The claims follow immediately from (ii), since the group product in $\widehat{E(R)}$ is continuous.
\end{proof}

\begin{lem}\label{lem:subgroups-in-completion}
For any $m\in\mathcal{M}$, let $Q_m$ be a minimal parabolic $R_m$-subgroup of $G_{R_m}$
contained in $P_{R_m}$. There are injective maps $\hat X_\alpha$,  $\alpha\in\Phi_{Q_m}$, that make
the diagram
\begin{equation*}
\xymatrix@R=20pt@C=35pt{
V_\alpha\otimes_{R_m}\hat R_m\ar[d]^{X_\alpha}\ar[r]^{\hat X_\alpha}&\widehat{E(R)}\ar[d]^{p}\\
\prod\limits_{m\in\mathcal{M}} E(\hat R_m)\ar[r]^{\cong}&\overline{E(R)}\\
}
\end{equation*}
commutative, and satisfy the following properties.

(i) The natural homomorphism $s_{Q_m}:St_{Q_m}(\hat R_m)\to E(\hat R_m)$ factors through
the homomorphism $p|_{\hat\Gamma_m}:\hat\Gamma_m\to E(\hat R_m)\le \overline{E(R)}$, where
$\hat\Gamma_m$ denotes the subgroup of $\widehat{E(R)}$ generated by
all subsets $\hat X_\alpha(V_\alpha\otimes_R\hat R_m)$, $\alpha\in\Phi_{Q_m}$.

(ii) For any two ideals $m,n\in\mathcal{M}$, the subgroups $\hat\Gamma_m$ and $\hat\Gamma_n$
commute elementwise inside $\widehat{E(R)}$.

(iii) For any $\delta\in\Phi_P$ one has
\begin{equation}\label{eq:XP-hat}
\widehat{U_{(\delta)}(R)}=\prod\limits_{m\in\mathcal{M}} \prod_{i\ge 1}\prod_{\alpha\in\pi_{PQ_m}^{-1}(i\delta)}\!\!\!
\hat X_\alpha(V_\alpha\otimes_{R_m}\hat R_m),
\end{equation}
where $\widehat{U_{(\delta)}(R)}\cong U_{(\delta)}(\hat R)$ is the closure of $U_{(\delta)}(R)$ in $\widehat{E(R)}$.
\end{lem}
\begin{proof}
First we define $\hat X_\alpha(\kappa\cdot v)$ for all $\alpha\in\Phi_{Q_m}$, $v\in V_\alpha$, $\kappa\in\hat R_m$.
By Lemma~\ref{lem:nu-lift} there is $\nu\in R\setminus m$ and $h(Y)\in E_P(R[Y],YR[Y])$ such that $F_m(h(Y))=X_\alpha(\nu Yv)$.
Denote by
$$f_m:\hat R_m\to \hat R
$$ the map
that sends $x$ to the element of $\hat R$ that has $x$ in the $m$-th position, and zero everywhere else.
Set $\mu=f_m(\kappa\cdot\nu^{-1})\in\hat R$, and define
\begin{equation}\label{eq:def-hatX}
\hat X_\alpha(\kappa\cdot v)=\hat{ev}_\mu(h(Y)),
\end{equation}
where
$$
\hat{ev}_\mu=\hat{ev}_{\mu,0}|_{E(R[Y])}:E(R[Y])\to\widehat{E(R)}
$$
is the restriction of the homomorphism $\hat{ev}_{\mu,\nu}$ of Lemma~\ref{lem:easy-completion} (iii).
Then we have
$$
F_m\circ p(\hat X_\alpha(\kappa\cdot v))=F_m\bigl(ev_{\mu}(h(Y))\bigr)=
ev_\mu\bigl(F_m(h(Y))\bigr)=X_\alpha(\nu F_m(\mu) v)=X_\alpha(\kappa v),
$$
as required. Observe that $\hat X_\alpha(\kappa\cdot v)$ is independent of a particular choice of $\nu$ and $h(Y)$.
Indeed, let $\nu'$ and $h'(Y)$ be another such pair. 
By Lemma~\ref{lem:nu-lift} there are $\nu_1,\nu_2\in R\setminus m$ such that $h(\nu_1 Y)=h'(\nu_2 Y)$.
Since
$$F_m(h(\nu_1 Y))=F_m(h'(\nu_2 Y))=X_\alpha(\nu\nu_1 v)=X_\alpha(\nu'\nu_2 v),
$$
 we have
$\nu_1^{-1}\nu_2=\nu{\nu'}^{-1}$.
 Then
$$
\begin{array}{rl}
\hat{ev}_{\mu}(h(Y))&=\hat{ev}_{\mu\cdot f_m(\nu_1^{-1})}(h(\nu_1 Y))=
\hat{ev}_{\mu\cdot f_m(\nu_1^{-1})}(h'(\nu_2 Y))=\hat{ev}_{\mu\cdot f_m(\nu_1^{-1}\nu_2)}(h'(Y))\\
&=\hat{ev}_{\mu\cdot f_m(\nu{\nu'}^{-1})}(h'(Y))=\hat{ev}_{f_m(\kappa{\nu'}^{-1})}(h'(Y))
=\hat{ev}_{f_m(\mu')}(h'(Y)).
\end{array}
$$

Note that by the last statement of Lemma~\ref{lem:nu-lift}, the elements $\hat X_\alpha(v)$, where $\alpha\in\Phi_{Q_m}$, $v\in V_\alpha$,
satisfy the analogs of relations~\eqref{eq:sum-St} and~\eqref{eq:Chev-St} that hold in $St_{Q_m}(R_m)$. This allows to
extend $\hat X_\alpha$ from the set of all $\kappa v$, $v\in V_\alpha$, $\kappa\in\hat R_m$,
to all elements of $V_\alpha\otimes_{R_m}\hat R_m$ using the relation~\eqref{eq:sum-St} and
induction on the height of $\alpha$. This definition implies that $F_m\circ p\circ\hat X_\alpha=X_\alpha$,
and the natural congruence topology on
$\hat X_\alpha(V_\alpha\otimes_{R_m}\hat R_m)$ coincides with the profinite topology
induced from $\widehat{E(R)}$, since they coincide on subsets $\hat X_\alpha(\hat R_m\cdot v)$, $v\in V_\alpha$,
by the continuity property of $\hat{ev}_\mu$ proved in Lemma~\ref{lem:easy-completion}.
As a corollary, the relations~\eqref{eq:sum-St} and~\eqref{eq:Chev-St} are satisfied by all elements of
$\hat X_\alpha(V_\alpha\otimes_{R_m}\hat R_m)$, $\alpha\in\Phi_{Q_m}$,
by continuity with respect to congruence topology. This shows that
$s_{Q_m}:St_{Q_m}(\hat R_m)\to E(\hat R_m)$ factors through
the homomorphism $p|_{\hat\Gamma_m}:\hat\Gamma_m\to E(\hat R_m)$.

Now we check the property (ii). By continuity of $\hat X_\alpha$ and $\hat X_\beta$ we can assume that
 $v\in V_\alpha$ and $u\in V_\beta$. Let $\nu_\alpha\in R\setminus m$ and $h_\alpha(Y)\in E_P(R[Y],YR[Y])$
and $\nu_\beta\in R\setminus n$, $h_\beta(Y)\in E_P(R[Y],YR[Y])$ be the elements involved in the
definition~\eqref{eq:def-hatX} for $\alpha$ and $\beta$ respectively. Then for any $\lambda\in R$ one has
$$
\hat X_\beta(\nu_\beta\lambda u)=\hat{ev}_{f_n(1)}\bigl(h_\beta(\lambda Y)\bigr)=\hat{ev}_{f_n(1)^3}\bigl(h_\beta(\lambda Y)\bigr)=
\hat{ev}_{f_n(1)}\bigl(h_\beta(\lambda Y^3)\bigr).
$$
Note that by Lemma~\ref{lem:YZ^3} one has
$$
[h_\alpha(Y),h_\beta(\lambda Z^3)]\in E_P(R[Y,Z],YZ\cdot R[Y,Z]).
$$
Then, since $f_m(\nu_\alpha^{-1})\cdot f_n(1)=0$ in $\hat R$, one has
$$
[\hat X_\alpha(v),\hat X_\beta(\nu_\beta\lambda u)]=
\hat{ev}_{f_m({\nu_\alpha}^{-1}),f_n(1)}\bigl([h_\alpha(Y),\, h_\beta(\lambda Z^3)]\bigr)=1.
$$
Since $R$ is dense in $\hat R_n$, by continuity of $\hat X_\beta$ we have
$$
[\hat X_\alpha(v),\hat X_\beta(\nu_\beta\hat R_n u)]=1.
$$
Since $\nu_\beta\in (\hat R_n)^\times$, then $[\hat X_\alpha(v),\hat X_\beta(u)]=1$.

It remains to check the property (iii). By Lemma~\ref{lem:easy-completion} (ii) the homomorphism $p$ maps $\widehat{U_{(\delta)}(R)}$
homeomorphically onto $\prod_m U_{(\delta)}(\hat R_m)\le \overline{E(R)}$. Since the analog of~\eqref{eq:XP-hat}
holds in $\overline{E(R)}$, it is enough to show that the right hand side
of~\eqref{eq:XP-hat} is contained in $\widehat{U_{(\delta)}(R)}$. For any $\alpha\in\pi_{PQ_m}^{-1}(\delta)$ one has
$V_\alpha\le V_\delta\otimes_R R_m$, hence for any $v\in V_\alpha$ there is
$\lambda\in R\setminus m$ such that $\lambda v\in V_\delta$. Then by Lemma~\ref{lem:PQ-relroots} one has
$X_\alpha(\lambda Yv)\in F_m\bigl(U_{(\delta)}(YR[Y])\bigr)$. Then $\hat X_\alpha(\kappa\lambda v)\in\widehat{U_{(\delta)}(R)}$
for any $\kappa\in \hat R_m$ by the definition of $\hat X$. Hence
$\hat X_\alpha(V_\alpha\otimes_R \hat R_m)\subseteq\widehat{U_{(\delta)}(R)}$, as required.
\end{proof}

Now we are ready to prove Theorem~\ref{thm:main}.
With Lemma~\ref{lem:subgroups-in-completion} at our disposal, the rest of the proof is similar to the concluding part
of the proof of the main theorem of A. Rapinchuk and I. Rapinchuk~\cite[pp. 3109--3111]{RR}.

\begin{proof}[Proof of Theorem~\ref{thm:main}]
By Lemma~\ref{lem:easy-completion} we have $\overline{E(R)}=\overline{G(R)}$. We proceed to establish the centrality
of the congruence kernel $\CE=\ker(p)$ in $\widehat{E(R)}$.
We use the notation of Lemma~\ref{lem:subgroups-in-completion}.

Let $\Delta$ be the subgroup of $\widehat{E(R)}$ generated by
all $\hat\Gamma_m$, $m\in\mathcal{M}$. By Lemma~\ref{lem:easy-completion} (ii) combined with Lemma~\ref{lem:subgroups-in-completion} (iii)
the group $\Delta$ contains $U_{(\delta)}(\hat R_m)\le \widehat{U_{(\delta)}(R)}\cong U_{(\delta)}(\hat R)$ for any $m\in\mathcal{M}$
and $\delta\in\Phi_P$,
and hence it contains $U_{(\delta)}\bigl(\sum_{m\in\mathcal{M}}\hat R_m\bigr)$. Since
$\sum_{m\in\mathcal{M}}\hat R_m$ is dense in $\hat R$, we conclude that the closure of $\Delta$ contains the
image of $E(R)$, and hence $\Delta$ is dense in $\widehat{E(R)}$.

Denote by $\hat\Gamma_m'$ the subgroup of $\widehat{E(R)}$ generated by all
$U_{(\delta)}\bigl(\prod_{n\neq m}\hat R_n\bigr)$, $\delta\in\Phi_P$. Then $\Delta_m=\hat\Gamma_m\cdot \hat\Gamma_m'$
contains $U_{(\delta)}(\hat R)$ for any $\delta\in\Phi_P$, and hence $\Delta_m$ is also dense in $\widehat{E(R)}$.
By Theorem~\ref{thm:K2} the intersection $\CE\cap \hat\Gamma_m=\ker(p|_{\hat\Gamma_m})$ lies in the center
of $\hat\Gamma_m$ for any $m\in\mathcal{M}$. Since $\hat\Gamma_m$ centralizes
$\hat\Gamma_m'$, we conclude that $\hat\Gamma_m$ centralizes the intersection $\CE\cap \Delta_m$.

By Lemma~\ref{lem:bounded-gen} there is $N>0$ such that for any $m\in\mathcal{M}$ any element of $E(\hat R_m)$ is
a product of $\le N$ elements of $U_P(\hat R_m)$ and $U_{P^-}(\hat R_m)$. Hence any element of
$E(\hat R)$ is a product of $\le N$ elements of $U_P(\hat R)$ and $U_{P^-}(\hat R)$.
Then $p$ maps the subset
$$
S=\bigl(\widehat{U_P(R)}\cdot \widehat{U_{P^-}(R)}\bigr)^{2N}=\Bigl(\prod_{\delta\in\Phi_P^+}\widehat{U_{(\delta)}(R)}\cdot
\prod_{\delta\in\Phi_P^-}\widehat{U_{(\delta)}(R)}\Bigr)^{2N}\subseteq\widehat{E(R)}
$$
surjectively onto $E(\hat R)$. Note that $S$ is compact and $S\subseteq \Delta_m$ for any $m\in\mathcal{M}$;
then by~\cite[Lemma 4.3]{RR} we conclude that
$\CE\cap\Delta_m$ is dense in $\CE$.
Therefore, $\hat\Gamma_m$ centralizes $\CE$ for any $m\in\mathcal{M}$. Since these subgroups generate the dense
subgroup $\Delta$ in $\widehat{E(R)}$, we conclude that $\CE$ is central in $\widehat{E(R)}$.
\end{proof}

\begin{proof}[Proof of Corollary~\ref{cor:cong-problem}]
For any $\FF_q$-algebra $R$, set $K_1^{G}(R)=G(R)/E(R)$.
By~\cite[Lemma 4.4]{St-serr} the map
$$
K_1^G(A[X_1,\ldots,X_n,Y_1^{\pm 1},\ldots,Y_m^{\pm 1}])\to
K_1^G(A\otimes_{\FF_q}\FF_q(Y_1,\ldots,Y_m)[X_1,\ldots,X_n])
$$
is injective. Since $A\otimes_{\FF_q}\FF_q(Y_1,\ldots,Y_m)$ is a regular ring containing
a perfect field, by~\cite[Theorem 1.3]{St-poly} we have
$$
K_1^G(A\otimes_{\FF_q}\FF_q(Y_1,\ldots,Y_m)[X_1,\ldots,X_n])=K_1^G(A\otimes_{\FF_q}\FF_q(Y_1,\ldots,Y_m)).
$$
Since $A\otimes_{\FF_q}\FF_q(Y_1,\ldots,Y_m)$ is semilocal and noetherian,
it is a finite product of connected semilocal rings, and hence
we have $K_1^G(A\otimes_{\FF_q}\FF_q(Y_1,\ldots,Y_m))=1$ by Lemma~\ref{lem:bounded-gen}.
Therefore, $K_1^G(A[X_1,\ldots,X_n,Y_1^{\pm 1},\ldots,Y_m^{\pm 1}])=1$.
The proof is finished by applying Theorem~\ref{thm:main} to each connected factor $R$ of the ring
$A[X_1,\ldots,X_n,Y_1^{\pm 1},\ldots,Y_m^{\pm 1}]$.
\end{proof}

\renewcommand{\refname}{References}

\end{document}